\documentclass[12pt, a4paper]{amsart}
\pdfoutput=1

\usepackage[english]{babel}
\usepackage[T1]{fontenc}
\usepackage{lmodern}
\usepackage[utf8]{inputenc}
\usepackage{microtype}
\usepackage{amssymb}
\usepackage{array}
\usepackage{lscape}
\usepackage{booktabs}
\usepackage{enumitem}
\usepackage{caption}
\usepackage{longtable}
\usepackage[longtable]{multirow}
\usepackage{tikz}
\usetikzlibrary{cd}
\usetikzlibrary{arrows,decorations.markings}
\usetikzlibrary{matrix}
\usetikzlibrary{graphs}
\usetikzlibrary{backgrounds}
\usetikzlibrary{quotes}
\usepackage{mathrsfs}
\usepackage[autostyle=true]{csquotes}
\usepackage{leftidx}
\usepackage{varwidth}
\usepackage{hyperref}
\usepackage{cleveref}

\calclayout

\newcommand*\Dynkindots{\hbox to 2em{.\hss.\hss.}}
\newcommand*\DynkinNodeSize{1.5mm}
\newcommand*\DynkinDoubleArrowLength{3.25mm}
\newcommand*\DynkinTripleArrowLength{3.5mm}
\tikzset{
  bigdnode/.style={
    circle,
    inner sep=0pt,
    minimum size=2*\DynkinNodeSize,
    fill=white,
    draw},
  dnode/.style={
    circle,
    inner sep=0pt,
    minimum size=\DynkinNodeSize,
    fill=white,
    draw},
  bnode/.style={
    circle,
    inner sep=0pt,
    minimum size=\DynkinNodeSize,
    fill=black,
    draw},
  middlearrow/.style={
    decoration={markings,
      mark=at position 0.7 with
      {\draw (0:0mm) -- +(+160:\DynkinDoubleArrowLength); \draw (0:0mm) -- +(-160:\DynkinDoubleArrowLength);},
    },
    postaction={decorate}
  },
  triplemiddlearrow/.style={
    decoration={markings,
      mark=at position 0.7 with
      {\draw (0:0mm) -- +(+160:\DynkinTripleArrowLength); \draw (0:0mm) -- +(-160:\DynkinTripleArrowLength);},
    },
    postaction={decorate}
  },	
  sedge/.style={
  },
  dedge/.style={
    middlearrow,
    double distance=1mm,
  },
  tedge/.style={
    triplemiddlearrow,
    double distance=1.0mm+\pgflinewidth,
    postaction={draw},
  },
}

\newcounter{tablepage}
\makeatletter
\long\def\widesplit#1#2#3#4{%
\def\row{\rowz\c@tablepage}%
\clearpage
\setcounter{tablepage}{0}%
\loop
\refstepcounter{tablepage}%
\noindent
\begin{minipage}{\columnwidth}
\captionsetup{type=table}
\centering
{#3}
 \ifnum\c@tablepage=\@ne
    \caption{#2}
    \label{#4}
 \else
\caption*{#2 ({\thetablepage}/{#1})}
 \fi
\end{minipage}

\ifnum#1>\c@tablepage
\repeat}
\def\rowa#1\thbreak#2\tbreak#3\\{#1&#2\\}
\def\rowz#1{%
  \ifnum#1=\@ne
    \expandafter\rowa
  \else
    \expandafter\rowq\expandafter{\the\numexpr#1-1\expandafter\relax\expandafter}%
  \fi}
\def\rowq#1#2\thbreak#3\tbreak{\rowz{#1}#2\thbreak}
\makeatother

\newcolumntype{C}{>{$}c<{$}}
\newcolumntype{L}{>{$}l<{$}}
\makeatletter
\newcommand*\xleftrightarrow[2][]{%
  \ext@arrow 9999{\longleftrightarrowfill@}{#1}{#2}}
\newcommand*\longleftrightarrowfill@{%
  \arrowfill@\leftarrow\relbar\rightarrow}
\makeatother
\definecolor{asmpgray}{gray}{0.85}
\newsavebox{\asmpbox}
\newenvironment{assumption}{
\vspace{5pt}
\begin{lrbox}{\asmpbox}
\begin{varwidth}{13cm}
\ignorespaces
}
{
\end{varwidth}
\end{lrbox}
\begin{center}
\setlength{\fboxsep}{5pt}
\fcolorbox{black}{asmpgray}{\usebox{\asmpbox}}
\end{center}
\ignorespacesafterend
\vspace{5pt}
}

\newtheorem{Thm}{Theorem}[section]
\newtheorem{Prop}[Thm]{Proposition}
\newtheorem{Lm}[Thm]{Lemma}
\newtheorem{Cor}[Thm]{Corollary}
\theoremstyle{definition}
\newtheorem{0}[Thm]{}
\newtheorem{Def}[Thm]{Definition}
\newtheorem{Rem}[Thm]{Remark}
\newtheorem{Ex}[Thm]{Example}
\newtheorem*{Not}{Notation}
\newtheorem*{Ack}{Acknowledgements}
\numberwithin{equation}{Thm}
\newcommand*{\mf}{\mathbf}
\newcommand*{\rom}[1]{\uppercase\expandafter{\romannumeral#1}}

\newcommand*{\F}{\mathbb{F}}
\newcommand*{\Fpbar}{\overline{\F}_{p}}
\newcommand*{\Fp}{\F_{p}}
\newcommand*{\Fq}{\F_{q}}
\newcommand*{\kunits}{k^\times}
\newcommand*{\N}{\mathbb{N}}
\newcommand*{\Q}{\mathbb{Q}}

\newcommand*{\Qlbar}{\overline{\Q}_{\ell}}

\newcommand*{\R}{\mathbb{R}}
\newcommand*{\Z}{\mathbb{Z}}

\newcommand*{\dt}{\mathsf}
\newcommand*{\fsep}{\quad}

\newcommand*{\I}{\mathrm{i}}
\newcommand*{\ii}{\mathfrak{i}}
\newcommand*{\jj}{\mathfrak{j}}
\newcommand*{\length}{l}

\newcommand*{\scB}{{\mathscr{B}}}

\newcommand*{\cE}{{\mathcal{E}}}

\newcommand*{\cH}{{\mathcal{H}}}

\newcommand*{\cM}{{\mathcal{M}}}

\newcommand*{\cN}{{\mathcal{N}}}
\newcommand*{\cO}{{\mathcal{O}}}

\DeclareMathOperator{\IC}{IC}
\DeclareMathOperator{\id}{id}

\DeclareMathOperator{\Irr}{Irr}

\DeclareMathOperator{\hgt}{ht}

\DeclareMathOperator{\Uch}{Uch}
\title[Values of unipotent characters at unipotent elements in type \texorpdfstring{$\dt E_8$}{E8} and \texorpdfstring{$\leftidx{^2}{\dt E}_6$}{2E6}]{The values of unipotent characters at unipotent elements for groups of type \texorpdfstring{$\dt E_8$}{E8} and \texorpdfstring{$\leftidx{^2}{\dt E}_6$}{2E6}}
\author{Jonas Hetz}
\address{Lehrstuhl für Algebra und Zahlentheorie, RWTH Aachen, Pontdriesch 14/16, D--52062 Aachen, Germany}
\email{jonas.hetz@rwth-aachen.de}
\thanks{The author gratefully acknowledges support by the Deutsche Forschungsgemeinschaft (DFG, German Research Foundation) --- Project-ID 286237555 -- TRR 195.}
\begin{document}

\begin{abstract}
In order to tackle the problem of generically determining the character tables of the finite groups of Lie type $\mf G(q)$ associated to a connected reductive group $\mf G$ over $\Fpbar$, Lusztig developed the theory of character sheaves in the 1980s. The subsequent work of Lusztig and Shoji in principle reduces this problem to specifying certain roots of unity. The situation is particularly well understood as far as character values at unipotent elements are concerned. We complete the computation of the values of unipotent characters at unipotent elements for the groups $\mf G(q)$ where $\mf G$ is the simple group of type $\dt E_8$, by specifying the aforementioned roots of unity for all prime powers $q$. We also resolve this task for the groups $\leftidx{^2}{\dt E}_6(q)$ when $q$ is a power of $p=2$. Our results thus conclude the project of computing the values of unipotent characters at unipotent elements for the simple exceptional groups of Lie type.
\end{abstract}

\maketitle

\section{Introduction}

Let $\mf G$ be a connected reductive group over an algebraic closure $k=\Fpbar$ of the finite field $\Fp$ with $p$ elements (for a prime $p$), and let $F\colon\mf G\rightarrow\mf G$ be a Frobenius map providing $\mf G$ with an $\Fq$-rational structure (where $q$ is a power of $p$). In order to tackle the problem of determining the generic character tables of the associated finite groups of Lie type $\mf G(q)=\mf G^F$ (with varying $F, q$), Lusztig developed the theory of character sheaves in the 1980s which, combined with subsequent work of Lusztig and Shoji, in principle reduces the problem to specifying certain roots of unity. In dependence of these (in general not yet known) roots of unity, Lusztig provides a theoretical algorithm to compute the values of the irreducible characters of $\mf G^F$ in \cite[\S 24]{LuCS5}. The situation is quite well understood with regard to character values at unipotent elements of $\mf G^F$ (see \cite{Luvaluni}), where the determination of the relevant roots of unity in Lusztig's algorithm would rather directly yield the values of the unipotent characters at unipotent elements. An important ingredient in this project is the computation of the ordinary Green functions of Deligne--Lusztig \cite{DL}, which has essentially been accomplished in all cases: If $p$ is a good prime for $\mf G$, this is due to Shoji \cite{ShGreenPolF4} (for groups of type $\dt F_4$), \cite{ShGreenPolClassical} (for classical groups), Spaltenstein \cite{Sp3D4} (for groups of type $\leftidx{^3}{\dt D}_4$) and Beynon--Spaltenstein \cite{BeSp} (for groups of type $\dt E_6$, $\leftidx{^2}{\dt E}_6$, $\dt E_7$ and $\dt E_8$); if $p$ is bad for $\mf G$, see again Spaltenstein \cite{Sp3D4} for groups of type $\leftidx{^3}{\dt D}_4$, Shoji \cite{ShGreen1}, \cite{ShGreen2}, \cite{ShGreen3} for classical groups (which includes the \enquote{generalised Green functions}), and Malle \cite{MUnipF4}, \cite{MGreenE6F4}, Porsch \cite{P}, Geck \cite{GcompGreen} for exceptional groups not of type $\dt E_8$.~---~For groups of type $\dt E_8$, the values of the Green functions in bad characteristic (that is, $p\leqslant5$) are known as well by now (Lübeck \cite{Lueb}). Subsequently, the program of determining the values of unipotent characters at unipotent elements has been completed for groups of type $\dt F_4$ by Marcelo--Shinoda \cite{MaShi} and for groups of type $\dt E_6$, $\dt E_7$ (relevant in bad characteristic) and $\leftidx{^2}{\dt E}_6$ in characteristic $p=3$ by the author \cite{HDiss}.

In view of these results, the remaining open problems in the above program concern the groups $\leftidx{^2}{\dt E}_6(q)$ when $q$ is a power of $p=2$ and the groups $\dt E_8(q)$ with $q$ a power of any prime $p$, which we resolve in this paper. Our main focus will lie on the groups of type $\dt E_8$; apart from a very small number of exceptions regarding these groups, our arguments provide an independent verification of the results on ordinary Green functions in \cite{BeSp} and \cite{Lueb}.

The core of our method is a reinterpretation of a formula of Ree (due to \cite[(11.28)]{CR1}, see \cite{LuFlag}, \cite{GCH}), which relates the values of the unipotent characters of $\mf G^F$ with characters of a Hecke algebra associated to $\mf G^F$. The only unknowns in said formula are the sizes of intersections of conjugacy classes of $\mf G^F$ with Bruhat cells. Based on this and motivated by \cite{LuWeylUni}, we introduce the notion of so-called \enquote{good} unipotent classes in $\dt E_8(q)$. It is remarkable that the vast majority of unipotent conjugacy classes of $\dt E_8$ turn out to contain exactly one good unipotent class of $\dt E_8(q)$; see \Cref{Thmgood}. (We have observed this already in \cite{HDiss} for certain classes in various exceptional groups; cf.\ \cite[Thm.~4.2]{LuEllip} and \cite[Thm.~5.2]{LuWeylUni} for a conceptual explanation in this direction.) This property will allow us to compute the values of any unipotent character at the elements of these classes (see \ref{Method} and \Cref{Roots1}). For a given characteristic $p$, there are at most three unipotent classes of $\dt E_8$ which do not contain a unique good class of $\dt E_8(q)$. Our method does not yield all the values of unipotent characters at elements of these classes, but the missing ones can be extracted from the results of \cite{BeSp}, \cite{Lueb}.

The paper is organised as follows. In \Cref{SecGood}, we give our definition of good unipotent elements in $\dt E_8(q)$ (\Cref{Defgood}) based on an approach of Lusztig \cite{LuWeylUni}, and we relate this notion with the one of split elements (which are only defined in good characteristic) due to Beynon--Spaltenstein \cite{BeSp}. In \Cref{ReeGreen}, we recall some of the theory concerning unipotent (almost) characters, character sheaves and the generalised Springer correspondence, so as to be able to describe Lusztig's algorithm to compute the generalised Green functions, and then we explain our method to get our hands on the unknown roots of unity in this algorithm. Sections \ref{p5}--\ref{p2} contain the explicit computations for the groups $\dt E_8(q)$, separated according to the characteristic $p$. While our arguments do of course depend on the classification of unipotent classes of groups of type $\dt E_8$ due to Mizuno \cite{MizE7E8} (with some mistakes as stated in \cite{Sp}; see Liebeck--Seitz \cite{LieSei} for an independent classification), they do not rely on the specific elements in terms of root subgroups as given in \cite{MizE7E8}, so they are robust to errors in loc.\ cit.\ in that regard. Since it may however be useful to be able to refer to the representatives in \cite{MizE7E8}, we also explain how one can single out good elements among them for almost all classes (see \Cref{Repgood}). We conclude by considering the groups $\leftidx{^2}{\dt E}_6(q)$ with $q$ a power of $p=2$ in \Cref{SectwE6}.

\begin{Not}
From now until the end of this paper (except in \Cref{SectwE6}), $\mf G$ denotes the simple algebraic group of type $\dt E_8$ over $k=\Fpbar$, so that $\mf G^F=\dt E_8(q)$ is the associated finite Chevalley group over $\Fq$. Let us fix an $F$-stable Borel subgroup $\mf B_0$ of $\mf G$ as well as an $F$-stable maximal torus $\mf T_0$ of $\mf G$ such that $\mf T_0\subseteq\mf B_0$. We denote by $\mf W=N_\mf G(\mf T_0)/\mf T_0$ the Weyl group of $\mf G$ (with respect to $\mf T_0$) and by $\Phi$ the root system associated to $\mf G$ (and $\mf T_0$). Let $\Phi^+\subseteq\Phi$ be the set of positive roots and $\Pi=\{\alpha_1,\alpha_2,\ldots,\alpha_8\}\subseteq\Phi^+$ the set of simple roots determined by $\mf B_0\supseteq\mf T_0$, labelled in such a way that the Dynkin diagram of $\mf G$ is as follows:
\begin{center}
\begin{tikzpicture}
    \draw (-1.25,-0.25) node[anchor=east]  {$\dt E_8$};

    \node[bnode,label=above:$\alpha_1$] 			(1) at (0,0) 	{};
    \node[bnode,label=right:$\alpha_2$] 			(2) at (2,-1) 	{};
    \node[bnode,label=above:$\alpha_3$] 			(3) at (1,0) 	{};
    \node[bnode,label=above:$\alpha_4$] 			(4) at (2,0) 	{};
    \node[bnode,label=above:$\alpha_5$] 			(5) at (3,0) 	{};
    \node[bnode,label=above:$\alpha_6$] 			(6) at (4,0) 	{};
	\node[bnode,label=above:$\alpha_7$] 			(7) at (5,0) 	{};
	\node[bnode,label=above:$\alpha_8$] 			(8) at (6,0) 	{};

    \path 	(1) edge[thick, sedge] (3)
          	(3) edge[thick, sedge] (4)
          	(4) edge[thick, sedge] (5)
			(4)	edge[thick, sedge] (2)
          	(5) edge[thick, sedge] (6)
			(6) edge[thick, sedge] (7)
			(7) edge[thick, sedge] (8);
\end{tikzpicture}
\end{center}
Let $S\subseteq\mf W$ be the reflections corresponding to the roots in $\Pi$; $S$ is a set of Coxeter generators for the Coxeter group $\mf W$. Let $\mf U_0:=R_\mathrm u(\mf B_0)$ be the unipotent radical of $\mf B_0$ and $\mf G_{\mathrm{uni}}\subseteq\mf G$ be the set of all unipotent elements in $\mf G$. For $\alpha\in\Phi^+$, there is an isomorphism $u_\alpha\colon k\xrightarrow{\sim}\mf U_\alpha$, $\xi\mapsto u_\alpha(\xi)$, where $\mf U_\alpha\subseteq\mf U_0$ is the root subgroup corresponding to $\alpha$. Regarding the unipotent classes in $\mf G$, we use the names of Mizuno \cite{MizE7E8} and Spaltenstein \cite{Sp}. As soon as $p$ is prescribed in a given setting, we tacitly assume to have fixed a prime $\ell\neq p$ and denote by $\Qlbar$ an algebraic closure of the field $\Q_\ell$ of $\ell$-adic numbers. For any finite group $\Gamma$, we then write $\Irr(\Gamma)$ for the set of all ordinary irreducible characters of $\Gamma$, viewed as functions $\Gamma\rightarrow\Qlbar$.
\end{Not}

\section{Good unipotent elements}\label{SecGood}

In this section, we give the definition of \enquote{good} unipotent elements (\Cref{Defgood}) and formulate the result that almost all unipotent conjugacy classes of $\mf G$ contain a unique $\mf G^F$-class consisting of good elements (\Cref{Thmgood}), which we will prove in Sections \ref{p5}--\ref{p2}. We also recall the notion of split unipotent elements (which only exist in good characteristic) of Beynon--Spaltenstein \cite{BeSp} and describe the relation between good and split elements in \Cref{Corsplitgood}, whose proof will be given in \Cref{ProofCorsplitgood}.

\begin{0}{\bf Unipotent classes and Lusztig's map.}\label{LuMap}
An important role in our argumentation will be played by making a ``good'' choice for a $\mf G^F$-conjugacy class contained in a given $F$-stable unipotent class of $\mf G$. So let $\cO\subseteq\mf G$ be a unipotent conjugacy class. The results of \cite{MizE7E8} show that we always have $F(\cO)=\cO$ and that there exists a representative $u_\cO\in\cO^F$ such that $F$ acts trivially on the group $A_\mf G(u_\cO)=C_\mf G(u_\cO)/C_\mf G^\circ(u_\cO)$, so it makes sense to introduce the following convention:
\begin{assumption}
Given an ($F$-stable) unipotent class $\cO\subseteq\mf G$, we will denote by $u_\cO\in\cO^F$ a (not necessarily fixed) element such that $F$ induces the identity on the group $A_\mf G(u_\cO)=C_\mf G(u_\cO)/C_\mf G^\circ(u_\cO)$. 
\end{assumption}
The $\mf G^F$-conjugacy classes contained in $\cO^F$ are then parametrised by the conjugacy classes of $A_\mf G(u_\cO)$: For $a\in A_\mf G(u_\cO)$, the conjugacy class of $a$ in $A_\mf G(u_\cO)$ corresponds to the $\mf G^F$-conjugacy class of $(u_\cO)_a:=gu_\cO g^{-1}\in\mf G^F$ where $g\in\mf G$ is such that $g^{-1}F(g)\in C_\mf G(u_\cO)$ is mapped to $a$ under the canonical map $C_\mf G(u_\cO)\rightarrow A_\mf G(u_\cO)$. (This follows from a standard application of the Lang--Steinberg Theorem.)

In general, the requirement on $u_\cO$ above does not uniquely specify the $\mf G^F$-conjugacy class of $u_\cO$, so we need an additional condition. Given $w\in\mf W$ with representative $\dot w\in N_\mf G(\mf T_0)^F$, we consider the Bruhat cell $\mf B_0\dot w\mf B_0$ and its set of $F$-stable points $(\mf B_0\dot w\mf B_0)^F=\mf B_0^F\dot w\mf B_0^F$. (This equality follows from the \enquote{sharp form} of the Bruhat decomposition; see, e.g., \cite[2.5.14]{C}.) Since both $\mf B_0\dot w\mf B_0$ and $\mf B_0^F\dot w\mf B_0^F$ are clearly independent of the chosen representative $\dot w$ of $w$, we will write $\mf B_0w\mf B_0:=\mf B_0\dot w\mf B_0$ and $\mf B_0^Fw\mf B_0^F:=\mf B_0^F\dot w\mf B_0^F$ from now on. In \cite[4.5]{LuWeylUni}, Lusztig defines a surjective map
\[\{\text{Conjugacy classes of }\mf W\}\rightarrow\{\text{Unipotent conjugacy classes of }\mf G\},\quad \Xi\mapsto\cO_\Xi,\]
which is characterised by the following property (see \cite[0.4]{LuWeylUni} and \cite[4.8]{LuEllip}): For any conjugacy class $\Xi$ of $\mf W$ and any $w\in\Xi$ which is of minimal length among the elements of $\Xi$ (with respect to the standard length function on $(\mf W,S)$), the following two conditions hold.
\begin{enumerate}[label=(\roman*)]
\item We have $\cO_\Xi\cap\mf B_0w\mf B_0\neq\varnothing$;
\item for any unipotent class $\cO\subseteq\mf G$ satisfying $\cO\cap\mf B_0w\mf B_0\neq\varnothing$, we have $\cO_\Xi\subseteq\overline{\cO}$.
\end{enumerate}
(This formulation is indeed equivalent to the one in \cite[4.8]{LuEllip}, see \cite[Remark~3.8]{GCH}; thus, by the argument in \cite[0.2]{LuWeylUni}, the conditions (i) and (ii) above are independent of the choice of $w\in\Xi$, as long as $w$ is of minimal length.) For example, the trivial class of $\mf W$ is sent to the trivial class of $\mf G$; the class of $\mf W$ containing the Coxeter elements is sent to the regular unipotent class of $\mf G$. Note that, since the number of conjugacy classes of $\mf W$ is larger than the number of unipotent classes of $\mf G$, Lusztig's map cannot be injective. The assignment $\Xi\mapsto\cO_\Xi$ can be explicitly obtained using Michel's {\sffamily CHEVIE}, see \cite[\S 6]{MiChv}. Given a unipotent class $\cO\subseteq\mf G$, let $w\in\mf W$ be an element of minimal length with the property that the conjugacy class of $w$ is sent to $\cO$ under Lusztig's map. (In particular, $w$ is of minimal length in its own conjugacy class.) This additional requirement determines the conjugacy class of $w$ uniquely (as we can see from the output of the {\sffamily CHEVIE} function mentioned above), and we write $w\looparrowright\cO$ for any such $w$.
\end{0}

\begin{Def}[cf.\ {\cite[3.2.22--3.2.23]{HDiss}}]\label{Defgood}
In the setting and with the notation of \ref{LuMap}, assume that $w\looparrowright\cO$. We say that $u_0\in\cO^F$ is \emph{good} if the following two conditions are satisfied.
\begin{enumerate}[label=(\alph*)]
\item $F$ acts trivially on $A_\mf G(u_0)$;
\item $u_0$ is $\mf G^F$-conjugate to an element of $\mf B_0^Fw\mf B_0^F$.
\end{enumerate}
\end{Def}

This definition is clearly independent of the choice of $u_0$ in its $\mf G^F$-conjugacy class, so we will sometimes also speak of a good $\mf G^F$-class. Our arguments in Sections \ref{p5}--\ref{p2} below will prove the following remarkable result (cf.\ \cite[Thm.~4.2]{LuEllip}, \cite[Thm.~5.2]{LuWeylUni}).
\begin{Thm}\label{Thmgood}
Let $\cO$ be a unipotent conjugacy class of $\mf G$ which does not appear in the following list:
\begin{enumerate}[label=(\alph*)]
\item $\dt D_8(a_3)$ (for $q\equiv-1\mod3$), $\dt E_7(a_1)+\dt A_1$, $\dt D_6(a_1)+\dt A_1$ when $p\geqslant5$;
\item $\dt E_7(a_1)+\dt A_1$, $\dt D_6(a_1)+\dt A_1$ when $p=3$;
\item $\dt D_8(a_3)$ (for $q\equiv-1\mod3$) when $p=2$.
\end{enumerate}
Then $\cO^F$ contains exactly one good $\mf G^F$-conjugacy class.
\end{Thm}

\begin{Rem}\label{Remsplit}
In good characteristic (that is, $p\geqslant7$), there is the notion of \enquote{split} unipotent elements, due to Beynon--Spaltenstein \cite[p.~590]{BeSp}: An element $u\in\mf G^F_{\mathrm{uni}}$ is called split if $F$ acts trivially on the set of all irreducible components of the variety $\scB_u$ consisting of the Borel subgroups of $\mf G$ which contain $u$. If $u\in\mf G^F_{\mathrm{uni}}$ is split, then any $\mf G^F$-conjugate of $u$ is split as well, and $F$ acts trivially on $A_\mf G(u)$. Moreover, with the exception of $\dt D_8(a_3)$ when $q\equiv-1\mod3$, any unipotent conjugacy class of $\mf G$ contains a unique $\mf G^F$-class consisting of split elements (while there exist no split elements in $\dt D_8(a_3)$ when $q\equiv-1\mod3$).
\end{Rem}

The following result is a consequence of \Cref{Thmgood} and the computations that we will perform in \Cref{p5} (see \ref{ProofCorsplitgood} for the proof).

\begin{Cor}\label{Corsplitgood}
Let $p\geqslant7$, and let $u_0\in\mf G^F_{\mathrm{uni}}$ a unipotent element which does not lie in $\dt E_7(a_1)+\dt A_1$ or $\dt D_6(a_1)+\dt A_1$. Then $u_0$ is split if and only if it is good. On the other hand, both $(\dt E_7(a_1)+\dt A_1)^F$ and $(\dt D_6(a_1)+\dt A_1)^F$ are the union of two $\mf G^F$-conjugacy classes, each of which is good (but only one of which is split).
\end{Cor}

\begin{Rem}
Hence, with respect to the vast majority of unipotent conjugacy classes of $\mf G$ in good characteristic, the notions of \enquote{split} and \enquote{good} elements turn out to coincide. At least for the unipotent classes which are not excluded in \Cref{Thmgood}, our definition of good elements may thus be seen as an extension of the notion of split elements to the bad characteristic case.
\end{Rem}

\section{Ree's formula and (generalised) Green functions}\label{ReeGreen}

In this section, we recall the theory on unipotent (almost) characters, character sheaves and the generalised Springer correspondence which is needed to describe Lusztig's algorithm to compute the generalised Green functions. We also describe a reinterpretation of a formula of Ree (due to \cite{LuFlag}, \cite{GCH}) relating the values of the unipotent almost characters of $\mf G^F$ with characters of a Hecke algebra associated to $\mf G^F$ (see \ref{Ree}), which will be the core of our method to determine the values of the unipotent characters at unipotent elements (see \ref{Method}). We conclude by stating a result (\Cref{Roots1}) yielding the values of the unipotent characters at unipotent elements of $\mf G^F$; this will be proven in Sections \ref{p5}--\ref{p2} below.

\begin{0}{\bf Unipotent (almost) characters.}\label{E8Setup}
The results of \cite[Chap.~4]{Luchars} provide a parameter set $X(\mf W)$ (only depending on $\mf W$, and not on $p$ or $q$) for both the unipotent characters and the unipotent ``almost characters'' of $\mf G^F$. Any unipotent character of $\mf G^F$ is expressed as an explicit linear combination of unipotent almost characters $R_x$, $x\in X(\mf W)$, by means of Lusztig's ``(non-abelian) Fourier transform matrix''. In particular, the computation of unipotent characters at unipotent elements is equivalent to the computation of the $R_x|_{\mf G^F_{\mathrm{uni}}}$ for $x\in X(\mf W)$. Furthermore, there is a certain embedding $\Irr(\mf W)\hookrightarrow X(\mf W)$, $\phi\mapsto x_\phi$ \cite[(4.21.3)]{Luchars}, and the principal series unipotent character $[\phi]$ of $\mf G^F$ parametrised by $\phi\in\Irr(\mf W)$ is expressed as
\[[\phi]=\sum_{x\in X(\mf W)}\{x_\phi,x\}R_x\]
where the $\{x_\phi,x\}$ are entries of the aforementioned Fourier matrix. By \cite[4.13]{Luchars}, we have $|X(\mf W)|=166$ and $|\Irr(\mf W)|=112$. The computation of the almost characters $R_\phi=R_{x_\phi}$ (for $\phi\in\Irr(\mf W)$) at unipotent elements is equivalent to the one of the Green functions. In good characteristic, this has been established by Beynon--Spaltenstein \cite{BeSp}; when $p\leqslant5$, the values of the $R_\phi$ are known as well by now \cite{Lueb} (see also \cite{GcompGreen} for one particular case when $p=2$). With very few exceptions, our arguments will provide an independent verification of these results. Other than that, it remains to consider the $54$ elements $x\in X(\mf W)\setminus\Irr(\mf W)$.
\end{0}

\begin{0}{\bf A reinterpretation of Ree's formula.}\label{Ree}
Associated with the Weyl group $\mf W$ is a generic Iwahori--Hecke algebra $\cH$ which we may (and will) assume to be taken over $\Qlbar[\mf q]$ (where $\mf q$ is an indeterminate). $\cH$ has a standard basis $T_w$, $w\in\mf W$ (see, for example, \cite[\S 11D]{CR1} or \cite[\S 8.4]{GePf}). By means of specialising $\cH$ at $q$ (according to the $\Fq$-rational structure on $\mf G$ defined by $F$), we obtain an Iwahori--Hecke algebra $\cH_q$, whose multiplication is determined by the following equations, where $w\in\mf W$, $s\in S$, and $\length\colon\mf W\rightarrow\N_0$ is the length function of $\mf W$ with respect to $S$:
\[T_w\cdot T_{s}=\begin{cases}\hfil T_{ws}&\text{if}\quad\length(ws)=\length(w)+1, \\ qT_{ws}+(q-1)T_{w}&\text{if}\quad\length(ws)=\length(w)-1.
\end{cases}\]
Each $\phi\in\Irr(\mf W)$ parametrises an irreducible character $\phi_q$ of $\cH_q$. For $u\in\mf G^F_{\mathrm{uni}}$, let us denote by $O_u\subseteq\mf G^F_{\mathrm{uni}}$ the $\mf G^F$-conjugacy class of $u$. By \cite[Rem.~3.6(b)]{GCH}, we have the identity
\begin{equation*}
\frac{|\mf B_0^Fw\mf B_0^F\cap O_u|\cdot|C_{\mf G^F}(u)|}{|\mf B_0^F|}=\sum_{\phi\in \Irr(\mf W)} 
\phi_q(T_w)[\phi](u)\qquad \text{for any } w\in\mf W
\end{equation*}
(which already appeared in \cite[1.5(a)]{LuFlag}). Setting
\[m(u,w):=\sum_{x\in X(\mf W)}c_x(w)R_x(u),\quad\text{with}\quad c_x(w):=\sum_{\phi\in\Irr(\mf W)}\{x_\phi,x\}\phi_q(T_w),\]
we thus get
\begin{equation}\label{HeckeFormula}
m(u,w)=\sum_{x\in X(\mf W)}c_x(w)R_x(u)=\frac{|\mf B_0^Fw\mf B_0^F\cap O_u|\cdot|C_{\mf G^F}(u)|}{|\mf B_0^F|}\qquad \text{for }\;w\in\mf W.
\tag{$\spadesuit$}
\end{equation}
Since the values $\phi_q(T_w)$ are known \cite{GeMi} (and explicitly accessible through {\sffamily CHEVIE} \cite{MiChv}), the only unknowns in this formula are the (sizes of the) intersections $\mf B_0^Fw\mf B_0^F\cap O_u$ and the values $R_x(u)$ for $x\in X(\mf W)$. To get information on $\mf B_0^Fw\mf B_0^F\cap O_u$, we will (as suggested by Definition~\ref{Defgood}) try to evaluate the above equation for $w\in\mf W$ such that $w\looparrowright\cO$, with $\cO$ the conjugacy class of $\mf G$ containing $u$. As for the values of $R_x(u)$, we now explain how the computation of the $R_x|_{\mf G^F_{\mathrm{uni}}}$ for $x\in X(\mf W)$ can be carried out \emph{up to certain roots of unity}, using Lusztig's algorithm provided in \cite[Chap.~24]{LuCS5} combined with the explicit knowledge of the generalised Springer correspondence \cite{LuIC}, \cite{Sp}, \cite{HSpring}.
\end{0}

\begin{Rem}\label{Criteriongood}
Let $\cO\subseteq\mf G$ be a unipotent class, and assume that there exists a unique $\mf G^F$-conjugacy class in $\cO^F$ which intersects the Bruhat cell $\mf B_0^Fw\mf B_0^F$ non-trivially for $w\looparrowright\cO$. (As already stated in \Cref{Thmgood}, this will turn out to be true in the vast majority of cases.) Then, denoting by $u_0\in\cO^F$ an element of this class (that is, $u_0$ satisfies condition (b) in \Cref{Defgood}), we have $\mf B_0^Fw\mf B_0^F\cap O_{u_0}=(\mf B_0w\mf B_0\cap\cO)^F$ and the number $|(\mf B_0w\mf B_0\cap\cO)^F|$ can be computed explicitly, see \cite[\S 5]{GCH}. Hence, in view of formula \eqref{HeckeFormula} in \ref{Ree}, the number $m(u_0,w)$ (for $w\looparrowright\cO$) and the centraliser order $|C_{\mf G^F}(u_0)|$ determine one another. So if $\cO$ is a class as above and if we manage to compute $m(u_0,w)$ for a good $u_0\in\cO^F$, we can read off the centraliser order $|C_{\mf G^F}(u_0)|$ from formula \eqref{HeckeFormula} in \ref{Ree}. This will in some cases be a useful tool to deduce that a certain representative $u_0$ which satisfies condition (b) in \Cref{Defgood} also satisfies condition (a) (so that $u_0$ is good).
\end{Rem}

\begin{0}{\bf The generalised Springer correspondence.}\label{GenSpring}
Following \cite{LuIC}, let us define $\cN_\mf G$ to be the set of all pairs $(\cO,\cE)$ where $\cO\subseteq\mf G$ is a unipotent class and $\cE$ is a $\mf G$-equivariant irreducible $\Qlbar$-local system on $\cO$ (taken up to isomorphism), and let $\cM_\mf G$ be the set of all triples $(\mf L,\cO_0,\cE_0)$ up to $\mf G$-conjugacy, where $\mf L\subseteq\mf G$ is the Levi complement of some parabolic subgroup of $\mf G$ and $\cE_0$ is an $\mf L$-equivariant irreducible $\Qlbar$-local system on $\cO_0$ (up to isomorphism) such that $(\cO_0,\cE_0)$ is a ``cuspidal pair'' for $\mf L$ in the sense of \cite[2.4]{LuIC}. Any $\jj=(\mf L,\cO_0,\cE_0)\in\cM_\mf G$ gives rise to a perverse sheaf $K_\jj$ on $\mf G$ whose endomorphism algebra is isomorphic to the group algebra $\Qlbar[W_\mf G(\mf L)]$ of the relative Weyl group $W_\mf G(\mf L)=N_\mf G(\mf L)/\mf L$; see \cite[Theorem~9.2]{LuIC}. Thus, any simple constituent $A$ of $K_\jj$ is naturally parametrised by an irreducible character of $W_\mf G(\mf L)$. On the other hand, for any such $A$, it is shown in \cite[\S 6]{LuIC} that there exists a unique pair $(\cO,\cE)\in\cN_\mf G$ such that
\[A|_{\mf G_{\mathrm{uni}}}\cong\IC(\overline{\cO},\cE)[\dim\mf Z(\mf L)+\dim\cO].\]
Conversely, for any $\ii=(\cO,\cE)\in\cN_\mf G$, there exists a unique $\jj=(\mf L,\cO_0,\cE_0)\in\cM_\mf G$ and a unique simple constituent $A$ of $K_\jj$ whose restriction to $\mf G_{\mathrm{uni}}$ is isomorphic to $\IC(\overline{\cO},\cE)[\dim\mf Z(\mf L)+\dim\cO]$. So we obtain a surjective map $\tau\colon\cN_\mf G\rightarrow\cM_\mf G$, and the fibre $\tau^{-1}(\jj)$ of a given $\jj=(\mf L,\cO_0,\cE_0)\in\cM_\mf G$ is in bijection with $\Irr(W_\mf G(\mf L))$. If $A$ arises from a pair $\ii\in\cN_\mf G$ as described above, we will write $A_\ii:=A$. The collection of the bijections
\begin{equation}\label{GenSpringBij}
\Irr(W_\mf G(\mf L))\xrightarrow{\sim}\tau^{-1}(\jj)\quad(\text{for }\jj=(\mf L,\cO_0,\cE_0)\in\cM_\mf G)\tag{$\heartsuit$}
\end{equation}
thus obtained is referred to as the generalised Springer correspondence.
\end{0}

\begin{0}\label{NG}
The set $\cN_\mf G$ is described in \cite[15.3]{LuIC}; it depends on $p$. We have
\[|\cN_\mf G|=\begin{cases} 113&\text{if }\;p\geqslant7,\\
117&\text{if }\; p=5, \\
127&\text{if }\; p=3, \\
146&\text{if }\;p=2.
\end{cases}\]
Any complex $A_\ii$ associated to $\ii\in\cN_\mf G$ in \ref{GenSpring} is in fact a ``unipotent character sheaf''. By the results of \cite[Chap.~23]{LuCS5} (combined with \cite{LuclCS}), $A_\ii$ is thus parametrised by some $x\in X(\mf W)$ and may also be denoted by $A_x$. We will then say that ``$\ii$ corresponds to $x$ under the generalised Springer correspondence'' and write $\ii\leftrightarrow x$. On the other hand, for any $x\in X(\mf W)$, the character sheaf $A_x$ gives rise to a $\mf G^F$-invariant ``characteristic function'' $\chi_x\colon\mf G^F\rightarrow\Qlbar$ (determined by $A_x$ only up to multiplication with a non-zero scalar), which by \cite[Proposition~4.4]{Sh2} combined with the cleanness results in \cite{LuclCS} coincides with the unipotent almost character $R_x$ (up to a scalar multiple). In particular, we have $R_x|_{\mf G^F_{\mathrm{uni}}}\neq0$ if and only if $\chi_x|_{\mf G^F_{\mathrm{uni}}}\neq0$, and the latter happens precisely when $A_x$ is isomorphic to $A_\ii$ for some $\ii\in\cN_\mf G$ in the setting of \ref{GenSpring}. Hence, to compute the $R_x$ at unipotent elements, it suffices to consider those $x\in X(\mf W)$ which correspond to elements of $\cN_\mf G$ under the generalised Springer correspondence. Moreover, for any of the $112$ irreducible characters $\phi\in\Irr(\mf W)$, there exists some $\ii\in\cN_\mf G$ such that $\ii\leftrightarrow x_\phi$; this is nothing but the ordinary Springer correspondence and as mentioned in \ref{E8Setup}, we have $R_{x_\phi}=R_\phi$ for $\phi\in\Irr(\mf W)$. Identifying $\cN_\mf G$ with a subset of $X(\mf W)$ as described above, we thus have
\[|\cN_\mf G\cap (X(\mf W)\setminus\Irr(\mf W))|=\begin{cases} \hfill1&\text{if }\;p\geqslant7,\\
\hfill5&\text{if }\; p=5, \\
\hfill15&\text{if }\; p=3, \\
34&\text{if }\;p=2.
\end{cases}\]
In the case where $p\geqslant5$, all elements of $\cN_\mf G\cap (X(\mf W)\setminus\Irr(\mf W))$ are ``cuspidal''; if $p=3$, the set $\cN_\mf G\cap (X(\mf W)\setminus\Irr(\mf W))$ consists of $3$ cuspidal and $12$ non-cuspidal elements; if $p=2$, the set $\cN_\mf G\cap (X(\mf W)\setminus\Irr(\mf W))$ consists of $5$ cuspidal and $29$ non-cuspidal elements (see again \cite[15.3]{LuIC}).
\end{0}

The following notation will be convenient: As the isomorphism classes of $\mf G$-equivariant irreducible $\Qlbar$-local systems on a given unipotent class $\cO\subseteq\mf G$ are naturally parametrised by the irreducible characters of $A_\mf G(u_\cO)$, we may (and will in many cases) label the elements of $\cN_\mf G$ by $(u_\cO,\varsigma)$ with $\varsigma\in\Irr(A_\mf G(u_\cO))$.

\begin{0} {\bf Characteristic functions of cuspidal character sheaves.} \label{CuspCS}
Let $\ii=(u_\cO,\varsigma)\in\cN_\mf G$, $\ii\leftrightarrow x\in X(\mf W)$, and assume that $A_x$ is a cuspidal (unipotent) character sheaf on $\mf G$. As mentioned in \ref{NG}, the characteristic function $\chi_x$ associated to $A_x$ is in general only defined up to a scalar multiple. However, a fixed choice of $u_\cO\in\cO^F$ uniquely determines a characteristic function $\chi_x^0\colon\mf G^F\rightarrow\Qlbar$ (see, e.g., \cite[\S 3]{Luvaluni}). It is given by
\[\chi_x^0(g)=\chi_{u_\cO,\varsigma}(g)=\begin{cases}q^{\frac12\dim{C_\mf G(u_\cO)}}\varsigma(a)&\text{if \;$g$ is $\mf G^F$-conjugate to $(u_\cO)_a$ with $a\in A_\mf G(u_\cO)$}, \\
\hfil0&\text{if }\;g\notin\cO^F.
\end{cases}\]
The function $\chi_x^0$ satisfies the requirement in \cite[25.1--25.2]{LuCS5}, so Shoji's Theorem \cite[4.1]{Sh2} shows that we have $R_x=\xi_x\chi_x^0$ for some \emph{root of unity} $\xi_x\in\Qlbar$ depending (only) on the choice of $u_\cO\in\cO^F$; we will thus sometimes also write $\xi_{(u_\cO,\varsigma)}:=\xi_x$.
\end{0}

\begin{0} {\bf \enquote{Intermediate} characteristic functions.} \label{Intermed}
As noted in \ref{NG}, in the case where $p\leqslant3$, there exist non-cuspidal elements $\ii\leftrightarrow x$ in $\cN_\mf G\cap(X(\mf W)\setminus\Irr(\mf W))$ for which $R_x|_{\mf G^F_{\mathrm{uni}}}\neq0$. In order to get our hands on these functions, we need to exploit the full power of Lusztig's algorithm \cite[\S 24]{LuCS5}, combined with the explicit knowledge of the bijection \eqref{GenSpringBij} in \ref{GenSpring}.

So let $\ii=(\cO,\cE)\leftrightarrow x$ be one of those \enquote{intermediate} elements. With the notation of \ref{GenSpring}, we then have $\tau(\ii)=\jj=(\mf L,\cO_0,\cE_0)\in\cM_\mf G$ where $\mf L\subseteq\mf G$ is of type $\dt E_6$ (when $p=3$) or of type $\dt E_7$ or $\dt D_4$ (when $p=2$), and $(\cO_0,\cE_0)$ is a cuspidal pair for $\mf L$. In all of these cases, $\cO_0$ is the regular unipotent class of $\mf L$. We will then often just write $\jj=(\dt E_7,\cO_0,\cE_0)$, $\jj=(\dt E_6,\cO_0,\cE_0)$, $\jj=(\dt D_4,\cO_0,\cE_0)$ if $\mf L$ is of type $\dt E_7$, $\dt E_6$, $\dt D_4$, respectively. While there is a unique cuspidal pair for $\mf L$ of type $\dt D_4$ (and $p=2$), there are two cuspidal pairs for $\mf L$ either of type $\dt E_6$ (and $p=3$) or of type $\dt E_7$ (and $p=2$). For $n\in\{6,7\}$ and a fixed representative $u_{\dt E_n}\in\cO_0\subseteq\mf L$, the isomorphism classes of $\mf L$-equivariant irreducible $\Qlbar$-local systems on $\cO_0$ are naturally parametrised by the irreducible characters of $A_\mf L(u_{\dt E_n})$, and it will be convenient to write $(\dt E_n,u_{\dt E_n},\varsigma):=(\mf L,\cO_0,\cE_0)$ if $\varsigma\in\Irr(A_\mf L(u_{\dt E_n}))$ corresponds to $\cE_0$.

Assuming that an isomorphism $F^\ast\cE_0\xrightarrow{\sim}\cE_0$ is chosen (which may be achieved by means of singling out a $\mf G^F$-conjugacy class contained in $\cO_0^F$), Lusztig \cite[24.2]{LuCS5} specifies a $\mf G^F$-invariant function $X_\ii\colon\mf G^F_{\mathrm{uni}}\rightarrow\Qlbar$. As in \cite[3.2.25]{HDiss} (see \cite[\S 3]{Luvaluni}), this allows us to define a $\mf G^F$-invariant function $\chi_x^0\colon\mf G^F\rightarrow\Qlbar$ by
\[\chi_x^0(g):=\begin{cases}q^{(\dim\mf G-\dim\cO-\dim\mf Z(\mf L))/2}X_\ii(g) &\text{if }g\in\mf G^F_{\mathrm{uni}},\\
\hfil0&\text{if }g\in\mf G^F\setminus\mf G^F_{\mathrm{uni}}.
\end{cases}\]
Then $\chi_x^0$ is a characteristic function of the character sheaf $A_x$ which meets the requirement in \cite[25.1--25.2]{LuCS5}. By \cite[Proposition~4.4]{Sh2}, we have $R_x=\xi_x\chi_x^0$ for some root of unity $\xi_x\in\Qlbar$.~---~While $\xi_x$ depends upon the choice of a $\mf G^F$-conjugacy class contained in $\cO_0^F$ as mentioned above, it does not depend on $x$ as long as $x$ gives rise to a fixed triple $\jj=(\mf L,\cO_0,\cE_0)$; this follows from \cite[3.5]{Luvaluni} (see \cite[3.4.23]{HDiss}). So we can write $\zeta_{\jj}:=\xi_x$ for any $x\leftrightarrow\ii$ such that $\tau(\ii)=\jj$. The computation of the function $X_\ii$ is described in \cite[Chap.~24]{LuCS5}. To explain it, we use the notation in \cite[1.3]{ShGreen1}: For $\ii'=(u_{\cO'},\varsigma')\in\cN_\mf G$, consider the function $Y_{\ii'}^0\colon\mf G^F_{\mathrm{uni}}\rightarrow\Qlbar$ defined by
\[Y_{\ii'}^0(u):=\begin{cases}
\hfil \varsigma'(a) &\text{if $u$ is $\mf G^F$-conjugate to $(u_{\cO'})_a$ with $a\in A_\mf G(u_{\cO'})$},  \\
\hfil0 &\text{if }u\notin\cO'^F.
\end{cases}\]
Then there exist roots of unity $\gamma_{\ii'}\in\Qlbar$ ($\ii'\in\cN_\mf G$) such that
\[X_{\ii}=\sum_{\ii'\in\cN_\mf G}p_{\ii',\ii}\gamma_{\ii'}Y_{\ii'}^0,\]
where the coefficients $p_{\ii',\ii}$ are rational integers which are given by means of a purely combinatorial algorithm \cite[Chap.~24]{LuCS5}. This algorithm may be accessed through Michel's {\sf CHEVIE} \cite{MiChv}. In fact, we always have $p_{\ii,\ii}=1$; moreover, if $\ii'=(u_{\cO'},\varsigma')$ is such that $p_{\ii',\ii}\neq0$, then $\tau(\ii')=\tau(\ii)=(\mf L,\cO_0,\cE_0)$ and $\cO'\subseteq\overline\cO$. To summarise, we have
\begin{align*}
R_x(u)&=\zeta_{\jj}\chi_x^0(u)=\zeta_{\jj}q^{(\dim\mf G-\dim\cO-\dim\mf Z(\mf L))/2}X_\ii(u) \\
&=\zeta_{\jj}q^{(\dim\mf G-\dim\cO-\dim\mf Z(\mf L))/2}\sum_{\ii'\in\tau^{-1}(\jj)}p_{\ii',\ii}\gamma_{\ii'}Y_{\ii'}^0(u)
\end{align*}
for any $u\in\mf G^F_{\mathrm{uni}}$, so the determination of $\zeta_{\jj}\gamma_{\ii'}$ for $\ii'\in\tau^{-1}(\jj)$ would yield the values of $R_x|_{\mf G^F_{\mathrm{uni}}}$ for any $x\leftrightarrow\ii\in\tau^{-1}(\jj)\subseteq\cN_\mf G\cap(X(\mf W)\setminus\Irr(\mf W))$.
\end{0}

\begin{0}{\bf The Green functions.} \label{Green}
For elements $\phi\in\Irr(\mf W)\subseteq X(\mf W)$, we have much more information on the values of the corresponding almost characters $R_{x_\phi}=R_\phi$. Given $\phi,\phi'\in\Irr(\mf W)$ and the associated elements $\ii=(u_\cO,\varsigma)$, $\ii'=(u_{\cO'},\varsigma')\in\cN_\mf G$ under the (ordinary) Springer correspondence, let us write
\[p_{\phi',\phi}:=p_{\ii',\ii}\text{ (cf.\ \ref{Intermed})}\quad\text{and}\quad d_\phi:=\tfrac12(\dim\mf G-\dim\cO-\dim\mf T_0),\]
and let
\[Y_{\phi'}^0\colon\mf G^F_{\mathrm{uni}}\rightarrow\Qlbar,\;u\mapsto\begin{cases}
\hfil \varsigma'(a) &\text{if $u$ is $\mf G^F$-conjugate to $(u_{\cO'})_a$ for some $a\in A_\mf G(u_{\cO'})$},  \\
\hfil0 &\text{if }u\notin\cO'^F.
\end{cases}\]
For $\phi\in\Irr(\mf W)$, we then have
\[R_\phi|_{\mf G^F_{\mathrm{uni}}}=\sum_{\phi'\in\Irr(\mf W)}q^{d_\phi}p_{\phi',\phi}\delta_{\phi'}Y_{\phi'}^0\]
where $\delta_{\phi'}\in\{\pm1\}$; see, e.g., \cite[\S 2, \S 3]{Gcompvaluni}. (In fact, in this special situation, an algorithm to determine the $p_{\phi',\phi}$ has been developed earlier by Shoji, see \cite{ShArc}.) So in order to compute the $R_\phi|_{\mf G^F_{\mathrm{uni}}}$, it only remains to determine the signs $\delta_{\phi'}$ for $\phi'\in\Irr(\mf W)$. Note that $\delta_{\phi'}$ depends upon the choice of $u_{\cO'}$; we will thus often also write $\delta_{(u_{\cO'},\varsigma')}:=\delta_{\phi'}$.

Given any unipotent class $\cO\subseteq\mf G$, the pair $(u_\cO,1)\in\cN_\mf G$ always corresponds to an irreducible character of $\mf W$ under the bijection \eqref{GenSpringBij} in \ref{GenSpring}. In fact, it directly follows from the Lusztig--Shoji algorithm to compute the $p_{\phi',\phi}$ that we have $\delta_{(u_\cO,1)}=+1$; see the argument in \cite[2.7, 3.4]{GcompGreen}. 
\end{0}

\begin{0}{\bf Our method.}\label{Method}
The principal idea to compute the values of the $R_x$ at $F$-stable elements $u$ of a given unipotent class $\cO\subseteq\mf G$ is to first evaluate $R_x(u)$ up to some unknown roots of unity as described in \ref{CuspCS}--\ref{Green} and then apply the formula \eqref{HeckeFormula} in \ref{Ree} to obtain information on these roots of unity, using the fact that we have $m(u,w)\geqslant0$ for any $u\in\cO^F$, $w\in\mf W$, and $m(u_0,w)>0$ if $w\looparrowright\cO$ and $u_0\in\cO^F$ satisfies condition (b) in \Cref{Defgood}. Note that, in view of the discussion in \ref{CuspCS}--\ref{Green}, we only need to determine those $\xi_x$, $\zeta_\jj\gamma_{\ii'}$ and $\delta_\phi$ for which $x$, $\ii'\in\tau^{-1}(\jj)$ and $\phi$ correspond to an element of the form $(u_\cO,\varsigma)\in\cN_\mf G$ under the generalised Springer correspondence. So there are $|\Irr(A_\mf G(u_\cO))|$ many roots of unity to be specified in order to get the full set of values $R_x|_{\cO^F}$ for all $x\in X(\mf W)$. The numbers $c_x(w)$ which appear in formula \eqref{HeckeFormula} in \ref{Ree} can be computed automatically with {\sffamily CHEVIE} \cite{MiChv}.

So our strategy is to go through the unipotent classes $\cO\subseteq\mf G$ one by one and determine the emerging roots of unity as described above. The computations will primarily depend on the isomorphism type of the group $A_\mf G(u_\cO)$ and the (multiplicities of the) different $\jj\in\cM_\mf G$ corresponding to the pairs $(u_\cO,\varsigma)$ under the bijection \eqref{GenSpringBij} in \ref{GenSpring}, so it makes sense to organise our argumentation accordingly. The simplest case occurs when $A_\mf G(u_\cO)=\{1\}$: Here we only need to consider the pair $(u_\cO,1)$, and as mentioned in \ref{Green}, we always have $\delta_{(u_\cO,1)}=+1$, so everything is already known for any such class $\cO$. In general, the possibilities for the groups $A_\mf G(u_\cO)$ are the following:
\[\Z/n\Z\;\text{ with }\;1\leqslant n\leqslant6,\quad\mathfrak S_3,\quad\mathfrak S_5,\quad D_8,\quad\Z/2\Z\times\Z/2\Z,\quad\Z/2\Z\times\mathfrak S_3\]
(where $D_8$ is the dihedral group of order $8$ and $\mathfrak S_n$ is the symmetric group on $n$ letters).
\end{0}

In Sections \ref{p5}--\ref{p2}, we will prove the following result (along with \Cref{Thmgood}; see \ref{Summp5p2} below).

\begin{Prop}\label{Roots1}
(a) Let $\cO\subseteq\mf G$ be a unipotent class not excluded in \Cref{Thmgood}, and let $u_0\in\cO^F$ be a good element. Then for each $\varsigma\in\Irr(A_\mf G(u_0))$, the root of unity $\delta_{(u_0,\varsigma)}$, $\zeta_{\tau((u_0,\varsigma))}\gamma_{(u_0,\varsigma)}$ or $\xi_{(u_0,\varsigma)}$ (depending on whether $(u_0,\varsigma)$ is as in \ref{Green}, \ref{Intermed} or \ref{CuspCS}, respectively) is equal to $+1$.

(b) If $\cO\subseteq\mf G$ is a unipotent class such that there exists some $\varsigma\in\Irr(A_\mf G(u_\cO))$ for which the pair $(u_\cO,\varsigma)\in\cN_\mf G$ is as in either \ref{Intermed} or \ref{CuspCS}, then $\cO$ is not among the classes excluded in \Cref{Thmgood}, so (a) applies.
\end{Prop}

Since the ordinary Green functions (that is, the signs $\delta_\phi$ for $\phi\in\Irr(\mf W)$) are computed independently in \cite{BeSp} (for $p>5$) and \cite{Lueb} (for $p\leqslant5$), the proof of \Cref{Roots1} will complete the determination of the values of unipotent characters at unipotent elements for the groups $\dt E_8(q)$ where $q$ is any power of any prime $p$.

\section{The case \texorpdfstring{$p\geqslant5$}{pgeq5}}\label{p5}

In this section, we assume that $p\geqslant5$.~---~While it is true that the classification of unipotent classes is different in the respective cases $p\geqslant7$ and $p=5$, the only place where this distinction affects the description occurs for the regular unipotent class of $\mf G$. Anything else can be treated in exactly the same manner, so it makes sense to consider the cases $p\geqslant7$ and $p=5$ simultaneously.

\begin{0}{\bf The regular unipotent class $\dt E_8$.}
If $p\geqslant7$, we have $A_\mf G(u_{\dt E_8})=\{1\}$, so there is nothing to do (see the remark in \ref{Method}).

The case $p=5$ has been handled in \cite[4.5.8--4.5.11]{HDiss}: We have $A_\mf G(u_{\dt E_8})=\langle\overline u_{\dt E_8}\rangle\cong\Z/5\Z$, there exists a good element $u_0\in\dt E_8^F$ in the sense of \Cref{Defgood}, and the $\mf G^F$-conjugacy class of $u_0$ is the unique good one contained in $\dt E_8^F$ (so that \Cref{Thmgood} holds with respect to $\cO=\dt E_8$). The four pairs $(u_0,\varsigma)\in\cN_\mf G$ with $1\neq\varsigma\in\Irr(A_\mf G(u_0))$ are all cuspidal, and we have $\xi_{(u_0,\varsigma)}=+1$ for any $1\neq\varsigma\in\Irr(A_\mf G(u_0))$. The pair $(u_0,1)\in\cN_\mf G$ corresponds to the trivial character $1_\mf W$ of $\mf W$, so we have $\delta_{(u_0,1)}=\delta_{1_\mf W}=+1$ (see \ref{Green}).
\end{0}

From now on we no longer need to distinguish between $p\geqslant7$ or $p=5$, as both the classification of the unipotent classes other than $\dt E_8$, the description of the generalised Springer correspondence and all the numerical ingredients which occur in \ref{CuspCS}--\ref{Green} are entirely analogous with respect to any unipotent class $\cO\neq\dt E_8$. For these remaining classes, the possibilities for the isomorphism types of the group $A_\mf G(u_\cO)$ are $\{1\}$ (in which case there is nothing to do), $\Z/2\Z$ (which will be treated in \ref{p5Z2} below), $\mathfrak S_3$ (see \ref{p5S3} below) and $\mathfrak S_5$ (see \ref{p5S5} below).

\begin{0}{\bf The classes $\cO$ with $A_\mf G(u_\cO)\cong\Z/2\Z$.}\label{p5Z2}
Let $\cO\subseteq\mf G$ be a unipotent class for which $A_\mf G(u_\cO)\cong\Z/2\Z$. By the results of \cite{MizE7E8}, $\cO$ is then one of the following classes:
\begin{align*}
&\dt E_7+\dt A_1,\; \dt D_8,\; \dt E_7(a_1)+\dt A_1,\; \dt D_8(a_1),\; \dt D_7(a_1),\; \dt D_6+\dt A_1,\; \dt E_6(a_1)+\dt A_1,\; \dt D_7(a_2),\; \dt E_6(a_1), \\
&\dt D_5+\dt A_2,\; \dt D_6(a_1)+\dt A_1,\; \dt D_6(a_1),\; \dt D_6(a_2),\; \dt A_5+2\dt A_1,\; (\dt A_5+\dt A_1)'',\; \dt D_4+\dt A_2,\; \dt A_4+2\dt A_1, \\
&\dt D_5(a_1),\; \dt A_4+\dt A_1,\; \dt D_4(a_1)+\dt A_2,\; \dt A_4,\; \dt A_3+\dt A_2,\; 2\dt A_2,\; \dt A_2+\dt A_1,\; \dt A_2.
\end{align*}
Note that, as $A_\mf G(u)\cong\Z/2\Z$ for $u\in\cO^F$, the group homomorphism on $A_\mf G(u)$ induced by $F$ is necessarily the identity, so let us for now just fix any element $u_\cO\in\cO^F$ (see \ref{LuMap}). We need to consider the two pairs $(u_\cO,\pm1)\in\cN_\mf G$. By consulting \cite{Sp}, we see that both of these pairs correspond to the triple $\jj=(\mf T_0,\{1\},1)\in\cM_\mf G$ under the bijection \eqref{GenSpringBij} in \ref{GenSpring}, that is, they are in the image of the ordinary Springer correspondence $\Irr(\mf W)\hookrightarrow\cN_\mf G$. As already noted in \ref{Green}, we have $\delta_{(u_\cO,1)}=1$ (and $\delta_{(u_\cO,-1)}\in\{\pm1\}$). We now evaluate $R_x((u_\cO)_a)$ and $m((u_\cO)_a,w)$ as explained in \ref{Method}. As it turns out, this computation can essentially be divided into the following two cases with regard to the above unipotent classes $\cO$:

(a) $\cO\notin\{\dt E_7(a_1)+\dt A_1, \dt D_6(a_1)+\dt A_1\}$. For $a\in A_\mf G(u_\cO)$ and $w\looparrowright\cO$, we get
\[0\leqslant m((u_\cO)_a,w)=f_w(q)\cdot(\delta_{(u_\cO,1)}+\delta_{(u_\cO,-1)}\cdot\varepsilon(a))=f_w(q)\cdot(1+\delta_{(u_\cO,-1)}\cdot\varepsilon(a))\]
where $\varepsilon$ is the non-trivial linear character of $A_\mf G(u_\cO)$ and $f_w(q)$ is the evaluation at $q$ of a certain polynomial $f_w$ with integral coefficients; we have $f_w(q)>0$ for any $q$ (see \cite[1.2]{LuWeylUni} or \cite[\S 5]{GCH} for a conceptual explanation of this fact). It follows that $m((u_\cO)_a,w)=2f_w(q)$ for one value of $a\in A_\mf G(u_\cO)$ and $m((u_\cO)_a,w)=0$ for the other value of $a\in A_\mf G(u_\cO)$. Hence (by possibly changing the representative $u_\cO$), we see that there exists an element $u_0\in\cO^F$ such that $m(u_0,w)=2f_w(q)>0$ and $m((u_0)_{-1},w)=0$. Thus, $u_0$ is good in the sense of \Cref{Defgood}, and the $\mf G^F$-conjugacy class of $u_0$ is the unique good one contained in $\cO^F$, so \Cref{Thmgood} holds for $\cO$. Furthermore, with respect to this $u_0\in\cO^F$, we have $\delta_{(u_0,1)}=\delta_{(u_0,-1)}=+1$. 

(b) $\cO\in\{\dt E_7(a_1)+\dt A_1, \dt D_6(a_1)+\dt A_1\}$. For $a\in A_\mf G(u_\cO)$ and $w\looparrowright\cO$, we get
\[0\leqslant m((u_\cO)_a,w)=f_w(q)\cdot\delta_{(u_\cO,1)}\]
where $f_w(q)>0$.~---~So here we do not obtain any information on the sign $\delta_{(u_\cO,-1)}$ (while we already knew in advance that $\delta_{(u_\cO,1)}=+1$), and any element of $\cO^F$ is good according to \Cref{Defgood}, which is why we needed to exclude $\dt E_7(a_1)+\dt A_1$ and $\dt D_6(a_1)+\dt A_1$ in the formulation of \Cref{Thmgood}. If $p\geqslant7$, it is shown in \cite{BeSp} that we have $\delta_{(u_0,-1)}=+1$ if $u_0\in\cO^F$ is chosen to be a split unipotent element as defined in \Cref{Remsplit}. For $p=5$, see \cite{Lueb}.
\end{0}

\begin{0}{\bf The classes $\cO$ with $A_\mf G(u_\cO)\cong\mathfrak S_3$.}\label{p5S3}
Let $\cO\subseteq\mf G$ be a unipotent class for which $A_\mf G(u_\cO)\cong\mathfrak S_3$. Then $\cO$ is one of the following classes (see \cite{MizE7E8}):
\[\dt E_7(a_2)+\dt A_1,\; \dt A_8,\; \dt D_8(a_3),\; \dt A_5+\dt A_2,\; \dt D_4(a_1)+\dt A_1,\; \dt D_4(a_1).\]
We also see from the results of \cite{MizE7E8} that for each of these classes, there is a unique $u_\cO\in\cO^F$ up to $\mf G^F$-conjugacy such that $F$ induces the identity map on $A_\mf G(u_\cO)$ (see \ref{LuMap}). Let us fix such a $u_\cO\in\cO^F$ and set $u_0:=u_\cO$. Denoting by $1, \varepsilon, r$ the trivial, sign, reflection characters of $A_\mf G(u_0)$, respectively, we need to consider the three pairs $(u_0,1), (u_0,\varepsilon), (u_0,r)\in\cN_\mf G$. According to \cite{Sp}, each of these pairs is in the image of the ordinary Springer correspondence $\Irr(\mf W)\hookrightarrow\cN_\mf G$. Recall from \ref{Green} that we have $\delta_{(u_0,1)}=1$. We now evaluate $R_x((u_0)_a)$ and $m((u_0)_a,w)$ as explained in \ref{Method}.

(a) $\cO\neq\dt D_8(a_3)$. For $a\in A_\mf G(u_0)$ and $w\looparrowright\cO$, we get
\[0\leqslant m((u_0)_a,w)=f_w(q)\cdot(1+\delta_{(u_0,\varepsilon)}\cdot\varepsilon(a)+2\delta_{(u_0,r)}\cdot r(a))\]
where $f_w(q)>0$ is the evaluation at $q$ of a certain polynomial $f_w$ with integral coefficients (cf.\ \ref{p5Z2}). Taking $a=1$, we get
\[0\leqslant m(u_0,w)=f_w(q)\cdot(1+\delta_{(u_0,\varepsilon)}+4\delta_{(u_0,r)}),\]
which only holds when $\delta_{(u_0,r)}=+1$. Taking now $a\in A_\mf G(u_0)\cong\mathfrak S_3$ to be a $3$-cycle, we obtain
\[0\leqslant f_w(q)\cdot(1+\delta_{(u_0,\varepsilon)}-2)=f_w(q)\cdot(\delta_{(u_0,\varepsilon)}-1),\]
which forces $\delta_{(u_0,\varepsilon)}=+1$.  Evaluating $m((u_0)_a,w)$ for $a\in A_\mf G(u_0)$ (and with the known values $\delta_{(u_0,1)}=\delta_{(u_0,\varepsilon)}=\delta_{(u_0,r)}=+1$), we see that $u_0\in\cO^F$ is good in the sense of \Cref{Defgood} and that the $\mf G^F$-conjugacy class of $u_0$ is the unique good one contained in $\cO^F$, so \Cref{Thmgood} holds for $\cO$.

(b) $\cO=\dt D_8(a_3)$. For $a\in A_\mf G(u_0)$ and $w\looparrowright\cO$, we get
\[0\leqslant m((u_0)_a,w)=q^{28}\cdot(1+\delta_{(u_0,\varepsilon)}\cdot\varepsilon(a)).\]
So here our computation does not provide any information on the signs $\delta_{(u_0,\varepsilon)}$ and $\delta_{(u_0,r)}$. It is shown in \cite[\S 3, Case~5]{BeSp} that
\[\delta_{(u_0,r)}=+1\quad\text{and}\quad\delta_{(u_0,\varepsilon)}\equiv q\mod 3.\]
So we have $m(u_0,w)=0$ if $q\equiv-1\mod 3$, which means that there do not exist any good elements in this case; on the other hand, if $q\equiv1\mod 3$, we have $m(u_0,w)=2q^{28}>0$, so $u_0$ is good (and the $\mf G^F$-conjugacy class of $u_0$ is the unique good one contained in $\dt D_8(a_3)^F$). This explains the condition for the class $\dt D_8(a_3)$ in \Cref{Thmgood}(a).
\end{0}

\begin{0}{\bf The class $\cO=2\dt A_4$ with $A_\mf G(u_\cO)\cong\mathfrak S_5$.}\label{p5S5}
Let $\cO=2\dt A_4\subseteq\mf G$ be the (unique) unipotent class with $A_\mf G(u_\cO)\cong\mathfrak S_5$. By \cite[Lemma~70]{MizE7E8}, there exists a unique $u_0=u_\cO\in\cO^F$ up to $\mf G^F$-conjugacy such that $F$ induces the identity map on $A_\mf G(u_0)$. We need to consider $7$ pairs $(u_0,\varsigma)\in\cN_\mf G$ ($\varsigma\in\Irr(A_\mf G(u_0))$). Let us write $\varsigma_\lambda$ for the irreducible character of $A_\mf G(u_0)\cong\mathfrak S_5$ of the Specht module of $\mathfrak S_5$ associated to the partition $\lambda$ of $5$. The results of \cite{Sp} show that the pair $(u_0,\varsigma_{(11111)})$ is a cuspidal pair for $\mf G$ (and has in fact already been considered in \cite[Prop.~4.5.7]{HDiss}), while the remaining $6$ pairs $(u_0,\varsigma)$ are in the image of the ordinary Springer correspondence $\Irr(\mf W)\hookrightarrow\cN_\mf G$. Denoting by $x_0\in X(\mf W)$ the element such that $A_{x_0}$ is the cuspidal unipotent character sheaf corresponding to $(u_0,\varsigma_{(11111)})$, let $\xi=\xi_{x_0}\in\Qlbar$ be the root of unity satisfying $R_{x_0}=\xi\chi_{x_0}^0$ (see \ref{CuspCS}). As for the other $6$ pairs $(u_0,\varsigma)$, we have
\begin{align*}
\phi_{4480,16}\leftrightarrow(u_0,\varsigma_{(5)}),\quad\phi_{5670,18}\leftrightarrow(u_0,\varsigma_{(41)}),\quad\phi_{4536,18}\leftrightarrow(u_0,\varsigma_{(32)}), \\
\quad\phi_{1680,22}\leftrightarrow(u_0,\varsigma_{(311)}),\quad\phi_{1400,20}\leftrightarrow(u_0,\varsigma_{(221)}),\quad\phi_{70,32}\leftrightarrow(u_0,\varsigma_{(2111)}).
\end{align*}
Evaluating $R_x((u_0)_a)$ and $m((u_0)_a,w)$ for $a\in A_\mf G(u_0)$ and $w\looparrowright2\dt A_4$ as explained in \ref{Method}, we get
\begin{align*}
m((u_0)_a,w)=q^{40}\bigl(\delta_{\phi_{4480,16}}\varsigma_{(5)}(a)&+4\delta_{\phi_{5670,18}}\varsigma_{(41)}(a)+5\delta_{\phi_{4536,18}}\varsigma_{(32)}(a)+6\delta_{\phi_{1680,22}}\varsigma_{(311)}(a) \\
&+5\delta_{\phi_{1400,20}}\varsigma_{(221)}(a)+4\delta_{\phi_{70,32}}\varsigma_{(2111)}(a)+\xi\varsigma_{(11111)}(a)\bigr).
\end{align*}
\Cref{Tbl2A4p7} contains the values $m((u_0)_a,w)/q^{40}$ in dependence of the $6$ unknown signs $\delta_\phi$ (for $\phi$ as above) and the root of unity $\xi$, which we can now easily determine using the fact that $m((u_0)_a,w)/q^{40}\geqslant0$: First of all, since $m((u_0)_a,w)/q^{40}$ is in particular a real number, the root of unity $\xi$ is necessarily a sign. Assume, if possible, that $\delta_{\phi_{1680,22}}=-1$. Then the first row of \Cref{Tbl2A4p7} shows that we must have $\delta_{\phi_{5670,18}}=\delta_{\phi_{70,32}}=-1$, but this contradicts the non-negativity of the last row of the table. We deduce that $\delta_{\phi_{1680,22}}=+1$. The row labelled by $(2,2,1)$ in \Cref{Tbl2A4p7} then forces $\delta_{\phi_{4480,16}}=\delta_{\phi_{4536,18}}=\delta_{\phi_{1400,20}}=\xi=+1$. In turn, the row labelled by $(3,1,1)$ yields that $\delta_{\phi_{5670,18}}=\delta_{\phi_{70,32}}=+1$ as well. Computing the $m((u_0)_a,w)$ with these known signs, we see that $m((u_0)_a,w)=0$ whenever $1\neq a\in A_\mf G(u_0)$, while $m(u_0,w)=120q^{40}$. Thus, $u_0\in(2\dt A_4)^F$ is good, and the $\mf G^F$-class of $u_0$ is the unique good one contained in $(2\dt A_4)^F$, so \Cref{Thmgood} holds for the class $2\dt A_4$.
\begin{table}[htbp]
\caption{$m(u,w)/q^{40}$ for $u\in2\dt A_4$ and $w\looparrowright2\dt A_4$}
\label{Tbl2A4p7}
\centering
\begin{tabular}{ll}
\toprule
$a$ & $m((u_0)_a,w)/q^{40}$ \\
\midrule[0.08em]
$(5)$ & $\delta_{\phi_{4480,16}}-\phantom{1}4\delta_{\phi_{5670,18}}\phantom{+25\delta_{\phi_{4536,18}}}\;\,+\phantom{1}6\delta_{\phi_{1680,22}}\phantom{+25\delta_{\phi_{1400,20}}}\;\,
-\phantom{1}4\delta_{\phi_{70,32}}+\xi$ \\
$(4,1)$ & $\delta_{\phi_{4480,16}}\phantom{+16\delta_{\phi_{5670,18}}}\;\,-\phantom{2}5\delta_{\phi_{4536,18}}\phantom{+36\delta_{\phi_{1680,22}}}\;\,
+\phantom{2}5\delta_{\phi_{1400,20}}\phantom{+16\delta_{\phi_{70,32}}}\;\,-\xi$ \\
$(3,2)$ & $\delta_{\phi_{4480,16}}-\phantom{1}4\delta_{\phi_{5670,18}}+\phantom{2}5\delta_{\phi_{4536,18}}\phantom{+36\delta_{\phi_{1680,22}}}\;\,
-\phantom{2}5\delta_{\phi_{1400,20}}+\phantom{1}4\delta_{\phi_{70,32}}-\xi$ \\
$(3,1,1)$ & $\delta_{\phi_{4480,16}}+\phantom{1}4\delta_{\phi_{5670,18}}-\phantom{2}5\delta_{\phi_{4536,18}}\phantom{+36\delta_{\phi_{1680,22}}}\;\,
-\phantom{2}5\delta_{\phi_{1400,20}}+\phantom{1}4\delta_{\phi_{70,32}}+\xi$ \\
$(2,2,1)$ & $\delta_{\phi_{4480,16}}\phantom{+16\delta_{\phi_{5670,18}}}\;\,+\phantom{2}5\delta_{\phi_{4536,18}}-12\delta_{\phi_{1680,22}}
+\phantom{2}5\delta_{\phi_{1400,20}}\phantom{+16\delta_{\phi_{70,32}}}\;\,+\xi$ \\
$(2,1,1,1)$ & $\delta_{\phi_{4480,16}}+\phantom{1}8\delta_{\phi_{5670,18}}+\phantom{2}5\delta_{\phi_{4536,18}}\phantom{+36\delta_{\phi_{1680,22}}}\;\,
-\phantom{2}5\delta_{\phi_{1400,20}}-\phantom{1}8\delta_{\phi_{70,32}}-\xi$ \\
$(1,1,1,1,1)$ & $\delta_{\phi_{4480,16}}+16\delta_{\phi_{5670,18}}+25\delta_{\phi_{4536,18}}+36\delta_{\phi_{1680,22}}
+25\delta_{\phi_{1400,20}}+16\delta_{\phi_{70,32}}+\xi$ \\
\bottomrule
\end{tabular}
\end{table}
\end{0}

\begin{0}{\bf Proof of \Cref{Corsplitgood}.}\label{ProofCorsplitgood}
Let $p\geqslant7$. We have seen in \ref{p5Z2}(b) that both of the two $\mf G^F$-classes contained in $(\dt E_7(a_1)+\dt A_1)^F$, $(\dt D_6(a_1)+\dt A_1)^F$ (respectively) consist entirely of good elements. Furthermore, the class $\dt D_8(a_3)$ neither contains any split nor any good elements when $q\equiv-1\mod3$ (see \Cref{Remsplit} and \ref{p5S3}(b)). If $\cO$ is one of the classes $\dt D_8(a_3)$ when $q\equiv1\mod3$ or $2\dt A_4$, there exists a unique $u_0\in\cO^F$ up to $\mf G^F$-conjugacy for which $F$ acts trivially on $A_\mf G(u_0)$ and both the split and the good elements have this property, so \Cref{Corsplitgood} holds with respect to these $\cO$ as well.

Let $\cO$ be any unipotent conjugacy class of $\mf G$ that we did not yet consider. Our arguments in this section show that $\cO^F$ contains a unique $\mf G^F$-class consisting of good elements, and any pair $(u_\cO,\varsigma)\in\cN_\mf G$ with $\varsigma\in\Irr(A_\mf G(u_\cO))$ is in the image of the ordinary Springer correspondence $\Irr(\mf W)\hookrightarrow\cN_\mf G$; moreover, the $\mf G^F$-class of a good element $u_0\in\cO^F$ is characterised by the property that $\delta_{(u_0,\varsigma)}=+1$ for any $\varsigma\in\Irr(A_\mf G(u_0))$. As mentioned in \cite[p.~590]{BeSp}, split elements also have this property, so the good and the split $\mf G^F$-class inside $\cO^F$ must coincide. This establishes \Cref{Corsplitgood}.
\end{0}

\section{The case \texorpdfstring{$p=3$}{p=3}}\label{p3}

In this section, we assume that $p=3$. In view of the results of \cite{MizE7E8} (see also \cite{Sp}), the possibilities for the isomorphism types of the group $A_\mf G(u_\cO)$ are $\{1\}$ (in which case there is nothing to do), $\Z/2\Z$ (which will be treated in \ref{p3Z2} below), $\Z/3\Z$ (see \ref{p3Z3} below), $\Z/6\Z$ (see \ref{p3Z6} below), $\mathfrak S_3$ (see \ref{p3S3} below) and $\mathfrak S_5$ (see \ref{p3S5} below).

\begin{0}{\bf The classes $\cO$ with $A_\mf G(u_\cO)\cong\Z/2\Z$.}\label{p3Z2}
Let $\cO\subseteq\mf G$ be a unipotent class for which $A_\mf G(u_\cO)\cong\Z/2\Z$. By the results of \cite{MizE7E8}, $\cO$ is then one of the following classes:
\begin{align*}
&\dt D_8,\; \dt E_7(a_1)+\dt A_1,\; \dt D_8(a_1),\; \dt D_7(a_1),\; \dt D_6+\dt A_1,\; \dt D_8(a_3),\; \dt E_6(a_1)+\dt A_1,\; \dt D_7(a_2),\; \dt E_6(a_1), \\
&\dt D_5+\dt A_2,\; \dt D_6(a_1)+\dt A_1,\; \dt D_6(a_1),\; \dt D_6(a_2),\; \dt A_5+2\dt A_1,\; (\dt A_5+\dt A_1)'',\; \dt D_4+\dt A_2,\; \dt A_4+2\dt A_1, \\
&\dt D_5(a_1),\; \dt A_4+\dt A_1,\; \dt D_4(a_1)+\dt A_2,\; \dt A_4,\; \dt A_3+\dt A_2,\; 2\dt A_2,\; \dt A_2+\dt A_1,\; \dt A_2.
\end{align*}
The argumentation is entirely analogous to the one in \ref{p5Z2}: In particular, if $\cO$ is one of the above classes not equal to $\dt E_7(a_1)+\dt A_1$ or $\dt D_6(a_1)+\dt A_1$, we find a good element $u_0\in\cO^F$ which is uniquely determined up to $\mf G^F$-conjugacy, both of the two pairs $(u_0,\pm1)\in\cN_\mf G$ are in the image of the ordinary Springer correspondence $\Irr(\mf W)\hookrightarrow\cN_\mf G$, and we have $\delta_{(u_0,1)}=\delta_{(u_0,-1)}=+1$. (This includes the class $\cO=\dt D_8(a_3)$, which did not appear in \ref{p5Z2}, so we do not have to exclude it in the formulation of \Cref{Thmgood}.)

If, on the other hand, $\cO\in\{\dt E_7(a_1)+\dt A_1, \dt D_6(a_1)+\dt A_1\}$, we can (as in \ref{p5Z2}) only conclude that $\delta_{(u_\cO,1)}=+1$ (regardless of the choice of $u_\cO\in\cO^F$) and that any element of $\cO^F$ is good, while we do not obtain any information on the sign $\delta_{(u_\cO,-1)}$. We refer to \cite{Lueb}.
\end{0}

\begin{0}{\bf The classes $\cO$ with $A_\mf G(u_\cO)\cong\Z/3\Z$.}\label{p3Z3}
Let $\cO\subseteq\mf G$ be a unipotent class for which $A_\mf G(u_\cO)\cong\Z/3\Z$. Then $\cO$ is one of the following classes (see \cite{MizE7E8}):
\[\dt E_8,\; \dt E_8(a_1),\; \dt E_7,\; \dt E_6+\dt A_1,\; \dt E_6.\]
For any $u\in\cO^F$, $F$ induces the identity map on $A_\mf G(u)$, so this condition does not constitute any restriction on the choice of $u_\cO\in\cO^F$ (see \ref{LuMap}). However, since $C_\mf G(u_\cO)=C_\mf G(u_\cO^{-1})$ and since any one of the above $\cO$ is uniquely determined by $\dim C_\mf G(u_\cO)$ along with the isomorphism type of $A_\mf G(u_\cO)$, we must have $u_\cO^{-1}\in\cO$. As $\cO^F$ is the union of an odd number of $\mf G^F$-conjugacy classes, we conclude that there must be (at least) one $\mf G^F$-conjugacy class contained in $\cO^F$ which is stable under taking inverses; let $u_0\in\cO^F$ be a representative for such a class, that is, $u_0$ is $\mf G^F$-conjugate to $u_0^{-1}$. We fix a generator $a$ of $A_\mf G(u_0)$ (unless $\cO=\dt E_6+\dt A_1$, we could take $a=\overline u_0$; see \cite[Table~17.4]{LieSei}), and let $\omega\in\Qlbar$ be a fixed primitive $3$rd root of unity, so that we can identify the irreducible characters of $A_\mf G(u_0)$ with their values at $a$. We need to consider the three pairs $(u_0,1), (u_0,\omega), (u_0,\omega^2)\in\cN_\mf G$. While the pair $(u_0,1)$ is in the image of the ordinary Springer correspondence $\Irr(\mf W)\hookrightarrow\cN_\mf G$, the pairs $(u_0,\omega)$ and $(u_0,\omega^2)$ are \enquote{intermediate} in the sense of \ref{Intermed}: We have $\tau((u_0,\omega))=(\dt E_6,u_{\dt E_6},\omega)\in\cM_\mf G$ and $\tau((u_0,\omega^2))=(\dt E_6,u_{\dt E_6},\omega^2)\in\cM_\mf G$. Let $x_1, x_2\in X(\mf W)$ be such that $x_1\leftrightarrow(u_0,\omega)$, $x_2\leftrightarrow(u_0,\omega^2)$. An inspection of the Fourier matrix (cf.\ \ref{E8Setup}) shows that we have $\overline R_{x_1}=R_{x_2}$ and that both $R_{x_1}$ and $R_{x_2}$ are linear combinations of unipotent characters with rational coefficients. By our choice of $u_0$, we thus have
\[R_{x_1}(u_0)=R_{x_1}(u_0^{-1})=\overline{R_{x_1}(u_0)}=R_{x_2}(u_0).\]
The discussion in \ref{Intermed} then shows that we have $\zeta_{(\dt E_6,u_{\dt E_6},\omega)}\gamma_{(u_0,\omega)}=\zeta_{(\dt E_6,u_{\dt E_6},\omega^2)}\gamma_{(u_0,\omega^2)}\in\{\pm1\}$. Evaluating $m(u_0,w)$ for $w\looparrowright\cO$ as explained in \ref{Method} gives
\[0\leqslant m(u_0,w)=f_w(q)(1+2\zeta_{(\dt E_6,u_{\dt E_6},\omega)}\gamma_{(u_0,\omega)})\]
where $f_w(q)>0$ (cf.\ \ref{p5Z2}). This would be false if $\zeta_{(\dt E_6,u_{\dt E_6},\omega)}\gamma_{(u_0,\omega)}=-1$, so we deduce that $\zeta_{(\dt E_6,u_{\dt E_6},\omega)}\gamma_{(u_0,\omega)}=\zeta_{(\dt E_6,u_{\dt E_6},\omega^2)}\gamma_{(u_0,\omega^2)}=+1$. We thus have $m(u_0,w)=3f_w(q)>0$, which shows that $u_0$ is good. Evaluating $m((u_0)_b,w)$ for $1\neq b\in A_\mf G(u_0)$, we get $m((u_0)_b,w)=0$. Hence, the $\mf G^F$-class of $u_0$ is the unique good one contained in $\cO^F$, and \Cref{Thmgood} holds for $\cO$. As a side remark, note that the only assumption which we made on $u_0$ in the first place is that it is $\mf G^F$-conjugate to $u_0^{-1}$, and this was sufficient to deduce that the $\mf G^F$-class of $u_0$ is the unique good one contained in $\cO^F$. Hence, the $\mf G^F$-class of $u_0$ is in fact the only one contained in $\cO^F$ which is stable under taking inverses (i.e., it is the unique real conjugacy class of $\mf G^F$ contained in $\cO^F$).
\end{0}

\begin{0}{\bf The class $\cO=\dt E_7+\dt A_1$ with $A_\mf G(u_\cO)\cong\Z/6\Z$.}\label{p3Z6}
Let $\cO=\dt E_7+\dt A_1\subseteq\mf G$ be the (unique) unipotent class with $A_\mf G(u_\cO)\cong\Z/6\Z$. This has been treated in \cite[4.5.14--4.5.23]{HDiss}, so let us only summarise the result here; note that the argument relies upon the solution of the last open cases in the generalised Springer correspondence \cite{HSpring}. There exists a unique good element $u_0\in\cO^F$ up to $\mf G^F$-conjugacy (so \Cref{Thmgood} holds for $\dt E_7+\dt A_1$), the Frobenius map $F$ acts trivially on $A_\mf G(u_0)$, and we can choose a generator $a$ of $A_\mf G(u_0)$ such that $a^2=\overline u_0^{\,2}$. Let $\omega\in\Qlbar$ be a fixed primitive $3$rd root of unity, and let us denote the irreducible characters of $A_\mf G(u_0)$ by their values at $a$. The pairs $(u_0,\pm1)\in\cN_\mf G$ are in the image of the ordinary Springer correspondence $\Irr(\mf W)\hookrightarrow\cN_\mf G$; we have $\phi_{112,3}\leftrightarrow(u_0,1)$ and $\phi_{28,8}\leftrightarrow(u_0,-1)$. The pairs $(u_0,\omega)$ and $(u_0,\omega^2)$ are \enquote{intermediate} in the sense of \ref{Intermed}; we have $\tau((u_0,\omega))=(\dt E_6,u_{\dt E_6},\omega)\in\cM_\mf G$ and $\tau((u_0,\omega^2))=(\dt E_6,u_{\dt E_6},\omega^2)\in\cM_\mf G$. The pairs $(u_0,-\omega)$ and $(u_0,-\omega^2)$ are cuspidal. Evaluating $m((u_0)_{a^j},w)$ for $0\leqslant j\leqslant5$ and $w\looparrowright\dt E_7+\dt A_1$, one finds that
\[\delta_{\phi_{112,3}}=\delta_{\phi_{28,8}}=\zeta_{(\dt E_6,u_{\dt E_6},\omega)}\gamma_{(u_0,\omega)}=\zeta_{(\dt E_6,u_{\dt E_6},\omega^2)}\gamma_{(u_0,\omega^2)}=\xi_{(u_0,-\omega)}=\xi_{(u_0,-\omega^2)}=+1.\]
\end{0}

\begin{0}{\bf The classes $\cO$ with $A_\mf G(u_\cO)\cong\mathfrak S_3$.}\label{p3S3}
Let $\cO\subseteq\mf G$ be a unipotent class for which $A_\mf G(u_\cO)\cong\mathfrak S_3$. Then $\cO$ is one of the following classes (see \cite{MizE7E8}):
\[\dt E_7(a_2)+\dt A_1,\; \dt A_8,\; \dt A_5+\dt A_2,\; \dt D_4(a_1)+\dt A_1,\; \dt D_4(a_1).\]
The argumentation with respect to the above classes is entirely analogous to the one in \ref{p5S3}; but note that the class $\dt D_8(a_3)$ does not appear here (since $A_\mf G(u)\cong\Z/2\Z$ for $u\in\dt D_8(a_3)$): If $\cO$ is one of the above classes, there is a unique element $u_0\in\cO^F$ up to $\mf G^F$-conjugacy for which $F$ acts trivially on $A_\mf G(u_0)$ and $u_0$ is good, so \Cref{Thmgood} holds for $\cO$. All of the $3$ pairs $(u_0,1), (u_0,\varepsilon), (u_0,r)$ are in the image of the ordinary Springer correspondence $\Irr(\mf W)\hookrightarrow\cN_\mf G$, and we have $\delta_{(u_0,1)}=\delta_{(u_0,\varepsilon)}=\delta_{(u_0,r)}=+1$.
\end{0}

\begin{0}{\bf The class $\cO=2\dt A_4$ with $A_\mf G(u_\cO)\cong\mathfrak S_5$.}\label{p3S5}
Let $\cO=2\dt A_4\subseteq\mf G$ be the (unique) unipotent class with $A_\mf G(u_\cO)\cong\mathfrak S_5$. The argument in this case can be copied verbatim from \ref{p5S5}. In particular, there exists a unique $u_0=u_\cO\in\cO^F$ up to $\mf G^F$-conjugacy such that $F$ induces the identity map on $A_\mf G(u_0)$, the $\mf G^F$-class of $u_0$ is the unique good one contained in $(2\dt A_4)^F$ (so \Cref{Thmgood} holds for the class $2\dt A_4$), and with the notation analogous to the one in \ref{p5S5}, we have
\[\delta_{\phi_{1680,22}}=\delta_{\phi_{4480,16}}=\delta_{\phi_{4536,18}}=\delta_{\phi_{1400,20}}=\delta_{\phi_{5670,18}}=\delta_{\phi_{70,32}}=\xi=+1.\]
\end{0}

\section{The case \texorpdfstring{$p=2$}{p=2}}\label{p2}

In this section, we assume that $p=2$. In view of the results of \cite{MizE7E8} (see also \cite{Sp}), the possibilities for the isomorphism types of the group $A_\mf G(u_\cO)$ are $\{1\}$ (in which case there is nothing to do), $\Z/2\Z$ (which will be treated in \ref{p2Z2} below), $\Z/4\Z$ (see \ref{p2Z4} below), $\Z/2\Z\times\Z/2\Z$ (see \ref{p2Z2Z2} below), $D_8$ (see \ref{p2D8} below), $\mathfrak S_3$ (see \ref{p2S3} below) $\mathfrak S_3\times\Z/2\Z$ (see \ref{p2S3Z2} below), and $\mathfrak S_5$ (see \ref{p2S5} below).

\begin{0}{\bf The classes $\cO$ with $A_\mf G(u_\cO)\cong\Z/2\Z$.}\label{p2Z2}
Let $\cO\subseteq\mf G$ be a unipotent class for which $A_\mf G(u_\cO)\cong\Z/2\Z$. By the results of \cite{MizE7E8}, $\cO$ is then one of the following classes:
\begin{align*}
&\dt E_8(a_2),\; \dt D_8,\; \dt E_7(a_1)+\dt A_1,\; \dt E_7(a_1),\; \dt D_7,\; \dt E_7(a_2),\; \dt E_6+\dt A_1,\;(\dt D_7(a_1))_2,\; \dt D_6+\dt A_1,\; \dt E_6(a_1)+\dt A_1, \\
& \dt E_6,\;\dt D_7(a_2),\; \dt D_6,\; \dt E_6(a_1), \; (\dt D_5+\dt A_2)_2,\; \dt A_6,\; \dt D_5+\dt A_1,\;\dt D_6(a_2),\; \dt A_5+2\dt A_1,\;\dt D_5,\; (\dt A_5+\dt A_1)'',
\\
&(\dt D_4+\dt A_2)_2,\; \dt A_4+2\dt A_1,\; \dt D_5(a_1),\; \dt A_4+\dt A_1,\; \dt D_4+\dt A_1,\; \dt D_4(a_1)+\dt A_2,\; \dt A_4,\; \dt D_4,\; 2\dt A_2,\; \dt A_2+\dt A_1,\; \dt A_2.
\end{align*}
Since $A_\mf G(u)\cong\Z/2\Z$ for $u\in\cO^F$, the group homomorphism on $A_\mf G(u)$ induced by $F$ is necessarily the identity, so let us for now just fix any element $u_\cO\in\cO^F$ (see \ref{LuMap}). We need to consider the two pairs $(u_\cO,\pm1)\in\cN_\mf G$. While $(u_\cO,1)$ is always in the image of the ordinary Springer correspondence $\Irr(\mf W)\hookrightarrow\cN_\mf G$, the element $\jj=\tau((u_\cO,-1))\in\cM_\mf G$ depends on $\cO$: We either have $\jj=(\mf T_0,\{1\},1)$ (so that $(u_\cO,-1)$ is in the image of the ordinary Springer correspondence), or else $\jj=(\dt D_4,\cO_0,\cE_0)$. So let us split our consideration according to these two cases.

(a) Let us assume that $(u_\cO,-1)$ is in the image of the ordinary Springer correspondence; then $\cO$ is one of the following classes:
\begin{align*}
&\dt D_8,\;  \dt D_6+\dt A_1,\; \dt E_6(a_1)+\dt A_1,\; \dt D_7(a_2),\; \dt E_6(a_1), \; \dt A_6,\;\dt D_6(a_2),\; \dt A_5+2\dt A_1,\; (\dt A_5+\dt A_1)'',
\\
&\dt A_4+2\dt A_1,\; \dt D_5(a_1),\; \dt A_4+\dt A_1,\; \dt D_4(a_1)+\dt A_2,\; \dt A_4,\; 2\dt A_2,\; \dt A_2+\dt A_1,\; \dt A_2.
\end{align*}
The argumentation is entirely analogous to the one in \ref{p5Z2}, but note that the classes $\dt E_7(a_1)+\dt A_1$ and $\dt D_6(a_1)+\dt A_1$ do not appear in the above list. So there always is a good element $u_0\in\cO^F$, which is uniquely determined up to $\mf G^F$-conjugacy (so \Cref{Thmgood} holds for $\cO$), and we have $\delta_{(u_0,1)}=\delta_{(u_0,-1)}=+1$.

(b) Let us assume that $(u_\cO,-1)$ is not in the image of the ordinary Springer correspondence, so that $\tau((u_\cO,-1))=(\dt D_4,\cO_0,\cE_0)$. Then $\cO$ is one of the following classes:
\begin{align*}
&\dt E_8(a_2),\; \dt E_7(a_1)+\dt A_1,\; \dt E_7(a_1),\; \dt D_7,\; \dt E_7(a_2),\; \dt E_6+\dt A_1,\; \dt E_6+\dt A_1,\;(\dt D_7(a_1))_2, \\
& \dt E_6,\; \dt D_6,\; (\dt D_5+\dt A_2)_2,\; \dt D_5+\dt A_1,\; \dt D_5,\;(\dt D_4+\dt A_2)_2,\; \dt D_4+\dt A_1,\; \dt D_4.
\end{align*}
Evaluating $R_x((u_\cO)_a)$ and $m((u_\cO)_a,w)$ for $a\in A_\mf G(u_\cO)$ and $w\looparrowright\cO$ gives
\[0\leqslant m((u_\cO)_a,w)=f_w(q)\cdot\bigl(1+\zeta_{(\dt D_4,\cO_0,\cE_0)}\gamma_{(u_\cO,-1)}\cdot\varepsilon(a)\bigr)\]
where $f_w(q)>0$ (cf.\ \ref{p5Z2}) and where $\varepsilon$ is the non-trivial linear character of $A_\mf G(u_\cO)\cong\Z/2\Z$. Since $m((u_\cO)_a,w)\in\R$ and $\zeta_{(\dt D_4,\cO_0,\cE_0)}\gamma_{(u_\cO,-1)}$ is a root of unity, we must have $\zeta_{(\dt D_4,\cO_0,\cE_0)}\gamma_{(u_\cO,-1)}\in\{\pm1\}$. Hence, as in \ref{p5Z2}, we conclude that we can choose $u_\cO=u_0\in\cO^F$ such that $m(u_0,w)=2f_w(q)>0$ and $m((u_0)_{-1},w)=0$. This shows that $u_0$ is good and that the $\mf G^F$-conjugacy class of $u_0$ is the unique good one contained in $\cO^F$, so \Cref{Thmgood} holds for $\cO$. Furthermore, with respect to this $u_0\in\cO^F$, we have $\zeta_{(\dt D_4,\cO_0,\cE_0)}\gamma_{(u_0,-1)}=+1$.
\end{0}

\begin{0}{\bf The classes $\cO$ with $A_\mf G(u_\cO)\cong\Z/4\Z$.}\label{p2Z4}
Let $\cO\subseteq\mf G$ be a unipotent class for which $A_\mf G(u_\cO)\cong\Z/4\Z$. Then $\cO$ is one of the following classes (see \cite{MizE7E8}):
\[\dt E_8,\; \dt E_8(a_1),\; \dt E_7.\]
We have $A_\mf G(u)=\langle\overline u\rangle$ for any $u\in\cO^F$, so $F$ acts trivially on $A_\mf G(u)$. Let us for now pick any $u_\cO\in\cO^F$ (see \ref{LuMap}). Denoting by $\I\in\Qlbar$ a fixed primitive $4$th root of unity, we identify the irreducible characters of $A_\mf G(u_\cO)$ with their values at $\overline u_\cO$.

(a) Let us assume that $\cO\in\{\dt E_8,\dt E_7\}$. The pair $(u_\cO,-1)$ corresponds to the triple $(\dt D_4,\cO_0,\cE_0)\in\cM_\mf G$, while $\tau((u_\cO,\I))=(\dt E_7,u_{\dt E_7},\I)$ and $\tau((u_\cO,-\I))=(\dt E_7,u_{\dt E_7},-\I)$. The values of $m\bigl((u_\cO)_{\overline u_\cO^{\,j}},w\bigr)/f_w(q)$ for $0\leqslant j\leqslant3$ and $w\looparrowright\cO$ are given in \Cref{Tblp2Z4} where, as before, $f_w(q)>0$ is the evaluation at $q$ of a certain polynomial attached to $w\looparrowright\cO$ (cf.\ \ref{p5Z2}).
\begin{table}[htbp]
\caption{$m(u,w)/f_w(q)$ for $u\in\cO$ and $w\looparrowright\cO$}
\label{Tblp2Z4}
\centering
\begin{tabular}{ll}
\toprule
$a$ & $m((u_\cO)_a,w)/f_w(q)$ \\
\midrule[0.08em]
$1$ & $1+\zeta_{(\dt D_4,\cO_0,\cE_0)}\gamma_{(u_\cO,-1)}+\phantom{\I}\zeta_{(\dt E_7,u_{\dt E_7},\I)}\gamma_{(u_\cO,\I)}+\phantom{\I}\zeta_{(\dt E_7,u_{\dt E_7},-\I)}\gamma_{(u_\cO,-\I)}$ \\
$\overline u_\cO$ & $1-\zeta_{(\dt D_4,\cO_0,\cE_0)}\gamma_{(u_\cO,-1)}+\I\zeta_{(\dt E_7,u_{\dt E_7},\I)}\gamma_{(u_\cO,\I)}-\I\zeta_{(\dt E_7,u_{\dt E_7},-\I)}\gamma_{(u_\cO,-\I)}$ \\
$\overline u_\cO^{\,2}$ & $1+\zeta_{(\dt D_4,\cO_0,\cE_0)}\gamma_{(u_\cO,-1)}-\phantom{\I}\zeta_{(\dt E_7,u_{\dt E_7},\I)}\gamma_{(u_\cO,\I)}-\phantom{\I}\zeta_{(\dt E_7,u_{\dt E_7},-\I)}\gamma_{(u_\cO,-\I)}$ \\
$\overline u_\cO^{\,3}$ & $1-\zeta_{(\dt D_4,\cO_0,\cE_0)}\gamma_{(u_\cO,-1)}-\I\zeta_{(\dt E_7,u_{\dt E_7},\I)}\gamma_{(u_\cO,\I)}+\I\zeta_{(\dt E_7,u_{\dt E_7},-\I)}\gamma_{(u_\cO,-\I)}$ \\
\bottomrule
\end{tabular}
\end{table}
Since $m((u_\cO)_a,w)/f_w(q)\geqslant0$ for $a\in A_\mf G(u_\cO)$ and since the sum of all values in \Cref{Tblp2Z4} is equal to $4>0$, there must be at least one entry which is strictly positive. Changing the representative $u_\cO$, if necessary, we may assume that $u_0=u_\cO\in\cO^F$ is such that the value in the first line of the table is strictly positive. Thus, we have $m(u_0,w)>0$ for $w\looparrowright\cO$, so $u_0$ is good. With this choice of $u_0=u_\cO$, adding the first and the third row of \Cref{Tblp2Z4} gives $2+2\zeta_{(\dt D_4,\cO_0,\cE_0)}\gamma_{(u_0,-1)}\in\R_{>0}$; since $\zeta_{(\dt D_4,\cO_0,\cE_0)}\gamma_{(u_0,-1)}$ is a root of unity, this forces $\zeta_{(\dt D_4,\cO_0,\cE_0)}\gamma_{(u_0,-1)}=+1$. Now both of the values in the second and fourth row of \Cref{Tblp2Z4} are non-negative and their sum is $2-2\zeta_{(\dt D_4,\cO_0,\cE_0)}\gamma_{(u_0,-1)}=0$, so both of the values in the second and fourth row are in fact equal to $0$, that is, we have $\zeta_{(\dt E_7,u_{\dt E_7},\I)}\gamma_{(u_0,\I)}=\zeta_{(\dt E_7,u_{\dt E_7},-\I)}\gamma_{(u_0,-\I)}$. From the first row of the table, we then get $2+2\zeta_{(\dt E_7,u_{\dt E_7},\I)}\gamma_{(u_0,\I)}\in\R_{>0}$, so we must have $\zeta_{(\dt E_7,u_{\dt E_7},\I)}\gamma_{(u_0,\I)}=\zeta_{(\dt E_7,u_{\dt E_7},-\I)}\gamma_{(u_0,-\I)}=+1$. We have shown that the entry in the first row of \Cref{Tblp2Z4} is the unique one which is non-zero (it is equal to $4$), which proves that the $\mf G^F$-conjugacy class of $u_0\in\cO^F$ is the unique good one contained in $\cO^F$, so \Cref{Thmgood} holds for $\cO$. Furthermore, with respect to this choice of $u_0$, we have
\[\delta_{(u_0,1)}=\zeta_{(\dt D_4,\cO_0,\cE_0)}\gamma_{(u_0,-1)}=\zeta_{(\dt E_7,u_{\dt E_7},\I)}\gamma_{(u_0,\I)}=\zeta_{(\dt E_7,u_{\dt E_7},-\I)}\gamma_{(u_0,-\I)}=+1.\]

(b) Let us now assume that $\cO=\dt E_8(a_1)$. This has been considered in \cite[4.5.24--4.5.31]{HDiss}, but we can also use a similar argument as above: While the pair $(u_\cO,-1)$ still corresponds to the triple $(\dt D_4,\cO_0,\cE_0)\in\cM_\mf G$, the two pairs $(u_\cO,\pm\I)$ are now cuspidal. Evaluating $m\bigl((u_\cO)_{\overline u_\cO^{\,j}},w\bigr)$ for $0\leqslant j\leqslant3$ and $w\looparrowright\dt E_8(a_1)$, one can then argue as in (a) to deduce that there exists a good element $u_0\in\dt E_8(a_1)^F$ which is uniquely determined up to $\mf G^F$-conjugacy (so \Cref{Thmgood} holds for $\dt E_8(a_1)$), and with respect to this choice of $u_0$, we have
\[\delta_{(u_0,1)}=\zeta_{(\dt D_4,\cO_0,\cE_0)}\gamma_{(u_0,-1)}=\xi_{(u_0,\I)}=\xi_{(u_0,-\I)}=+1.\]
\end{0}

\begin{0}{\bf The classes $\cO$ with $A_\mf G(u_\cO)\cong\Z/2\Z\times\Z/2\Z$.}\label{p2Z2Z2}
Let $\cO\subseteq\mf G$ be a unipotent class for which $A_\mf G(u_\cO)\cong\Z/2\Z\times\Z/2\Z$. By the results of \cite{MizE7E8}, $\cO$ is then one of the following classes:
\begin{align*}
\dt E_7+\dt A_1,\;\dt D_6(a_1).
\end{align*}
For any $u\in\cO^F$, $F$ acts trivially on $A_\mf G(u)$, so let us for now just fix any $u_\cO\in\cO^F$ (see \ref{LuMap}). Denoting by $ 1$ the trivial and by $\varepsilon$ the non-trivial linear character of $\Z/2\Z$, we have
\[\Irr(A_\mf G(u_\cO))=\{1\boxtimes1,\varepsilon\boxtimes1,1\boxtimes\varepsilon,\varepsilon\boxtimes\varepsilon\}.\]
The pairs $(u_\cO,1\boxtimes1), (u_\cO,\varepsilon\boxtimes\varepsilon)\in\cN_\mf G$ are in the image of the ordinary Springer correspondence, while $\tau((u_\cO,1\boxtimes\varepsilon))=\tau((u_\cO,\varepsilon\boxtimes1))=(\dt D_4,\cO_0,\cE_0)\in\cM_\mf G$. The values of $m((u_\cO)_a,w)/f_w(q)$ for $a\in A_\mf G(u_\cO)$ and $w\looparrowright\cO$ (and with $f_w(q)>0$ as before, cf.\ \ref{p5Z2}) are given in \Cref{TblZ2Z2}.
\begin{table}[htbp]
\caption{$m(u,w)/f_w(q)$ for $u\in\cO$ and $w\looparrowright\cO$}
\label{TblZ2Z2}
\centering
\begin{tabular}{ll}
\toprule
$a$ & $m((u_\cO)_a,w)/f_w(q)$ \\
\midrule[0.08em]
$(0,0)$ & $1+\delta_{(u_\cO,\varepsilon\boxtimes\varepsilon)}+\zeta_{(\dt D_4,\cO_0,\cE_0)}\gamma_{(u_\cO,1\boxtimes\varepsilon)}+\zeta_{(\dt D_4,\cO_0,\cE_0)}\gamma_{(u_\cO,\varepsilon\boxtimes1)}$ \\
$(0,1)$ & $1-\delta_{(u_\cO,\varepsilon\boxtimes\varepsilon)}-\zeta_{(\dt D_4,\cO_0,\cE_0)}\gamma_{(u_\cO,1\boxtimes\varepsilon)}+\zeta_{(\dt D_4,\cO_0,\cE_0)}\gamma_{(u_\cO,\varepsilon\boxtimes1)}$ \\
$(1,0)$ & $1-\delta_{(u_\cO,\varepsilon\boxtimes\varepsilon)}+\zeta_{(\dt D_4,\cO_0,\cE_0)}\gamma_{(u_\cO,1\boxtimes\varepsilon)}-\zeta_{(\dt D_4,\cO_0,\cE_0)}\gamma_{(u_\cO,\varepsilon\boxtimes1)}$ \\
$(1,1)$ & $1+\delta_{(u_\cO,\varepsilon\boxtimes\varepsilon)}-\zeta_{(\dt D_4,\cO_0,\cE_0)}\gamma_{(u_\cO,1\boxtimes\varepsilon)}-\zeta_{(\dt D_4,\cO_0,\cE_0)}\gamma_{(u_\cO,\varepsilon\boxtimes1)}$ \\
\bottomrule
\end{tabular}
\end{table}
Since $m((u_\cO)_a,w)/f_w(q)\geqslant0$ for $a\in A_\mf G(u_\cO)$ and since the sum of all values in \Cref{TblZ2Z2} is equal to $4>0$, there must be at least one entry in the table which is strictly positive. By changing the representative $u_\cO$, if necessary, we may assume without loss of generality that the first row is strictly positive and write $u_0=u_\cO$ for the corresponding representative. Thus, we have $m(u_0,w)>0$, so $u_0\in\cO^F$ is good. Taking the sum of the first and the last row of the table, we obtain $2+2\delta_{(u_\cO,\varepsilon\boxtimes\varepsilon)}>0$, which forces $\delta_{(u_\cO,\varepsilon\boxtimes\varepsilon)}=+1$. Adding the first and the third row of the table gives $2+2\zeta_{(\dt D_4,\cO_0,\cE_0)}\gamma_{(u_\cO,1\boxtimes\varepsilon)}\in\R_{>0}$, which forces the root of unity $\zeta_{(\dt D_4,\cO_0,\cE_0)}\gamma_{(u_\cO,1\boxtimes\varepsilon)}$ to be $+1$. Similarly, taking the sum of the first and the second row of the table, we deduce that $\zeta_{(\dt D_4,\cO_0,\cE_0)}\gamma_{(u_\cO,\varepsilon\boxtimes1)}=+1$. Evaluating $m((u_0)_a,w)$ (for $a\in A_\mf G(u_0)$ and $w\looparrowright\cO$) with these known signs, we see that the $\mf G^F$-conjugacy class of $u_0$ is the unique good one contained in $\cO^F$, so \Cref{Thmgood} holds for $\dt E_7+\dt A_1$ and $\dt D_6(a_1)$.
\end{0}

\begin{0}{\bf The class $\cO=\dt D_8(a_1)$ with $A_\mf G(u_\cO)\cong D_8$.}\label{p2D8}
Let $\cO=\dt D_8(a_1)\subseteq\mf G$ be the (unique) unipotent class with $A_\mf G(u_\cO)\cong D_8$. By \cite[Lm.~44]{MizE7E8}, there exist two $\mf G^F$-classes of elements $u\in\cO^F$ for which $F$ acts trivially on $A_\mf G(u)$. Let us for now denote by $u_\cO$ a representative of one of them (see \ref{LuMap}). The group $A_\mf G(u_\cO)\cong D_8$ is generated by two elements $s,t$ of order $2$ such that $st$ is of order $4$, and the conjugacy classes of $A_\mf G(u_\cO)$ are
\[\{1\},\;\{s,tst\},\;\{t,sts\},\;\{st,ts\},\;\{stst\}.\]
Let $1$ be the trivial character of $A_\mf G(u_\cO)$, $r$ the (irreducible) character of the reflection representation of $A_\mf G(u_\cO)$ (so that $r(1)=2$, $r(s)=r(t)=r(st)=0$ and $r(stst)=-2$), and let $\varepsilon_s, \varepsilon_t, \varepsilon$ be the linear characters of $A_\mf G(u_\cO)$ satisfying
\[\varepsilon_s(t)=\varepsilon_t(s)=+1,\quad\varepsilon_s(s)=\varepsilon_t(t)=\varepsilon(s)=\varepsilon(t)=-1.\]
The pairs $(u_\cO,1), (u_\cO,\varepsilon_s), (u_\cO,\varepsilon_t)\in\cN_\mf G$ are in the image of the ordinary Springer correspondence, the pair $(u_\cO,r)$ corresponds to the triple $(\dt D_4,\cO_0,\cE_0)\in\cM_\mf G$, and the pair $(u_\cO,\varepsilon)$ is cuspidal. The values $m((u_\cO)_a,w)$ for $a\in A_\mf G(u_\cO)$ and $w\looparrowright\cO$ are given in \Cref{TblD8}.
\begin{table}[htbp]
\caption{$m(u,w)/q^{20}$ for $u\in\cO$ and $w\looparrowright\cO$}
\label{TblD8}
\centering
\begin{tabular}{ll}
\toprule
$a$ & $m((u_\cO)_a,w)/q^{20}$ \\
\midrule[0.08em]
$1$ & $1+\delta_{(u_\cO,\varepsilon_s)}+\delta_{(u_\cO,\varepsilon_t)}+4\zeta_{(\dt D_4,\cO_0,\cE_0)}\gamma_{(u_\cO,r)}+\xi_{(u_\cO,\varepsilon)}$ \\
$s$ & $1-\delta_{(u_\cO,\varepsilon_s)}+\delta_{(u_\cO,\varepsilon_t)}\phantom{+4\zeta_{(\dt D_4,\cO_0,\cE_0)}\gamma_{(u_\cO,r)}}\;\,-\xi_{(u_\cO,\varepsilon)}$ \\
$t$ & $1+\delta_{(u_\cO,\varepsilon_s)}-\delta_{(u_\cO,\varepsilon_t)}\phantom{+4\zeta_{(\dt D_4,\cO_0,\cE_0)}\gamma_{(u_\cO,r)}}\;\,-\xi_{(u_\cO,\varepsilon)}$ \\
$st$ & $1-\delta_{(u_\cO,\varepsilon_s)}-\delta_{(u_\cO,\varepsilon_t)}\phantom{+4\zeta_{(\dt D_4,\cO_0,\cE_0)}\gamma_{(u_\cO,r)}}\;\,+\xi_{(u_\cO,\varepsilon)}$ \\
$stst$ & $1+\delta_{(u_\cO,\varepsilon_s)}+\delta_{(u_\cO,\varepsilon_t)}-4\zeta_{(\dt D_4,\cO_0,\cE_0)}\gamma_{(u_\cO,r)}+\xi_{(u_\cO,\varepsilon)}$ \\
\bottomrule
\end{tabular}
\end{table}
Since each value in \Cref{TblD8} is a non-negative number, the roots of unity $\xi_{(u_\cO,\varepsilon)}$ and $\zeta_{(\dt D_4,\cO_0,\cE_0)}\gamma_{(u_\cO,r)}$ are in particular real numbers, so they must be $+1$ or $-1$. Hence, either the entry in the first or in the last row of the table is equal to $1+\delta_{(u_\cO,\varepsilon_s)}+\delta_{(u_\cO,\varepsilon_t)}-4+\xi_{(u_\cO,\varepsilon)}\geqslant0$, so we must have $\delta_{(u_\cO,\varepsilon_s)}=\delta_{(u_\cO,\varepsilon_t)}=\xi_{(u_\cO,\varepsilon)}=+1$.

We then see that exactly one of the rows of the table (either the first or the last) is strictly positive, that is, there exists a unique $u_0\in\cO^F$ up to $\mf G^F$-conjugacy which satisfies condition (b) in \Cref{Defgood}: If $\zeta_{(\dt D_4,\cO_0,\cE_0)}\gamma_{(u_\cO,r)}=+1$, we set $u_0:=u_\cO$; if $\zeta_{(\dt D_4,\cO_0,\cE_0)}\gamma_{(u_\cO,r)}=-1$, we put $u_0:=(u_\cO)_{stst}\in\cO^F$. The fact that $F$ still acts trivially on $A_\mf G(u_0)$ in the latter case can be deduced by computing the centraliser order $|C_{\mf G^F}(u_0)|$ as explained in \Cref{Criteriongood} (and then consulting \cite[Lm.~44]{MizE7E8}). So we have shown that there exists a unique good $\mf G^F$-conjugacy class in $\cO^F$ (so \Cref{Thmgood} holds for $\cO$) and for an element $u_0$ of this class, we have
\[\delta_{(u_0,1)}=\delta_{(u_0,\varepsilon_s)}=\delta_{(u_0,\varepsilon_t)}=\xi_{(u_0,\varepsilon)}=\zeta_{(\dt D_4,\cO_0,\cE_0)}\gamma_{(u_0,r)}=+1.\]
\end{0}

\begin{0}{\bf The classes $\cO$ with $A_\mf G(u_\cO)\cong\mathfrak S_3$.}\label{p2S3}
Let $\cO\subseteq\mf G$ be a unipotent class for which $A_\mf G(u_\cO)\cong\mathfrak S_3$. Then $\cO$ is one of the following classes (see \cite{MizE7E8}):
\[\dt A_8,\; \dt D_8(a_3),\; \dt A_5+\dt A_2,\; \dt D_4(a_1)+\dt A_1,\; \dt D_4(a_1).\]
The argumentation with respect to the above classes is entirely analogous to the one in \ref{p5S3}: If $\cO$ is one of the above classes, there is a unique element $u_0\in\cO^F$ up to $\mf G^F$-conjugacy for which $F$ acts trivially on $A_\mf G(u_0)$, and the three pairs $(u_0,1), (u_0,\varepsilon), (u_0,r)\in\cN_\mf G$ are all in the image of the ordinary Springer correspondence.

(a) If $\cO\neq\dt D_8(a_3)$, we get
\[\delta_{(u_0,1)}=\delta_{(u_0,\varepsilon)}=\delta_{(u_0,r)}=+1,\]
$u_0$ is good, and the $\mf G^F$-conjugacy of $u_0$ is the unique good one contained in $\cO^F$, so \Cref{Thmgood} holds for $\cO$.

(b) If $\cO=\dt D_8(a_3)$, the evaluation of $m((u_0)_a,w)$ for $a\in A_\mf G(u_0)$ and $w\looparrowright\dt D_8(a_3)$ gives
\[0\leqslant m((u_0)_a,w)=q^{28}\cdot(1+\delta_{(u_0,\varepsilon)}\cdot\varepsilon(a)),\]
so as in \ref{p5S3}, we do not obtain any information on the signs $\delta_{(u_0,\varepsilon)}$ and $\delta_{(u_0,r)}$. It is shown in \cite[\S 9]{GcompGreen} that we have $\delta_{(u_0,\varepsilon)}\equiv q\mod 3$. See \cite{Lueb} for the sign $\delta_{(u_0,r)}$. So we have $m(u_0,w)=0$ if $q\equiv-1\mod 3$, which means that there do not exist any good elements in this case; on the other hand, if $q\equiv1\mod 3$, we have $m(u_0,w)=2q^{28}>0$, so $u_0$ is good (and the $\mf G^F$-conjugacy class of $u_0$ is the unique good one contained in $\dt D_8(a_3)^F$). This explains the condition for the class $\dt D_8(a_3)$ in \Cref{Thmgood}(c).
\end{0}

\begin{0}{\bf The class $\cO=\dt E_7(a_2)+\dt A_1$ with $A_\mf G(u_\cO)\cong\mathfrak S_3\times\Z/2\Z$.}\label{p2S3Z2}
Let $\cO=\dt E_7(a_2)+\dt A_1\subseteq\mf G$ be the (unique) unipotent class for which $A_\mf G(u_\cO)\cong\mathfrak S_3\times\Z/2\Z$. By \cite[Lm.~46]{MizE7E8}, there exist two $\mf G^F$-classes of elements $u\in\cO^F$ for which $F$ acts trivially on $A_\mf G(u)$. Let us for now denote by $u_\cO$ a representative of one of them (see \ref{LuMap}). As usual, we denote the irreducible characters of $\mathfrak S_3$ by $1,\varepsilon,r$ (unit, sign, reflection) and the non-trivial linear character of $\Z/2\Z$ by $\varepsilon$, so that
\[\Irr(A_\mf G(u_\cO))=\{1\boxtimes1,1\boxtimes\varepsilon,\varepsilon\boxtimes1,\varepsilon\boxtimes\varepsilon,r\boxtimes1,r\boxtimes\varepsilon\}.\]
The pairs $(u_\cO,1\boxtimes1), (u_\cO,\varepsilon\boxtimes1), (u_\cO,r\boxtimes1)\in\cN_\mf G$ are in the image of the ordinary Springer correspondence, the pairs $(u_\cO,1\boxtimes\varepsilon)$ and $(u_\cO,r\boxtimes\varepsilon)$ correspond to the triple $(\dt D_4,\cO_0,\cE_0)\in\cM_\mf G$, and the pair $(u_\cO,\varepsilon\boxtimes\varepsilon)$ is cuspidal. The values of $m((u_\cO)_a,w)$ for $a\in A_\mf G(u_\cO)$ and $w\looparrowright\cO=\dt E_7(a_2)+\dt A_1$ are given in \Cref{TblS3Z2}.
\begin{table}[htbp]
\caption{$m(u,w)/q^{22}$ for $u\in\cO=\dt E_7(a_2)+\dt A_1$ and $w\looparrowright\cO$}
\label{TblS3Z2}
\centering
\begin{tabular}{ll}
\toprule
$a$ & $m((u_\cO)_a,w)/q^{22}$ \\
\midrule[0.08em]
$((111),0)$ & $1+\delta_{(u_\cO,\varepsilon\boxtimes1)}+4\delta_{(u_\cO,r\boxtimes1)}+\zeta_{(\dt D_4,\cO_0,\cE_0)}\gamma_{(u_\cO,1\boxtimes\varepsilon)}+4\zeta_{(\dt D_4,\cO_0,\cE_0)}\gamma_{(u_\cO,r\boxtimes\varepsilon)}+\xi_{(u_\cO,\varepsilon\boxtimes\varepsilon)}$ \\
$((21),0)$ & $1-\delta_{(u_\cO,\varepsilon\boxtimes1)}\phantom{+4\delta_{(u_\cO,r\boxtimes1)}}\;\,+\zeta_{(\dt D_4,\cO_0,\cE_0)}\gamma_{(u_\cO,1\boxtimes\varepsilon)}\phantom{+4\zeta_{(\dt D_4,\cO_0,\cE_0)}\gamma_{(u_\cO,r\boxtimes\varepsilon)}}\;\,-\xi_{(u_\cO,\varepsilon\boxtimes\varepsilon)}$ \\
$((3),0)$ & $1+\delta_{(u_\cO,\varepsilon\boxtimes1)}-2\delta_{(u_\cO,r\boxtimes1)}+\zeta_{(\dt D_4,\cO_0,\cE_0)}\gamma_{(u_\cO,1\boxtimes\varepsilon)}-2\zeta_{(\dt D_4,\cO_0,\cE_0)}\gamma_{(u_\cO,r\boxtimes\varepsilon)}+\xi_{(u_\cO,\varepsilon\boxtimes\varepsilon)}$ \\
$((111),1)$ & $1+\delta_{(u_\cO,\varepsilon\boxtimes1)}+4\delta_{(u_\cO,r\boxtimes1)}-\zeta_{(\dt D_4,\cO_0,\cE_0)}\gamma_{(u_\cO,1\boxtimes\varepsilon)}-4\zeta_{(\dt D_4,\cO_0,\cE_0)}\gamma_{(u_\cO,r\boxtimes\varepsilon)}-\xi_{(u_\cO,\varepsilon\boxtimes\varepsilon)}$ \\
$((21),1)$ & $1-\delta_{(u_\cO,\varepsilon\boxtimes1)}\phantom{+4\delta_{(u_\cO,r\boxtimes1)}}\;\,-\zeta_{(\dt D_4,\cO_0,\cE_0)}\gamma_{(u_\cO,1\boxtimes\varepsilon)}\phantom{+4\zeta_{(\dt D_4,\cO_0,\cE_0)}\gamma_{(u_\cO,r\boxtimes\varepsilon)}}\;\,+\xi_{(u_\cO,\varepsilon\boxtimes\varepsilon)}$ \\
$((3),1)$ & $1+\delta_{(u_\cO,\varepsilon\boxtimes1)}-2\delta_{(u_\cO,r\boxtimes1)}-\zeta_{(\dt D_4,\cO_0,\cE_0)}\gamma_{(u_\cO,1\boxtimes\varepsilon)}+2\zeta_{(\dt D_4,\cO_0,\cE_0)}\gamma_{(u_\cO,r\boxtimes\varepsilon)}-\xi_{(u_\cO,\varepsilon\boxtimes\varepsilon)}$ \\
\bottomrule
\end{tabular}
\end{table}
Since each value in \Cref{TblS3Z2} is a (non-negative) real number, we first of all deduce that the roots of unity $\zeta_{(\dt D_4,\cO_0,\cE_0)}\gamma_{(u_\cO,1\boxtimes\varepsilon)}$, $\zeta_{(\dt D_4,\cO_0,\cE_0)}\gamma_{(u_\cO,r\boxtimes\varepsilon)}$ and $\xi_{(u_\cO,\varepsilon\boxtimes\varepsilon)}$ are real numbers as well, so they can only be $+1$ or $-1$. Taking the sum of the first and fourth row of \Cref{TblS3Z2}, we obtain $2+2\delta_{(u_\cO,\varepsilon\boxtimes1)}+8\delta_{(u_\cO,r\boxtimes1)}\geqslant0$, which forces $\delta_{(u_\cO,r\boxtimes1)}=+1$. In turn, adding the third and the last row of the table gives $2+2\delta_{(u_\cO,\varepsilon\boxtimes1)}-4\geqslant0$, so we must have $\delta_{(u_\cO,\varepsilon\boxtimes1)}=+1$. Now let us assume that $\zeta_{(\dt D_4,\cO_0,\cE_0)}\gamma_{(u_\cO,r\boxtimes\varepsilon)}=-1$. Then the non-negativity of the entry in the last row of \Cref{TblS3Z2} shows that we also have $\zeta_{(\dt D_4,\cO_0,\cE_0)}\gamma_{(u_\cO,1\boxtimes\varepsilon)}=\xi_{(u_\cO,\varepsilon\boxtimes\varepsilon)}=-1$, and the fourth row of the table is the unique one which is strictly positive (and is equal to $12$). On the other hand, assuming that $\zeta_{(\dt D_4,\cO_0,\cE_0)}\gamma_{(u_\cO,r\boxtimes\varepsilon)}=+1$, the non-negativity of the third row of \Cref{TblS3Z2} forces $\zeta_{(\dt D_4,\cO_0,\cE_0)}\gamma_{(u_\cO,1\boxtimes\varepsilon)}=\xi_{(u_\cO,\varepsilon\boxtimes\varepsilon)}=+1$, and then the first row of the table is the unique one which is strictly positive (and is equal to $12$). So we see that, regardless of the value of $\zeta_{(\dt D_4,\cO_0,\cE_0)}\gamma_{(u_\cO,r\boxtimes\varepsilon)}$, there exists a unique $\mf G^F$-conjugacy class which satisfies condition (b) in \Cref{Defgood}: If $\zeta_{(\dt D_4,\cO_0,\cE_0)}\gamma_{(u_\cO,r\boxtimes\varepsilon)}=+1$, we set $u_0:=u_\cO$, while in the case where $\zeta_{(\dt D_4,\cO_0,\cE_0)}\gamma_{(u_\cO,r\boxtimes\varepsilon)}=-1$, we take $u_0:=(u_\cO)_{(111),1}$. Using the criterion in \Cref{Criteriongood} (combined with \cite[Lm.~46]{MizE7E8}), we deduce that $u_0$ is good. So the $\mf G^F$-class of $u_0$ is the unique good one contained in $(\dt E_7(a_2)+\dt A_1)^F$, and \Cref{Thmgood} holds for $\dt E_7(a_2)+\dt A_1$. Furthermore, with respect to this choice of $u_0$, we have
\[\delta_{(u_0,1\boxtimes1)}=\delta_{(u_0,\varepsilon\boxtimes1)}=\delta_{(u_0,r\boxtimes1)}=\zeta_{(\dt D_4,\cO_0,\cE_0)}\gamma_{(u_0,1\boxtimes\varepsilon)}=\zeta_{(\dt D_4,\cO_0,\cE_0)}\gamma_{(u_0,r\boxtimes\varepsilon)}=\xi_{(u_0,\varepsilon\boxtimes\varepsilon)}=+1.\]
\end{0}

\begin{0}{\bf The class $\cO=2\dt A_4$ with $A_\mf G(u_\cO)\cong\mathfrak S_5$.}\label{p2S5}
Let $\cO=2\dt A_4\subseteq\mf G$ be the (unique) unipotent class with $A_\mf G(u_\cO)\cong\mathfrak S_5$. As in \ref{p3S5}, the argument in this case can be copied verbatim from \ref{p5S5}. In particular, there exists a unique $u_0=u_\cO\in\cO^F$ up to $\mf G^F$-conjugacy such that $F$ induces the identity map on $A_\mf G(u_0)$, the $\mf G^F$-class of $u_0$ is the unique good one contained in $(2\dt A_4)^F$ (so \Cref{Thmgood} holds for the class $2\dt A_4$), and (with the notation analogous to the one in \ref{p5S5}) we have
\[\delta_{\phi_{1680,22}}=\delta_{\phi_{4480,16}}=\delta_{\phi_{4536,18}}=\delta_{\phi_{1400,20}}=\delta_{\phi_{5670,18}}=\delta_{\phi_{70,32}}=\xi=+1.\]
\end{0}

\begin{0}{\bf Summary of Sections \ref{p5}--\ref{p2}.}\label{Summp5p2}
Our case-by-case analysis shows that \Cref{Thmgood} holds. Moreover, for any unipotent class $\cO\subseteq\mf G$ not excluded in \Cref{Thmgood}, any good element $u_0\in\cO^F$ and any $\varsigma\in\Irr(A_\mf G(u_0))$, the root of unity $\delta_{(u_0,\varsigma)}$, $\zeta_{\tau((u_0,\varsigma))}\gamma_{(u_0,\varsigma)}$ or $\xi_{(u_0,\varsigma)}$ (respectively) is always equal to $+1$, which proves \Cref{Roots1}(a). Finally, we note that if $\cO\subseteq\mf G$ is a unipotent class which admits a pair $(u_\cO,\varsigma)$ as in either \ref{Intermed} or \ref{CuspCS}, $\cO$ is not among the classes excluded in \Cref{Thmgood}, so \Cref{Roots1}(b) holds as well.

Our arguments also show that the classes excluded in the formulation of \Cref{Thmgood} need to be excluded, that is, for any of these \enquote{exceptional} classes $\cO\subseteq\mf G$, $\cO^F$ either does not contain any good elements or it contains more than one good $\mf G^F$-class.

As noted after \Cref{Roots1}, this completes the determination of the values of unipotent characters at unipotent elements for the groups $\dt E_8(q)$ where $q$ is any power of any prime $p$.
\end{0}

\begin{Rem}
(a) While our argumentation does of course depend upon the classification of the unipotent classes in $\mf G$ and the ones in $\mf G^F$, it does not rely on the specific representatives given in \cite{MizE7E8}, but merely on the structure of the groups $A_\mf G(u)$ for $u\in\mf G^F_{\mathrm{uni}}$. So even if some of the elements in Mizuno's list might be incorrect, this does not affect our arguments. Note that the classification of the unipotent classes in \cite{MizE7E8} (with some minor corrections as stated in \cite{Sp}) has been established independently by Liebeck--Seitz \cite{LieSei}.

(b) Now that we know all the values of the $R_x|_{\mf G^F_{\mathrm{uni}}}$ for $x\in X(\mf W)$, we can use formula \eqref{HeckeFormula} in \ref{Ree} to compute the numbers $|\mf B_0^Fw\mf B_0^F\cap O_u|$ for any $w\in\mf W$ and any $u\in\mf G^F_{\mathrm{uni}}$.

(c) Recall that our notion $w\looparrowright\cO$ (see \ref{LuMap}) includes the requirement that $w$ is of minimal length among the elements of the union of all conjugacy classes of $\mf W$ which are sent to the unipotent class $\cO\subseteq\mf G$ under Lusztig's map. By the definition of this map, one may expect that it should be enough to take any conjugacy class $\Xi\subseteq\mf W$ which is sent to $\cO$, and $w\in\Xi$ to be a representative of minimal length among the elements of $\Xi$. However, our argumentation would then no longer work in all cases. For example, considering the unipotent class $\cO=\dt D_7(a_1)$ in characteristic $p\geqslant5$ (see \ref{p5Z2}), both the classes $\dt D_7(a_1)$ and $\dt D_8(a_2)$ of $\mf W$ (notation of {\sffamily CHEVIE}) are mapped to $\cO$ under Lusztig's map (while if $p=2$, we have $\dt D_8(a_2)\mapsto\dt D_7(a_1)$ and $\dt D_7(a_1)\mapsto(\dt D_7(a_1))_2$; the class $(\dt D_7(a_1))_2$ only exists when $p=2$), but if we take $w$ of minimal length in $\dt D_8(a_2)$ (which is larger than the minimal length of elements in $\dt D_7(a_1)$), the computation in \ref{p5Z2} would read
\[0\leqslant m((u_\cO),w)=q^{26}\delta_{(u_\cO,1)}-q^{25}\delta_{(u_\cO,-1)}\varepsilon(a)\]
for $a\in A_\mf G(u_\cO)$, so here we could not pinpoint the sign $\delta_{(u_\cO,-1)}$. It is interesting that in the cases where the preimage of a given unipotent class $\cO$ under Lusztig's map is independent of $p$ (which is true for most $\cO$), our argument works regardless of which class $\Xi\subseteq\mf W$ with $\cO=\cO_\Xi$ we take in any given characteristic $p$, so we do not have to assume that $w$ is of minimal length among all conjugacy classes of $\mf W$ which are mapped to $\cO$. We only have to be careful with the unipotent classes $\dt D_7(a_1), \dt D_5+\dt A_2, \dt D_4+\dt A_2, \dt A_3+\dt A_2$.
\end{Rem}

\section{Representatives for the good unipotent \texorpdfstring{$\mf G^F$}{GF}-classes}\label{Repgood}

Considering the explicit list of representatives in \cite{MizE7E8}, it may be desirable to be able to recognise a good element (in the sense of \Cref{Defgood}) among this list. In this section, we show how this can be achieved for almost all unipotent classes of $\mf G$.

\begin{0}\label{CGFu0}
Let $\cO\subseteq\mf G$ be a unipotent class which is not excluded in \Cref{Thmgood}, and let $u_0\in\cO^F$ be a representative for the unique good $\mf G^F$-class contained in $\cO^F$. Now that we have determined the values $m(u_0,w)$ for $w\looparrowright\cO$, we can compute the centraliser order $|C_{\mf G^F}(u_0)|$ as explained in \Cref{Criteriongood}. So if the $\mf G^F$-class of $u_0$ happens to be the unique one inside $\cO^F$ with centraliser order $|C_{\mf G^F}(u_0)|$, we can identify an explicit representative from the list in \cite{MizE7E8}. (This holds in particular if $u_0\in\cO^F$ is unique up to $\mf G^F$-conjugacy with the property that $F$ acts trivially on $A_\mf G(u_0)$.)
\end{0}

In the cases where \ref{CGFu0} does not apply, one can try to perform direct computations to check condition (b) in \Cref{Defgood} for a given concrete representative $u_0=z_n$ in Mizuno's list \cite{MizE7E8}. The following lemma is very useful for this purpose.

\begin{Lm}[see {\cite[Lemma~3.2.24]{HDiss}}; cf.\ {\cite[\S 2]{LuWeylUni}}]\label{LmBwB}
Let $w\in\mf W$ be an element of length $e\in\N_0$ in $(\mf W,S)$, and let $w=s_{i_1}\cdot s_{i_2}\cdot\ldots\cdot s_{i_e}$ be a reduced expression for $w$ (where $s_{i_j}\in S$ is the reflection with respect to the simple root $\alpha_{i_j}$, $1\leqslant i_j\leqslant8$). Let $t_1, t_2,\ldots, t_e\in\Fp^\times$, and let
\[u_0:=u_{\alpha_{i_1}}(t_1)\cdot u_{\alpha_{i_2}}(t_2)\cdot\ldots\cdot u_{\alpha_{i_e}}(t_e)\in\mf U_0^F.\]
Then, denoting by $\dot w_0\in N_\mf G(\mf T_0)^F$ a representative of the longest element $w_0$ of $(\mf W,S)$, we have
\[\dot w_0 u_0\dot w_0^{-1}\in\mf B_0^Fw\mf B_0^F.\]
\end{Lm}

\begin{proof}
This is a basic computation in terms of root systems and groups with a $BN$-pair, for which we refer to \cite[3.2.24]{HDiss}.
\end{proof}

We now explain how \Cref{LmBwB} can be applied to show that a given representative $z_n$ in \cite{MizE7E8} is good. We use the following notation until the end of this paper: If $\alpha\in\Phi^+$ is written $\alpha=\sum_{i=1}^8c_i\alpha_i$ with $c_i\in\N_0$, we set $u_{1^{c_1}2^{c_2}\ldots8^{c_8}}:=u_\alpha(1)$, and in addition we omit $i$ in the case where $c_i=0$. So for example, we write $u_{234^25}$ instead of $u_{\alpha_2+\alpha_3+2\alpha_4+\alpha_5}(1)$.

\begin{0}\label{Rel}
In most cases, we first need to manipulate the $z_n$ in Mizuno's list in order to obtain an expression of the form $u_{i_1}u_{i_2}\cdots u_{i_m}$ where $1\leqslant i_1,\ldots, i_m\leqslant8$ (so that we can apply \Cref{LmBwB}). To this end, we use the relations (R) in \cite[Chapter~3~(p.~23)]{StChev}. However, note that these relations depend upon the structure constants with respect to a chosen Chevalley basis of the Lie algebra underlying $\mf G$. Since we take Mizuno's representatives, we have to rely on his choices, but unfortunately, some of the structure constants in \cite[p.~460]{MizE7E8} are incorrect. At least if $p=2$, this does not play any role since $u_\alpha(1)^{-1}=u_\alpha(-1)=u_\alpha(1)$ for any $\alpha\in\Phi$. To illustrate the main ideas of our computations, we assume that $p=2$ for now (see \ref{Comppneq2} below for the case $p\neq2$). For $1\leqslant j\leqslant8$, we set
\[\omega_j:=u_{\alpha_j}(1)u_{-\alpha_j}(1)u_{\alpha_j}(1)\in\mf G^F.\]
Using the fact that $1=-1$ in $k$ and that the root system of type $\dt E_8$ is simply laced, we have the following identities in $\mf G^F$, where $\alpha,\beta\in\Phi^+$ and $1\leqslant j\leqslant8$ (see \cite[p.~23]{StChev}):
\begin{enumerate}[label=(\roman*)]
\item $u_\alpha(1)^2=1$.
\item $u_\alpha(1)\cdot u_\beta(1)\cdot u_\alpha(1)\cdot u_\beta(1)=\begin{cases}
u_{\alpha+\beta}(1) &\text{if } \alpha+\beta\in\Phi, \\
\hfil1 &\text{if } \alpha+\beta\notin\Phi. \end{cases}$
\item $\omega_ju_\alpha(1)\omega_j^{-1}=u_{s_j(\alpha)}(1)$.
\end{enumerate}
\end{0}

\begin{0}\label{ConjMeth}
As in \ref{Rel}, we assume that $p=2$. The simplest case occurs for the regular unipotent class $\dt E_8\subseteq\mf G$: Let us consider the element
\[z_1=u_1u_3u_4u_2u_5u_6u_7u_8\in\dt E_8^F\]
(see \cite[Lm.~37]{MizE7E8}). We form the corresponding product of simple reflections in $\mf W$, that is, we set $w:=s_1s_3s_4s_2s_5s_6s_7s_8\in\mf W$ (a Coxeter element of $\mf W$). Since this is a reduced expression for $w$, \Cref{LmBwB} shows that we have $\mf B_0^Fw\mf B_0^F\cap O_{z_1}\neq\varnothing$. Furthermore, $w$ is a Coxeter element of $\mf W$, so we have $w\looparrowright\dt E_8$ (see \ref{LuMap}) and $u_0\in\dt E_8^F$ is therefore good.

Considering the class $\dt E_8(a_1)\subseteq\mf G$, our candidate for the good element in Mizuno's list \cite[Lm.~38]{MizE7E8} is
\[z_{11}=u_1u_2u_{24}u_{34}u_5u_6u_7u_8\in\dt E_8(a_1)^F.\]
Using the relations in \Cref{Rel}, we get
\[u_{24}=u_2u_4u_2u_4\quad\text{and}\quad u_{34}=u_4u_3u_4u_3,\]
and then
\[z_{11}=u_1u_2^2u_4u_2u_4^2u_3u_4u_3u_5u_6u_7u_8=u_1u_4u_2u_3u_4u_3u_5u_6u_7u_8.\]
Setting $w:=s_1s_4s_2s_3s_4s_3s_5s_6s_7s_8\in\mf W$, $w$ is of length $10$, and we use {\sffamily CHEVIE} \cite[\S 6]{MiChv} to verify that $w\looparrowright\dt E_8(a_1)$, so \Cref{LmBwB} shows that $z_{11}$ is good. The argument for $z_{17}\in\dt E_8(a_2)^F$ (see \cite[Lm.~39]{MizE7E8}) is similar: We have
\[z_{17}=u_1u_3u_2u_{24}u_{45}u_{56}u_{67}u_8=u_1u_3u_4u_2u_5u_4u_6u_5u_7u_6u_7u_8,\]
$w:=s_1s_3s_4s_2s_5s_4s_6s_5s_7s_6s_7s_8$ is of length $12$, and $w\looparrowright\dt E_8(a_2)$.

In the other cases, we first need to conjugate a given Mizuno representative $z_n\in\cO^F$ by a suitable element of $\mf G^F$ to obtain a desired element $u_0$ with which we can apply \Cref{LmBwB}. The main idea to achieve this is to repeatedly use relation (iii) in \Cref{Rel}. More precisely, let us assume that $z_n$ is written as $z_n=u_{\beta_1}(1)\cdot u_{\beta_2}(1)\cdots u_{\beta_r}(1)$ in \cite{MizE7E8}, with $\beta_i\in\Phi^+$ for $1\leqslant i\leqslant r$. (For each unipotent class $\cO\subseteq\mf G$, there exists a unique representative $z_n\in\cO^F$ in \cite{MizE7E8} which is of this form, and this will always be our candidate for the good element in $\cO^F$.) We now want to find a list $[i_s, i_{s-1},\ldots, i_1]$ of numbers $i_j\in\{1,\ldots,8\}$ with the property
\[\omega_{i_j}\cdots\omega_{i_1}\cdot z_n\cdot\omega_{i_1}^{-1}\cdots\omega_{i_j}^{-1}=u_{\beta_1^{(j)}}(1)\cdot u_{\beta_2^{(j)}}(1)\cdots u_{\beta_r^{(j)}}(1)\quad\text{for }j\in\{1,\ldots,s\}\]
where $\beta_i^{(j)}\in\Phi^+$ for all $1\leqslant i\leqslant r$, $1\leqslant j\leqslant s$, in order to arrive at a suitable product $u_{i_1}u_{i_2}\cdots u_{i_m}$ with $1\leqslant i_1,\ldots, i_m\leqslant8$ (after possibly using the relations (i) and (ii) in \Cref{Rel}) and with $m\geqslant1$ as small as possible. Then we form the corresponding product $w:=s_{i_1}s_{i_2}\cdots s_{i_m}\in\mf W$ and check whether this is a reduced expression and $w\looparrowright\cO$ (after which \Cref{LmBwB} would show that $z_n$ is good). This can typically be achieved by means of minimising the sum of the heights $\sum_{i=1}^r\hgt(\beta_i^{(s)})$ (with respect to the simple roots $\alpha_1,\ldots,\alpha_8$) in the above expression.
\end{0}

\begin{Ex}
As an example for the type of computation outlined in \ref{ConjMeth}, let us consider the class $\dt E_7\subseteq\mf G$ and the representative
\[z_{30}=u_1u_{234}u_{345}u_{245}u_6u_7u_8\in\dt E_7^F\]
(see \cite[Lm.~41]{MizE7E8}). The sum of the heights of the roots which appear as indices is $13$, which we want to reduce by means of conjugating $z_{30}$ by suitable $\omega_j$. We illustrate this via the following picture, where the $\omega_j$ on the left is the element that we conjugate with, and where the number in parentheses is the sum of the heights of the roots in the indices:
\begin{align*}
 & u_1u_{234}u_{345}u_{245}u_{6}u_{7}u_{8}\quad(13) \\
\omega_2\quad\leadsto\quad & u_{1}u_{34}u_{2345}u_{45}u_{6}u_{7}u_{8}\quad(12) \\
\omega_3\quad\leadsto\quad & u_{13}u_{4}u_{245}u_{345}u_{6}u_{7}u_{8}\quad(12) \\
\omega_5\quad\leadsto\quad & u_{13}u_{45}u_{24}u_{34}u_{56}u_{7}u_{8}\quad(12) \\
\omega_4\quad\leadsto\quad & u_{134}u_{5}u_{2}u_{3}u_{456}u_{7}u_{8}\quad(11) \\
\omega_1\quad\leadsto\quad & u_{34}u_{5}u_{2}u_{13}u_{456}u_{7}u_{8}\quad(11) \\
\omega_3\quad\leadsto\quad & u_{4}u_{5}u_{2}u_{1}u_{3456}u_{7}u_{8}\quad(10)
\end{align*}
So we managed to decrease the sum of the heights of the roots on the right, but how to continue? Since we require all of the indices of the element on the right to remain positive roots (and since we do not make any progress by conjugating with $\omega_3$ again), the only choice is to conjugate with $\omega_6$.~---~However, note that this increases the sum of the heights on the right! One possibility to eventually arrive at the desired product of the form $u_{i_1}\cdots u_{i_7}$ on the right is to continue the above computation by setting
\[g:=\omega_8\omega_7\omega_6\omega_5\omega_4\omega_2\omega_3\omega_4\omega_1\omega_3
\omega_5\omega_4\omega_2\omega_6\omega_5\omega_4\omega_7\omega_8\omega_6\omega_5\omega_7
\omega_6\omega_3\omega_1\omega_4\omega_5\omega_3\omega_2\in\mf G^F,\]
and we then have
\[gz_{30}g^{-1}=u_6u_7u_5u_2u_4u_3u_1.\]
The corresponding element $w:=s_6s_7s_5s_2s_4s_3s_1\in\mf W$ is of length $7$ and $w\looparrowright\dt E_7$, so \Cref{LmBwB} shows that $z_{30}$ is good. But our chosen $g$ is certainly by far not the only or a \enquote{canonical} one which will suit our demands, and we could only find it by ad hoc and partly non-systematic computations. The situation is even worse for representatives $z_n$ in \cite{MizE7E8} which are the product of eight $u_\beta(1)$.
\end{Ex}

\begin{0}\label{DescrTblMiz}
We still assume that $p=2$. \Cref{TblMizgood} lists the representatives $z_n$ in a given unipotent class $\cO$ which we claim to be good. (To save space, we do not list the classes $\cO$ with $A_\mf G(u_\cO)=\{1\}$ since in this case $\cO^F$ is a single $\mf G^F$-class and the representative $z_n\in\cO^F$ is necessarily good.) The only unipotent class in characteristic $2$ which does not contain any good elements is $\dt D_8(a_3)$ when $q\equiv-1\mod 3$; any other unipotent class contains a unique good $\mf G^F$-class (see \ref{p2S3}(b), \Cref{Thmgood}). Unfortunately, we could not show that the representatives $z_{39}\in(\dt E_7(a_1)+\dt A_1)^F$, $z_{44}\in\dt D_8(a_1)^F$ and $z_{50}\in(\dt E_7(a_2)+\dt A_1)^F$ are good (since \ref{CGFu0} is not applicable and since we did not manage to successfully perform the argument in \ref{ConjMeth}). For any other unipotent class $\cO$ with $A_\mf G(u_\cO)\neq\{1\}$, the entry in the column of \Cref{TblMizgood} named \enquote{Remark} indicates the proof that the listed $z_n$ is good, with the following conventions/notation: We refer to \ref{CGFu0} whenever it is applicable. Otherwise, an entry $g\leadsto u_{i_1}u_{i_2}\cdots u_{i_m}$ means that $g\in\mf G^F$ satisfies $gz_ng^{-1}=u_{i_1}u_{i_2}\cdots u_{i_m}$ (with $1\leqslant i_1,\ldots, i_m\leqslant8$), that $w:=s_{i_1}s_{i_2}\cdots s_{i_m}\in\mf W$ is a reduced expression and $w\looparrowright\cO$ (so that $z_n$ is good by \Cref{LmBwB}); the expression $[j_sj_{s-1}\ldots j_1]$ with $1\leqslant j_i\leqslant8$ stands for the product $\omega_{j_s}\omega_{j_{s-1}}\cdots\omega_{j_1}\in\mf G^F$.
\end{0}

\begin{0}\label{Comppneq2}
We now assume that $p\neq2$ and show how one can see that a given representative $z_n$ in Mizuno's list is good. As noted in \ref{Rel}, the computation of the kind in \ref{ConjMeth} becomes more subtle since we have to take the structure constants with respect to a chosen Chevalley basis of the Lie algebra underlying $\mf G$ into account and since some of these constants given in \cite[p.~460]{MizE7E8} are incorrect. The good news is that there are fewer classes where we actually need to do a computation as in \ref{ConjMeth} now. First of all, recall from \ref{p5Z2}(b) and \ref{p3Z2} that any $F$-stable element of the classes $\dt E_7(a_1)+\dt A_1$ and $\dt D_6(a_1)+\dt A_1$ is good, so we do not have to consider them here (in particular, $z_{39}\in(\dt E_7(a_1)+\dt A_1)^F$ and $z_{106}\in(\dt D_6(a_1)+\dt A_1)^F$ are good). Furthermore, the argument in \ref{CGFu0} applies to more classes when $p\geqslant3$. The unipotent classes $\cO\subseteq\mf G$ with $A_\mf G(u_\cO)\neq\{1\}$ which are not covered by this (so that a concrete computation as in \ref{ConjMeth} is required) are the following:
\begin{align*}
\dt E_8\text{ (only when }p=3,5),\;\dt E_8(a_1) \;(p=3),\;\dt E_7+\dt A_1,\;\dt E_7\;(p=3),\;\dt D_8,\;\dt D_8(a_1), \\
\dt E_6+\dt A_1\;(p=3),\; \dt D_8(a_3),\;\dt D_6+\dt A_1,\;\dt E_6\;(p=3),\;\dt A_5+2\dt A_1,\; (\dt A_5+\dt A_1)''.
\end{align*}
Regarding the regular unipotent class $\dt E_8$, note that the argument in \ref{ConjMeth} does not use the relations (i)--(iii) in \ref{Rel}, so it is valid for $p\geqslant3$ as well.

Let us consider the class $\dt E_8(a_1)$ and the representative $z_{11}\in\dt E_8(a_1)^F$ defined as in \ref{ConjMeth}. With the structure constants of \cite{MizE7E8}, we have
\[u_{24}=u_2^{-1}u_4^{-1}u_2u_4\quad\text{and}\quad u_{34}^{-1}=u_4^{-1}u_3^{-1}u_4u_3,\]
and we see that the computation analogous to the one in \ref{ConjMeth} does not quite yield the desired result. To fix this, we first conjugate $z_{11}$ with a suitable element of $\mf T_0^F$: For $\alpha\in\Phi$, let $\alpha^\vee\colon\kunits\rightarrow\mf T_0$ be the corresponding co-root. Then we have $t:=\alpha_1^\vee(-1)\in\mf T_0^F$ and
\[tz_{11}t^{-1}=u_1u_2u_{24}u_{34}(-1)u_5u_6u_7u_8=u_1u_4(-1)u_2u_3(-1)u_4u_3u_5u_6u_7u_8,\]
so we can apply \Cref{LmBwB} with $w:=s_1s_4s_2s_3s_4s_3s_5s_6s_7s_8\looparrowright\dt E_8(a_1)$ (as in \ref{ConjMeth}) to deduce that $z_{11}$ is good.

As for the class $\dt E_7+\dt A_1$, setting $t:=\alpha_1^\vee(-1)\alpha_2^\vee(-1)\alpha_3^\vee(-1)\alpha_4^\vee(-1)\alpha_5^\vee(-1)\in\mf T_0^F$ gives
\[(u_3^{-1}t)\cdot z_{21}\cdot(u_3^{-1}t)^{-1}=u_1u_3(-1)u_1(-1)u_2u_4u_2(-1)u_3u_5u_4(-1)u_6u_5(-1)u_6(-1)u_7u_8,\]
so we can apply \Cref{LmBwB} with $w:=s_1s_3s_1s_2s_4s_2s_3s_5s_4s_6s_5s_6s_7s_8\looparrowright\dt E_7+\dt A_1$ (which we also used in the case $p=2$) to deduce that $z_{21}$ is good. Similarly, regarding the class $\dt D_8$, we set $t:=\alpha_1^\vee(-1)\alpha_3^\vee(-1)\alpha_4^\vee(-1)\alpha_7^\vee(-1)\alpha_8^\vee(-1)\in\mf T_0^F$ to obtain
\[(u_3^{-1}t)\cdot z_{36}\cdot(u_3^{-1}t)^{-1}=u_1u_3^{-1}u_1^{-1}u_2u_4u_2^{-1}u_3u_5u_4^{-1}u_6u_5^{-1}  u_7u_6^{-1}u_8u_7^{-1}u_8^{-1},\]
and $w:=s_1s_3s_1s_2s_4s_2s_3s_5s_4s_6s_5s_7s_6s_8s_7s_8\looparrowright\dt D_8$ is again the same element of $\mf W$ that we used when $p=2$, so $z_{36}$ is good.

Now the classes $\dt E_7$, $\dt E_6+\dt A_1$ and $\dt E_6$ only need to be considered when $p=3$, and in this case we have seen in \ref{p3Z3} that a good representative in one of these classes is characterised by the property that it is $\mf G^F$-conjugate to its inverse. So it suffices to prove that the elements $z_{30}\in\dt E_7^F$, $z_{65}\in(\dt E_6+\dt A_1)^F$ and $z_{73}\in\dt E_6^F$ have this property. This can be achieved with the method described in \cite[Lm.~4.1.13]{HDiss}: Setting
\[t:=\alpha_3^\vee(-1)\alpha_7^\vee(-1)\alpha_8^\vee(-1)\in\mf T_0^F\quad\text{and}\quad v:=u_8u_7u_8u_6u_7u_8u_1^{-1}\in\mf U_0^F\]
gives $(tv)z_{30}(tv)^{-1}=z_{30}^{-1}$, and setting
\[t:=\alpha_4^\vee(-1)\alpha_5^\vee(-1)\alpha_7^\vee(-1)\alpha_8^\vee(-1)\in\mf T_0^F\quad\text{and}\quad v:=u_8u_7u_8u_{1345}^{-1}u_{1234}^{-1}\in\mf U_0^F\]
gives both $(tv)z_{65}(tv)^{-1}=z_{65}^{-1}$ and also $(tv)z_{73}(tv)^{-1}=z_{73}^{-1}$.

Let us consider the class $\dt D_6+\dt A_1$ and the element $z_{85}\in(\dt D_6+\dt A_1)^F$. We want to use a similar argument as in the case where $p=2$. Setting 
\begin{align*}
g:=&\omega_4\omega_5\omega_8\omega_7\omega_3\omega_1\omega_6\omega_5\omega_4\omega_2\omega_3
\omega_4\omega_5\omega_6\omega_5\omega_4\omega_1\omega_3\omega_4\omega_2\omega_5\omega_6
\omega_3\omega_4\omega_2\omega_7\omega_6\cdot \\
&\omega_5\omega_4\omega_3\omega_1\omega_6\omega_5
\omega_4\omega_2\omega_5\omega_4\omega_3\omega_2\omega_4\omega_5\omega_3
\omega_4\omega_1\omega_3\omega_4\omega_2\omega_8\omega_7\omega_5\omega_3\omega_6\omega_7\in\mf G^F,
\end{align*}
we get
\begin{equation}\label{Conjz85}
gz_{85}g^{-1}=u_{6}(\delta_{6})u_{7}(\delta_{7})u_{234}(\delta_{234})u_{345}(\delta_{345})u_{245}(\delta_{245})u_{24}(\delta_{24})u_{1}(\delta_1),
\end{equation}
where the coefficients $\delta_\beta\in\{\pm1\}$ depend on the structure constants in \cite[p.~460]{MizE7E8}. Instead of explicitly computing these coefficients, we show that the products on the right hand side of \eqref{Conjz85} are pairwise conjugate for any possible distribution of the signs $\delta_6, \delta_7, \delta_{234}, \ldots,\delta_1$. It suffices to show that $u_{6}u_{7}u_{234}u_{345}u_{245}u_{24}u_{1}$ is $\mf G^F$-conjugate to $u_{6}(\delta_{6})u_{7}(\delta_{7})u_{234}(\delta_{234})u_{345}(\delta_{345})u_{245}(\delta_{245})u_{24}(\delta_{24})u_{1}(\delta_1)$. To see this, we make the ansatz
\[t\cdot u_{6}u_{7}u_{234}u_{345}u_{245}u_{24}u_{1}\cdot t^{-1}\overset{!}{=}u_{6}(\delta_{6})u_{7}(\delta_{7})u_{234}(\delta_{234})u_{345}(\delta_{345})u_{245}(\delta_{245})u_{24}(\delta_{24})u_{1}(\delta_1),\]
where $t=\alpha_1^\vee(\xi_1)\cdot\alpha_2^\vee(\xi_2)\cdots\alpha_8^\vee(\xi_8)\in\mf T_0^F$, with $\xi_i\in\{\pm1\}$. This leads to a system of equations in the unknowns $\xi_1,\ldots,\xi_8$, a solution to which is given by
\begin{align*}
\xi_1=\delta_{245}\delta_{345},\;&\xi_2=\delta_1\delta_6\delta_{245}\delta_{234}\delta_{345},\;\xi_3=\delta_1,\;\xi_4=\delta_{24}\delta_{245}\delta_{234}\delta_{345}, \\
&\xi_5=\delta_6,\;\xi_6=\delta_{234}\delta_{345},\;\xi_7=1,\;\xi_8=\delta_7\delta_{234}\delta_{345}.
\end{align*}
Hence, by conjugating with a suitable torus element as above, we can modify the signs $\delta_{\beta}$ in \eqref{Conjz85} to our likings. We see that we have
\[u_2\cdot(u_{6}u_{7}u_{234}u_{345}u_{245}(-1)u_{24}(-1)u_{1})\cdot u_2^{-1}=u_6u_7u_4^{-1}u_3^{-1}u_4u_2u_5^{-1}u_4^{-1}u_3u_2^{-1}u_4u_5u_1,\]
so we can apply \Cref{LmBwB} with $w:=s_6s_7s_4s_3s_4s_2s_5s_4s_3s_2s_4s_5s_1\looparrowright\dt D_6+\dt A_1$ to deduce that $z_{85}$ is good. With a similar computation, we find that the elements $z_{122}\in(\dt A_5+2\dt A_1)^F$ and $z_{132}\in(\dt A_5+\dt A_1)''^F$ are good.

We did not find an argument for the classes $\dt D_8(a_3)$ (when $p=3$) and $\dt D_8(a_1)$.~---~Note that a computation for the class $\dt D_8(a_3)$ would have to be very subtle since there does not even exist a good element when $p\geqslant5$ and $q\equiv-1\mod 3$ (see \ref{p5S3}(b)). In particular, since the computation described in \Cref{ConjMeth} does not catch such congruence relations, it cannot work here.
\end{0}

\begin{0}{\bf Summary of \Cref{Repgood}.}
The obvious candidate for the good representative of a given unipotent class in \cite{MizE7E8} is always the first one listed in the corresponding lemma (that is, the one which appears in Table~3 in loc.\ cit., but this table contains some typographical errors). If $p=2$, we could verify this in \ref{ConjMeth}--\ref{DescrTblMiz} with the exception of the elements $z_{39}\in(\dt E_7(a_1)+\dt A_1)^F$, $z_{44}\in\dt D_8(a_1)^F$ and $z_{50}\in(\dt E_7(a_2)+\dt A_1)^F$. It is in fact not guaranteed that there even exists a uniform representative in \cite{MizE7E8} which will be good for all $q$: For example, with regard to the split elements, it is noted in \cite[p.~591]{BeSp} that this happens for the class $\dt E_7(a_1)+\dt A_1$ in good characteristic, so something similar may very well apply to some of the above classes with respect to our notion of good elements. As for the case $p>2$, we could not single out a good representative for the classes $\dt D_8(a_3)$ (when $p=3$) and $\dt D_8(a_1)$, but our arguments in \ref{Comppneq2} work for any other class. (In contrast to the case where $p=2$, this includes the classes $\dt E_7(a_1)+\dt A_1$ and $\dt E_7(a_2)+\dt A_1$.)
\end{0}

\begin{landscape}
\begin{longtable}[htbp]{p{2cm} p{10.5cm} p{7.5cm}}
\caption{Good representatives in \cite{MizE7E8} when $p=2$} \label{TblMizgood} \\
\toprule
$\cO\subseteq G$ & Good element $z_n\in\cO^F$ in \cite{MizE7E8} & Remark \\
\midrule[0.08em]
$\dt E_8$ & $z_{1{\phantom1}}=u_1u_3u_4u_2u_5u_6u_7u_8$ & $1\leadsto u_1u_3u_4u_2u_5u_6u_7u_8$ (\ref{LmBwB}) \\
\midrule
$\dt E_8(a_1)$ & $z_{11}=u_1u_2u_{24}u_{34}u_5u_6u_7u_8$ & $1\leadsto u_1u_4u_2u_3u_4u_3u_5u_6u_7u_8$ (\ref{LmBwB}) \\
\midrule
$\dt E_8(a_2)$ & $z_{17}=u_1u_3u_2u_{24}u_{45}u_{56}u_{67}u_8$ & $1\leadsto u_1u_3u_4u_2u_5u_4u_6u_5u_7u_6u_7u_8$ (\ref{LmBwB}) \\
\midrule
$\dt E_7+\dt A_1$ & $z_{21}=u_{13}u_{24}u_{34}u_{45}u_{345}u_{56}u_7u_8$ & $u_3$ \mbox{$\leadsto u_1u_3u_1u_2u_4u_2u_3u_5u_4u_6u_5u_6u_7u_8$ (\ref{LmBwB})} \\
\midrule
\multirow{2}{*}{$\dt E_7$} & \multirow{2}{*}{$z_{30}=u_1u_{234}u_{345}u_{245}u_6u_7u_8$} & $[8765423413542654786576314532]$ \mbox{$\leadsto u_6u_7u_5u_2u_4u_3u_1$ (\ref{LmBwB})} \\
\midrule
$\dt D_8$ & $z_{36}=u_{13}u_{24}u_{34}u_{45}u_{345}u_{56}u_{67}u_{78}$ & $u_3$ $\leadsto$ 
$u_1u_3u_1u_2u_4u_2u_3u_5u_4u_6u_5u_7u_6u_8u_7u_8$ \\
\midrule
\multirow{2}{*}{$\dt E_7(a_1)$} & \multirow{2}{*}{$z_{42}=u_1u_{234}u_{345}u_{245}u_{2456}u_{67}u_8$} & $[8765432451364257645831764523]$ $\leadsto~u_6u_7u_5u_4u_2u_3u_4u_3u_1$ \\
\midrule
\multirow{2}{*}{$\dt D_7$} & \multirow{2}{*}{$z_{57}=u_1u_{234}u_{345}u_{245}u_{456}u_{567}u_{678}$} & $[134256453176451876245134567813245678$ \\
& & $734621345875678465]$ $\leadsto~u_5u_4u_6u_8u_3u_2u_7$ \\
\midrule
\multirow{2}{*}{$\dt E_7(a_2)$} & \multirow{2}{*}{$z_{59}=u_{1234}u_{1345}u_{245}u_{3456}u_{2456}u_7u_8$} & $u_2\cdot[6756876542786731435465786724315643$ \\
& & $152]\leadsto~u_7u_6u_5u_4u_5u_6u_5u_2u_4u_3u_1$ \\
\midrule
$\dt A_8$ & $z_{61}=u_{134}u_{1234}u_{1345}u_{2345}u_{456}u_{123456}u_{67}u_{78}$ & $z_{61}$ unique with $F|_{A_\mf G(z_{61})}=\id$ (\ref{CGFu0}) \\
\midrule
\multirow{2}{*}{$\dt E_6+\dt A_1$} & \multirow{2}{*}{$z_{65}=u_{1234}u_{1345}u_{234^25}u_{3456}u_{2456}u_7u_8$} & $[786756453413245314324654324786754361$ \\
& & $3524]\leadsto~u_6u_1u_8u_5u_3u_4u_2$ \\
\midrule
\multirow{2}{*}{$(\dt D_7(a_1))_2$} & \multirow{2}{*}{$z_{69}=u_{134}u_{2345}u_{3456}u_{2456}u_{567}u_{4567}u_{78}$} & $[131425431654765348765624315423456785$ \\
& & $6784567467]\leadsto~u_5u_3u_4u_3u_8u_6u_2u_4u_7$ \\
\midrule
\multirow{2}{*}{$\dt D_8(a_3)$} & $z_{77}=u_{1234}u_{1345}u_{2345}u_{3456}u_{2456}u_{4567}u_{78}u_{678}$ & \multirow{2}{*}{$z_{77}$ unique with $F|_{A_\mf G(z_{77})}=\id$ (\ref{CGFu0})}  \\
& only when $q\equiv1\mod3$ (see \ref{p2S3}(b)) &  \\
\midrule
$\dt D_6+\dt A_1$ & $z_{85}=u_{13}u_{234^25}u_{3456}u_{2456}u_{4567}u_{34567}u_{78}$ & $u_2\cdot[4587316542345654134256342765431654$ \\
& & $2543245341342875367]$ \\
& & $\leadsto~u_6u_7u_4u_3u_4u_2u_5u_4u_3u_2u_4u_5u_1$ \\
\midrule
$\dt E_6(a_1)+\dt A_1$ & $z_{90}=u_{1234}u_{1345}u_{234^25}u_{3456}u_{24567}u_8u_{78}$ & see \ref{CGFu0} \\
\midrule
\multirow{2}{*}{$\dt E_6$} & \multirow{2}{*}{$z_{73}=u_{1234}u_{1345}u_{3456}u_{2456}u_7u_8$} & $[765423143876542314534678567432454631]$ \\
& &  $\leadsto$ $u_1u_6u_3u_5u_4u_2$ \\
\midrule
\multirow{2}{*}{$\dt D_6$} & \multirow{2}{*}{$z_{93}=u_{1}u_{234^25}u_{23456}u_{34567}u_{24567}u_{8}$} & $[134254675631453487654312435642543145$ \\
& & $7865467342]$ $\leadsto$ $u_4u_3u_2u_5u_7u_6$ \\
\midrule
$\dt D_7(a_2)$ & $z_{95}=u_{1234}u_{1345}u_{234^25}u_{23456}u_{4567}u_{24567}u_{678}$ & see \ref{CGFu0} \\
\midrule
$\dt E_6(a_1)$ & $z_{98}=u_{1234}u_{1345}u_{3456}u_{24567}u_{8}u_{78}$ & see \ref{CGFu0} \\
\midrule
\multirow{2}{*}{$(\dt D_5+\dt A_2)_2$} & \multirow{2}{*}{$z_{101}=u_{12345}u_{234^25}u_{13456}u_{23456}u_{34567}u_{24567}u_{78}$} & $[678567456341354243154365424563451347$ \\
& & $8657642654634527]\leadsto u_7u_5u_4u_1u_8u_2u_3$ \\
\midrule
\multirow{2}{*}{$\dt D_6(a_1)$} & \multirow{2}{*}{$z_{108}=u_{234^25}u_{23456}u_{123456}u_{24567}u_{134567}u_8$} & $[134287654316543245678657645624342546$ \\
& & $57643453134654562]$ $\leadsto$ $u_3u_4u_3u_2u_4u_7u_5u_6$\\
\midrule
$\dt A_6$ & $z_{113}=u_{1234}u_{1345}u_{23456}u_{34567}u_{24567}u_{45678}$ & see \ref{CGFu0} \\
\midrule
\multirow{3}{*}{$\dt D_5+\dt A_1$} & \multirow{3}{*}{$z_{115}=u_{1234^25}u_{123456}u_{234^25^26}u_{34567}u_{24567}u_{8}$} & $[675645341324531432487654324543241342$ \\
& & $654786543145674531435643143656134245]$ \\
& &$\leadsto u_5u_2u_7u_4u_1u_3$ \\
\midrule
$2\dt A_4$ & $z_{117}=u_{12345}u_{23456}u_{234^256}u_{1234^256}u_{34567}u_{1234^2567}u_{123^24^2567}u_{45678}$ & $z_{117}$ unique with $F|_{A_\mf G(z_{117})}=\id$ (\ref{CGFu0})  \\
\midrule
$\dt A_5+\dt A_2$ & $z_{118}=u_{1234^25}u_{123^24^25}u_{23456}u_{24567}u_{134567}u_{45678}u_{245678}$ & $z_{118}$ unique with $F|_{A_\mf G(z_{118})}=\id$ \\
\midrule
\multirow{3}{*}{$\dt A_5+2\dt A_1$} & \multirow{3}{*}{$z_{122}=u_{1234^25}u_{123^24^25}u_{13456}u_{123456}u_{234^25^26}u_{234567}u_{45678}$} & $u_5\cdot [5786756431543245314523454678567456$ \\
& & $34513524531]$ \\ 
& & $\leadsto~u_6u_5u_3u_1u_8u_4u_2u_4u_3u_4u_2u_5u_4$ \\
\midrule
$\dt D_6(a_2)$ & $z_{125}=u_{1234^25}u_{23456}u_{24567}u_{134567}u_{345678}u_{1345678}$ & see \ref{CGFu0} \\
\midrule
$\dt D_5$ & $z_{128}=u_{34567}u_{24567}u_{1234^25}u_{123456}u_{8}$ & $[654231437654231453487654312435642547$ \\
& & $8675453421]\leadsto u_4u_1u_2u_5u_3$ \\
\midrule
\multirow{2}{*}{$(\dt A_5+\dt A_1)''$} & \multirow{2}{*}{$z_{132}=u_{1234^25}u_{13456}u_{123^24^25}u_{123456}u_{234567}u_{45678}$} & $u_3\cdot[3765431243876542345134678567425431]$ \\
& & $\leadsto u_1u_3u_5u_6u_4u_5u_2u_4u_3u_4u_2u_4$ \\
\midrule
\multirow{2}{*}{$(\dt D_4+\dt A_2)_2$} & \multirow{2}{*}{$z_{135}=u_{123^24^25}u_{1234^256}u_{234^25^26}u_{1234567}u_{234^2567}u_{5678}$} & $[65786713456245342134564237865426534$ \\
& & $7651342]\leadsto u_5u_2u_8u_7u_3u_4$ \\
\midrule
$\dt A_4+2\dt A_1$ & $z_{149}=u_{123^24^256}u_{1234^25^26}u_{1234^2567}u_{234^25^267}u_{245678}u_{1345678}$ & see \ref{CGFu0} \\
\midrule
$\dt D_5(a_1)$ & $z_{144}=u_{123^24^256}u_{1234^25^26}u_{1234^25^267}u_{234^25^26^27}u_{78}$ & see \ref{CGFu0} \\
\midrule
$\dt A_4+\dt A_1$ & $z_{154}=u_{123^24^256}u_{1234^25^267}u_{234^25^26^27}u_{345678}u_{245678}$ & see \ref{CGFu0} \\
\midrule
\multirow{2}{*}{$\dt D_4+\dt A_1$} & \multirow{2}{*}{$z_{147}=u_{123^24^2567}u_{1234^25^267}u_{234^25^26^27}u_{12^23^24^35^26}u_8$} & $[675613452431876542345643786541324565$ \\
& & $3241765342]\leadsto u_5u_2u_3u_7u_4$ \\
\midrule
$\dt D_4(a_1)+\dt A_2$ & $z_{161}=u_{123^24^256}u_{1234^2567}u_{1234^25^267}u_{234^25^26^27}u_{12345678}u_{234^25^2678}$ & see \ref{CGFu0} \\
\midrule
$\dt A_4$ & $z_{158}=u_{123^24^256}u_{1234^25^267}u_{345678}u_{245678}$ & see \ref{CGFu0} \\
\midrule
$\dt D_4(a_1)+\dt A_1$ & $z_{168}=u_{12^23^24^35^26}u_{123^24^2567}u_{1234^25^267}u_{1234^25^2678}u_{234^25^26^278}$ & $z_{168}$ unique with $F|_{A_\mf G(z_{168})}=\id$ \\
\midrule
\multirow{2}{*}{$\dt D_4$} & \multirow{2}{*}{$z_{152}=u_{123^24^2567}u_{1234^25^267}u_{234^25^26^27}u_8$} & $[651342768754316543245678436542134567]$ \\
& & $\leadsto u_3u_2u_5u_4$ \\
\midrule
$\dt D_4(a_1)$ & $z_{173}=u_{123^24^2567}u_{1234^25^267}u_{1234^25^2678}u_{234^25^26^278}$ & $z_{173}$ unique with $F|_{A_\mf G(z_{173})}=\id$ \\
\midrule
$2\dt A_2$ & $z_{181}=u_{12^23^24^35^267}u_{123^24^35^26^27}u_{123^24^25^2678}u_{1234^25^26^278}$ & see \ref{CGFu0} \\
\midrule
$\dt A_2+\dt A_1$ & $z_{186}=u_{1^22^23^34^45^36^27}u_{123^24^35^36^278}u_{12^23^24^35^26^27^28}$ & see \ref{CGFu0} \\
\midrule
$\dt A_2$ & $z_{190}=u_{123^24^35^36^278}u_{12^23^24^35^26^27^28}$ & see \ref{CGFu0} \\
\bottomrule
\end{longtable}
\end{landscape}

\section{Groups of type \texorpdfstring{$\leftidx{^2}{\dt E_6}$}{2E6}}\label{SectwE6}

From now until the end of the paper, we denote by $\mf G$ the simple adjoint group of type $\dt E_6$ over $k=\overline{\mathbb F}_2$ and by $F\colon\mf G\rightarrow\mf G$ a Frobenius map for a twisted $\Fq$-rational structure on $\mf G$ where $q$ is a power of $p=2$, so that $\mf G^F=\leftidx{^2}{\dt E_6}(q)$. The further notation is analogous to the one in the introduction: In particular, we fix an $F$-stable Borel subgroup $\mf B_0$ of $\mf G$ and an $F$-stable maximal torus $\mf T_0$ of $\mf G$ such that $\mf T_0\subseteq\mf B_0$, and then $\mf W=N_\mf G(\mf T_0)/\mf T_0$, $\Phi\supseteq\Phi^+\supseteq\Pi$, $S\subseteq\mf W$ and $\mf U_0$ are assumed to be with respect to the choice of $\mf T_0\subseteq\mf B_0$. Let $\sigma\colon\mf W\xrightarrow{\sim}\mf W$ be the automorphism induced by $F$. Then $\sigma$ is naturally induced from an automorphism of the root system $\Phi$ which stabilises $\Pi$. We have $\sigma(w)=w_0ww_0^{-1}$ for $w\in\mf W$, where $w_0\in\mf W$ is the longest element of $(\mf W,S)$. We label the simple roots $\Pi=\{\alpha_1,\alpha_2,\ldots,\alpha_6\}$ in such a way that the Dynkin diagram of $\mf G$ (together with the permutation of $\Pi$ induced by $F$, as indicated by the arrows) is given as follows:
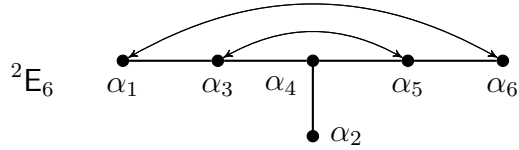
\begin{center}
\begin{tikzpicture}
    \draw (-0.75,-0.25) node[anchor=east]  {$\leftidx{^2}{\dt E}_6$};

    \node[bnode,label=below:$\alpha_1$] 		(1) at (0,0) 	{};
    \node[bnode,label=right:$\alpha_2$] 		(2) at (2.5,-1) 	{};
    \node[bnode,label=below:$\alpha_3$] 		(3) at (1.25,0) 	{};
    \node[bnode,label=below left:$\alpha_4$] 	(4) at (2.5,0) 	{};
    \node[bnode,label=below:$\alpha_5$] 		(5) at (3.75,0) 	{};
    \node[bnode,label=below:$\alpha_6$] 		(6) at (5,0) 	{};

    \path 	(1) edge[thick, sedge] (3)
          	(3) edge[thick, sedge] (4)
          	(4) edge[thick, sedge] (5)
			(4)	edge[thick, sedge] (2)
          	(5) edge[thick, sedge] (6)
          	(1) edge[->,>=stealth',bend left] node [right] {} (6)
            (6) edge[->,>=stealth',bend right] node [left] {} (1)
          	(3) edge[->,>=stealth',bend left] node [right] {} (5)
            (5) edge[->,>=stealth',bend right] node [left] {} (3);
          ;
\end{tikzpicture}
\end{center}
To compute the values of unipotent characters at unipotent elements in $\mf G^F$, we would like to argue similarly as for groups of type $\dt E_8$ (cf.\ also \cite[4.1.23--4.1.27]{HDiss}, where the non-twisted groups $\dt E_6(q)$ have been considered), but the twisted $\Fq$-structure on $\mf G$ forces us to modify several arguments.

\begin{0}\label{2E6Setup}
In order to describe the parametrisation of the unipotent (almost) characters according to Lusztig \cite{Luchars}, we need to slightly expand the setting of \Cref{E8Setup}; we use a similar notation as in \cite[4.1.2, 4.1.5]{HDiss}. The parametrisation of the unipotent characters via \cite[Main Theorem 4.23]{Luchars} is now given in terms of a set $\overline X(\mf W,\sigma)$, which may be identified with the parameter set $X(\mf W)$ for the unipotent characters of the non-twisted group $\dt E_6(q)$ (cf.\ \Cref{E8Setup}); let us denote this identification by
\[X(\mf W)\xrightarrow{\sim}\overline X(\mf W,\sigma),\fsep x\mapsto\overline x.\]
The group $\mf G^F$ has $30$ unipotent characters (see \cite[1.10--1.16]{LuUniexc}). For $\rho\in\Uch(\mf G^F)$, let $x_\rho\in X(\mf W)$ be such that $\overline x_\rho\in\overline X(\mf W,\sigma)$ parametrises $\rho$. As for the unipotent almost characters, we need to work with a larger set than $X(\mf W)$, which can be identified with $X(\mf W)\times\{\pm1\}$ (see \cite[4.19]{Luchars}). The unipotent characters and unipotent almost characters are then related by a pairing $\overline X(\mf W,\sigma)\times(X(\mf W)\times\{\pm1\})\rightarrow\Qlbar$, given by
\[\{\overline x,(y,\delta)\}=\delta\{x,y\}\quad\text{for }x,y\in X(\mf W),\,\delta\in\{\pm1\},\]
where $\{x,y\}$ on the right hand side is described by the Fourier matrix with respect to $X(\mf W)$ that we already referred to in \Cref{E8Setup}. For $z\in X(\mf W)\times\{\pm1\}$, the unipotent almost character $R_z$ is defined by
\[R_z=\sum_{\rho\in\Uch(\mf G^F)}\Delta(\overline x_\rho)\{\overline x_\rho,z\}\rho,\]
where $\Delta\colon\overline X(\mf W,\sigma)\rightarrow\{\pm1\}$ is as in \cite[4.21]{Luchars}. Taking any fixed subset $X_0(\mf W,\sigma)\subseteq X(\mf W)\times\{\pm1\}$ of representatives of the fibres of the projection map $X(\mf W)\times\{\pm1\}\rightarrow X(\mf W)$, a unipotent character $\rho\in\Uch(\mf G^F)$ is then expressed as
\[\rho=\Delta(\overline x_\rho)\sum_{z\in X_0(\mf W,\sigma)}\{\overline x_\rho,z\}R_z.\]
In particular, the computation of the unipotent characters (at unipotent elements) is equivalent to the one of the almost characters $R_z$ (at unipotent elements) for $z\in X_0(\mf W,\sigma)$.
\end{0}

\begin{0}\label{Green2E6}
We would like to follow a strategy similar to the one described in \ref{Method}.
Let $\mf W(\sigma):=\mf W\rtimes\langle\sigma\rangle$ be the semidirect product of $\mf W$ with the cyclic group of order $2$ generated by $\sigma$, where $\sigma\cdot w\cdot\sigma^{-1}=\sigma(w)$ in $\mf W(\sigma)$ (for $w\in\mf W$). Since $\sigma$ is an inner automorphism of $\mf W$, it leaves any character of $\mf W$ invariant, and therefore any $\phi\in\Irr(\mf W)$ can be extended to an irreducible character of $\mf W(\sigma)$, in two different ways. In \cite[\S 17.2]{LuCS4}, Lusztig singles out a \enquote{preferred extension} of $\phi\in\Irr(\mf W)$, which we shall from now on always denote by $\tilde\phi\colon\mf W(\sigma)\rightarrow\Qlbar$. It is given by
\[\tilde\phi\colon\mf W(\sigma)\rightarrow\Qlbar,\fsep w\sigma^i\mapsto{(-1)}^{i\cdot a_\phi}\phi(ww_0^i)\quad\text{(for }w\in\mf W,\;i=0,1),\]
where $a_\phi\in\N_0$ is as defined in \cite[4.1]{Luchars}. Let $\Irr(\mf W(\sigma))_{\mathrm{ex}}$ be the set of all irreducible characters of $\mf W(\sigma)$ whose restrictions to $\mf W$ are in $\Irr(\mf W)$. As discussed above, we have $\Irr(\mf W(\sigma))_{\mathrm{ex}}=\{\pm\tilde\phi\mid\phi\in\Irr(\mf W)\}$. Recall from \ref{E8Setup} that there is an embedding $\Irr(\mf W)\hookrightarrow X(\mf W)$, $\phi\mapsto x_\phi$. This allows the definition of an embedding
\[\Irr(\mf W(\sigma))_{\mathrm{ex}}\hookrightarrow X(\mf W)\times\{\pm1\},\fsep\delta\tilde\phi\mapsto(x_\phi,\delta)\quad\text{(for }\phi\in\Irr(\mf W),\;\delta\in\{\pm1\}),\]
see \cite[4.19]{Luchars}. The purpose of this procedure is that, given $\phi\in\Irr(\mf W)$, there is a more direct definition of an almost character $R_{\tilde\phi}$ (see \cite[(3.7.1)]{Luchars}): In our situation, this definition reads
\[R_{\tilde\phi}=\frac{1}{|\mf W|}\sum_{w\in\mf W}\tilde\phi(w\sigma)R_w,\]
where $R_w=R_{\mf T_w}^{\mf G}(1)$ is the Deligne--Lusztig character \cite{DL} associated to the maximal torus $\mf T_w\subseteq\mf G$ obtained from $\mf T_0$ by twisting with $w\in\mf W$ (and to the trivial character of $\mf T_w^F$). By \cite[4.23]{Luchars} and with the notation of \ref{2E6Setup}, we then have $R_{\tilde\phi}=R_{(x_\phi,1)}$. The computation of the values at unipotent elements of these $25$ almost characters $R_{\tilde\phi}$ for $\phi\in\Irr(\mf W)$ is equivalent to the determination of the (ordinary) Green functions and has been carried out already by Malle (see \cite[Thm.~2]{MGreenE6F4}). It remains to deal with the $5$ almost characters which do not arise from an element of $\Irr(\mf W(\sigma))_{\mathrm{ex}}$. Now note that the isomorphism classes of unipotent character sheaves on $\mf G$ are parametrised by $X(\mf W)$ (cf.\ \ref{NG}), and it is still true that a characteristic function of any unipotent character sheaf $A_x$ ($x\in X(\mf W)$) coincides with the associated almost characters $R_{(x,\pm1)}$ up to multiplication with a non-zero scalar (see \cite[Thm.~4.1]{Sh2}). This includes the two cuspidal character sheaves, whose support is given by a non-unipotent conjugacy class of $\mf G$ (see the proof of \cite[Prop.~20.3(a)]{LuCS4}), so we already know that the corresponding two almost characters will be identically $0$ on $\mf G^F_{\mathrm{uni}}$. Hence, there only remain $3$ unipotent almost characters whose values at unipotent elements of $\mf G^F$ are not yet known.
\end{0}

\begin{0}\label{IntermedTwist}
Note that the algorithm in \cite[\S 24]{LuCS5} is also applicable for twisted groups (and the generalised Springer correspondence does not even depend on the $\Fq$-structure on $\mf G$). Thus, the general part of the discussion in \ref{Intermed} transfers almost verbatim to our situation at hand, with one difference: If $x\in X(\mf W)$ corresponds to $\ii\in\cN_\mf G$ and if, for $z=(x,\delta)\in X(\mf W)\times\{\pm1\}$, $\xi_z\in\Qlbar$ is such that $R_z=\xi_z\chi_x^0$, we can no longer say that $\xi_z$ only depends on $\tau(\ii)\in\cM_\mf G$ (as $\chi_x^0$ only depends on $x$, while $R_{(x,-1)}=-R_{(x,1)}$). However, for any $z=(x,\delta)$ such that the $\ii\in\cN_\mf G$ corresponding to $x$ gives rise to a fixed element $\jj=\tau(\ii)\in\cM_\mf G$, there exists a sign $\epsilon_z\in\{\pm1\}$ such that $\zeta_\jj:=\epsilon_z\xi_z$ only depends on $\jj$ (see \cite[Proposition 4.4]{Sh2}), and we have $\epsilon_{(x,-1)}=-\epsilon_{(x,1)}$. From the results of \cite{Sp}, it follows that the $3$ \enquote{intermediate} $x\in X(\mf W)$ which we need to consider correspond to the pairs $(\dt E_6,-1)$, $(\dt D_5,-1)$, $(\dt D_4,-1)$ of $\cN_\mf G$ under the generalised Springer correspondence, and for any such pair $\ii$ we have $\tau(\ii)=\jj=(\mf L,\cO_0,\cE_0)\in\cM_\mf G$ where $(\cO_0,\cE_0)$ is the unique cuspidal pair on a Levi subgroup $\mf L\subseteq\mf G$ of type $\dt D_4$. (We will henceforth usually just write $(\dt D_4,\cO_0,\cE_0):=(\mf L,\cO_0,\cE_0)$.) Hence, if $z=(x,\delta)$ where $x\in X(\mf W)$ corresponds to one of the three pairs $\ii=(\cO,\cE)\in\{(\dt E_6,-1),(\dt D_5,-1),(\dt D_4,-1)\}$ above and using a notation analogous to the one in \ref{Intermed}, we have
\begin{align*}
R_z(u)&=\zeta_\jj\epsilon_z\chi_x^0(u)=\zeta_\jj\epsilon_zq^{(\dim\mf G-\dim\cO-\dim\mf Z(\mf L))/2}X_\ii(u) \\
&=\zeta_\jj\epsilon_zq^{(\dim\mf G-\dim\cO-\dim\mf Z(\mf L))/2}\sum_{\ii'\in\tau^{-1}(\jj)}p_{\ii',\ii}\gamma_{\ii'}Y_{\ii'}^0(u)
\end{align*}
for any $u\in\mf G^F_{\mathrm{uni}}$, so the determination of $\zeta_{\jj}\epsilon_z\gamma_{\ii'}$ for the three elements $\ii'\in\tau^{-1}(\jj)$ would yield the values of $R_z|_{\mf G^F_{\mathrm{uni}}}$.
\end{0}

\begin{0}\label{ReeTwist}
A formula such as \eqref{HeckeFormula} in \ref{Ree} also exists for the present situation, but we need to take the twisted $\Fq$-rational structure on $\mf G$ into account. Namely, the group $\mf W^\sigma$ is a Weyl group associated to a root system of type $\dt F_4$ which is naturally induced from the one of type $\dt E_6$. This may be pictured as follows (see \cite[Ex.~4.3]{GHF4conv} or \cite[4.1.5]{HDiss}):
\begin{center}
\begin{tikzpicture}
    \draw (-0.75,-0.25) node[anchor=east]  {$\leftidx{^2}{\dt E}_6$};

    \node[bnode,label=below:$\alpha_1$] 		(1) at (0,0) 	{};
    \node[bnode,label=right:$\alpha_2$] 		(2) at (2.5,-1) 	{};
    \node[bnode,label=below:$\alpha_3$] 		(3) at (1.25,0) 	{};
    \node[bnode,label=below left:$\alpha_4$] 	(4) at (2.5,0) 	{};
    \node[bnode,label=below:$\alpha_5$] 		(5) at (3.75,0) 	{};
    \node[bnode,label=below:$\alpha_6$] 		(6) at (5,0) 	{};

    \path 	(1) edge[thick, sedge] (3)
          	(3) edge[thick, sedge] (4)
          	(4) edge[thick, sedge] (5)
			(4)	edge[thick, sedge] (2)
          	(5) edge[thick, sedge] (6)
          	(1) edge[->,>=stealth',bend left] node [right] {} (6)
            (6) edge[->,>=stealth',bend right] node [left] {} (1)
          	(3) edge[->,>=stealth',bend left] node [right] {} (5)
            (5) edge[->,>=stealth',bend right] node [left] {} (3);
          	
    \draw (6.5,-0.25) node[anchor=east]  {$\leadsto$};
    
    \draw (7.5,-0.25) node[anchor=east]  {$\dt F_4$};

    \node[bnode,label=above:$\alpha_2$,label=below:$\overline s_2$]	(9) at (8,-0.25) {};
    \node[bnode,label=above:$\alpha_4$,label=below:$\overline s_4$] (10) at (9.25,-0.25) {};
    \node[bnode,label=above:$\frac{\alpha_3+\alpha_5}2$,label=below:$\overline s_{35}$] (11) at (10.5,-0.25) {};
    \node[bnode,label=above:$\frac{\alpha_1+\alpha_6}2$,label=below:$\overline s_{16}$] (12) at (11.75,-0.25) {};

    \path (9) edge[sedge] (10)
          (10) edge[dedge] (11)
          (11) edge[sedge] (12)
          ;
\end{tikzpicture}
\end{center}
Here, the labels $\overline s_2, \overline s_4, \overline s_{35}, \overline s_{16}$ on the bottom of the diagram on the right denote the reflections of $\mf W^\sigma$ corresponding to the simple roots $\alpha_2, \alpha_4, \frac{\alpha_3+\alpha_5}2, \frac{\alpha_1+\alpha_6}2$, respectively. We have $\overline s_2=s_2$, $\overline s_4=s_4$, $\overline s_{35}=s_3s_5$, $\overline s_{16}=s_1s_6$, and the set $S_\sigma:=\{\overline s_2, \overline s_4, \overline s_{35}, \overline s_{16}\}\subseteq\mf W^\sigma$ generates $\mf W^\sigma$ as a Coxeter group. Applying the procedure outlined in \Cref{Ree} to $(\mf W^\sigma,S_\sigma)$, we obtain an Iwahori--Hecke algebra $\cH_{\sigma,q}$ with a standard basis $T_w$, $w\in\mf W^\sigma$, whose multiplication is determined by the following equations, where $w\in\mf W^\sigma$, $s\in S_\sigma$, and $\length_\sigma\colon\mf W^\sigma\rightarrow\N_0$ is the length function of $\mf W^\sigma$ with respect to $S_\sigma$:
\[T_w\cdot T_{s}=\begin{cases}\hfil T_{ws}&\text{if}\quad\length_\sigma(ws)=\length_\sigma(w)+1, \\ qT_{ws}+(q-1)T_{w}&\text{if}\quad\length_\sigma(ws)=\length_\sigma(w)-1\text{ and }s\in\{\overline s_2,\overline s_4\}, \\
q^2T_{ws}+(q^2-1)T_{w}&\text{if}\quad\length_\sigma(ws)=\length_\sigma(w)-1\text{ and }s\in\{\overline s_{35},\overline s_{16}\}.
\end{cases}\]
Each $\psi\in\Irr(\mf W^\sigma)$ parametrises an irreducible character $\psi_q$ of $\cH_{\sigma,q}$. Then, denoting by $O_u\subseteq\mf G^F_{\mathrm{uni}}$ the $\mf G^F$-conjugacy class of $u\in\mf G^F_{\mathrm{uni}}$, we have
\begin{equation*}
\frac{|\mf B_0^Fw\mf B_0^F\cap O_u|\cdot|C_{\mf G^F}(u)|}{|\mf B_0^F|}=\sum_{\psi\in \Irr(\mf W^\sigma)} 
\psi_q(T_w)\cdot[\psi](u)\qquad \text{for } u\in\mf G^F_{\mathrm{uni}},\,w\in\mf W^\sigma.
\end{equation*}
(Here, $[\psi]\in\Uch(\mf G^F)$ is the principal series unipotent character parametrised by $\psi\in\Irr(\mf W^\sigma)$.) Let $X_0(\mf W,\sigma)\subseteq X(\mf W)\times\{\pm1\}$ be as in \ref{2E6Setup} and, for $u\in\mf G^F_{\mathrm{uni}}$, $w\in\mf W^\sigma$, let
\[m(u,w):=\sum_{z\in X_0(\mf W,\sigma)}c_z(w)R_z(u),\text{ with}\quad c_z(w):=\sum_{\psi\in\Irr(\mf W^\sigma)}\Delta(\overline x_{[\psi]})\{\overline x_{[\psi]},z\}\psi_q(T_w).\]
This is a generalisation of the definition of the $m(u,w)$ in \ref{Ree} in order to have
\begin{equation}\label{HeckeFormula2E6}
m(u,w)=\sum_{z\in X_0(\mf W,\sigma)}c_z(w)R_z(u)=\frac{|\mf B_0^Fw\mf B_0^F\cap O_u|\cdot|C_{\mf G^F}(u)|}{|\mf B_0^F|}\;(u\in\mf G^F_{\mathrm{uni}},\;w\in\mf W^\sigma).
\tag{${\spadesuit}$}
\end{equation}
(In particular, note that $m(u,w)$ does not depend on the choice of $X_0(\mf W,\sigma)$.) Again, the values $\psi_q(T_w)$ can be accessed through {\sffamily CHEVIE} \cite{MiChv}.
\end{0}

The following lemma is the first step towards determining the roots of unity $\zeta_{\jj}\epsilon_z\gamma_{\ii'}$ which appear in \ref{IntermedTwist}.

\begin{Lm}\label{epsTwist}
In the set-up and with the notation of \ref{IntermedTwist}, let us write $\epsilon_\ii:=\epsilon_{(x,1)}$ if $\ii\in\cN_\mf G$ corresponds to $x\in X(\mf W)$. We then have 
\[\epsilon_{(\dt D_4,-1)}=\epsilon_{(\dt D_5,-1)}=-\epsilon_{(\dt E_6,-1)}\in\{\pm1\}.\]
\end{Lm}

\begin{proof}
Let us take $X_0(\mf W,\sigma):=X(\mf W)\times\{1\}$ for the set of representatives of the fibres of the projection map $X(\mf W)\times\{\pm1\}\rightarrow X(\mf W)$ (see \ref{2E6Setup}). As discussed in \ref{IntermedTwist}, we have $\jj:=\tau(\ii)=(\dt D_4,\cO_0,\cE_0)$ for any $\ii\in\{(\dt E_6,-1),(\dt D_5,-1),(\dt D_4,-1)\}$, so the roots of unity $\zeta_\jj\epsilon_z\gamma_{\ii'}$ for the corresponding three $z=(x,1)\in X_0(\mf W,\sigma)$ (and any fixed $\ii'\in\cN_\mf G$) differ only by the signs $\epsilon_z$. Assume, if possible, that $\epsilon_{(\dt D_5,-1)}=\epsilon_{(\dt E_6,-1)}$. Then, taking $w$ to be a Coxeter element of $(\mf W^\sigma,S_\sigma)$ and evaluating the above three $R_z$ at elements of $\dt D_5^F$, we compute the sum in the formula \eqref{HeckeFormula2E6} in \ref{ReeTwist} and obtain
\[0\leqslant m(u,w)=-2q^7\zeta_\jj\epsilon_{(\dt E_6,-1)}\gamma_{(\dt D_5,-1)}Y_{(\dt D_5,-1)}^0(u)\quad\text{for any }u\in\dt D_5^F.\]
However, by the definition of $Y_{(\dt D_5,-1)}^0$ (and since $A_\mf G(u)\cong\Z/2\Z$ for $u\in\dt D_5^F$, with $F$ acting trivially on $A_\mf G(u)$), we have $Y_{(\dt D_5,-1)}^0(u)=-Y_{(\dt D_5,-1)}^0(u')\neq0$ if $u,u'\in\dt D_5^F$ are not $\mf G^F$-conjugate, contradicting the above inequality for any $u\in\dt D_5^F$. So we must have $\epsilon_{(\dt D_5,-1)}=-\epsilon_{(\dt E_6,-1)}$. Similarly, if we assume that $\epsilon_{(\dt D_4,-1)}=\epsilon_{(\dt E_6,-1)}$, we can take $w:=\overline s_2\cdot\overline s_4\cdot\overline s_{35}\cdot\overline s_{16}\cdot\overline s_4\cdot\overline s_{35}\in\mf W^\sigma$ and get
\[0\leqslant m(u,w)=2q^{11}\zeta_\jj\epsilon_{(\dt E_6,-1)}\gamma_{(\dt D_4,-1)}Y_{(\dt D_4,-1)}^0(u)\quad\text{for any }u\in\dt D_4^F,\]
which contradicts the fact that $Y_{(\dt D_4,-1)}^0(u)=-Y_{(\dt D_4,-1)}^0(u')\neq0$ whenever $u,u'\in\dt D_4^F$ are not $\mf G^F$-conjugate. This proves that $\epsilon_{(\dt D_4,-1)}=-\epsilon_{(\dt E_6,-1)}$.
\end{proof}


Thanks to \Cref{epsTwist}, we know the products $\zeta_\jj\epsilon_\ii$ for $\jj=(\dt D_4,\cO_0,\cE_0)$ and the three $\ii\in\tau^{-1}(\jj)\subseteq\cN_\mf G$ up to a global sign. In order to obtain the necessary information on the roots of unity $\gamma_{\ii'}$ for $\ii'\in\tau^{-1}(\jj)$ (see \ref{IntermedTwist}), we will once again evaluate $m(u,w)$, with $u$ a unipotent element in one of the three \enquote{critical} unipotent classes $\dt E_6, \dt D_5, \dt D_4$, and with an appropriately chosen $w\in\mf W^\sigma$ depending on $u$.

As in \Cref{Repgood}, we want to be able to single out good elements (cf.\ \Cref{Defgood}) among a concrete list of representatives for the unipotent conjugacy classes of $\mf G^F$ (as given for example in \cite[Table~9]{MGreenE6F4}). To achieve this, we shall need the following lemma, which (just as \Cref{LmBwB}) is a special case of \cite[Lemma~3.2.24]{HDiss}. Note that $w_0(\Pi)=-\Pi$, so $-w_0$ defines a permutation on the simple roots $\Pi$, which corresponds to the non-trivial graph automorphism of the Dynkin diagram of $(\mf G,F)$ as pictured in the introduction to this section.

\begin{Lm}[see {\cite[Lemma~3.2.24]{HDiss}}; cf.\ \Cref{LmBwB}]\label{LmBwBTwist}
Let $w\in\mf W^\sigma\subseteq\mf W$ be an element of length $e\in\N_0$ in $(\mf W,S)$, and let $w=s_{i_1}\cdot s_{i_2}\cdot\ldots\cdot s_{i_e}$ be a reduced expression for $w$ (where $s_{i_j}\in S$ is the reflection with respect to the simple root $\alpha_{i_j}$, $1\leqslant i_j\leqslant6$). Then
\[u_0:=u_{-w_0(\alpha_{i_1})}(1)\cdot u_{-w_0(\alpha_{i_2})}(1)\cdot\ldots\cdot u_{-w_0(\alpha_{i_e})}(1)\in\mf U_0^F\]
is $\mf G^F$-conjugate to an element of $\mf B_0^Fw\mf B_0^F$.
\end{Lm}

\begin{proof}
See \cite[3.2.24]{HDiss}.
\end{proof}

\begin{0}{\bf The regular unipotent class $\dt E_6$.}
We have $A_\mf G(u)\cong\Z/2\Z$ for any regular unipotent element $u\in\dt E_6^F$, with $F$ acting trivially on $A_\mf G(u)$. Following Malle \cite[Table~9]{MGreenE6F4}, we choose the representative
\[u_0:=u_6(1)u_1(1)u_2(1)u_5(1)u_3(1)u_4(1)\in\dt E_6^F.\]
(This element is denoted by \enquote{$u_{26}$} in \cite{MGreenE6F4}.) Setting $w:=s_1s_6s_2s_3s_5s_4=\overline s_{16}\overline s_2\overline s_{35}\overline s_4\in\mf W^\sigma$ (a Coxeter element of both $(\mf W,S)$ and $(\mf W^\sigma,S_\sigma)$), $u_0$ is $\mf G^F$-conjugate to an element of $\mf B_0^Fw\mf B_0^F$ by \Cref{LmBwBTwist}, so $m(u_0,w)>0$. We now evaluate the sum on the left hand side of the formula \eqref{HeckeFormula2E6} in \ref{ReeTwist} and obtain
\[m(u_0,w)=q^6(1-\zeta_\jj\epsilon_{(\dt E_6,-1)}\gamma_{(\dt E_6,-1)}).\]
As $m(u_0,w)$ is a real number, this first of all forces the root of unity $\zeta_\jj\epsilon_{(\dt E_6,-1)}\gamma_{(\dt E_6,-1)}$ to be either $+1$ or $-1$. But it cannot be $+1$ since $m(u_0,w)>0$, so we must have $\zeta_\jj\epsilon_{(\dt E_6,-1)}\gamma_{(\dt E_6,-1)}=-1$.

\end{0}

\begin{0}{\bf The class $\dt D_5$.}\label{D5twist}
We have $A_\mf G(u)\cong\Z/2\Z$ for $u\in\dt D_5^F$, with $F$ acting trivially on $A_\mf G(u)$. Following \cite[Table~9]{MGreenE6F4}, we pick the representative \[u_0:=u_2(1)u_4(1)u_{\alpha_1+\alpha_3}(1)u_{\alpha_5+\alpha_6}(1)u_{\alpha_3+\alpha_4+\alpha_5}(1)\in\dt D_5^F.\]
(This element is denoted by \enquote{$u_{23}$} in \cite{MGreenE6F4}.) Applying the relations given in \ref{Rel} for groups of type $\dt E_6$ and using the analogous notation (note that $p=2$ here as well), we see that
\[(\omega_3\omega_5u_3u_5)u_0(\omega_3\omega_5u_3u_5)^{-1}=u_2u_{34}u_3u_{45}u_3u_1u_6=u_2u_4u_3u_5u_4u_3u_5u_1u_6.\]
We have $F(u_3)=u_5$, $F(u_5)=u_3$, $F(\omega_3)=\omega_5$, $F(\omega_5)=\omega_3$ and $u_3u_5=u_5u_3$, $\omega_3\omega_5=\omega_5\omega_3$, so $\omega_3\omega_5u_3u_5\in\mf G^F$. Hence, setting $w:=s_2s_4s_5s_3s_4s_5s_3s_6s_1=\overline s_2\overline s_4\overline s_{35}\overline s_4\overline s_{35}\overline s_{16}\in\mf W^\sigma$, \Cref{LmBwBTwist} shows that $u_0$ is $\mf G^F$-conjugate to an element of $\mf B_0^Fw\mf B_0^F$, so $m(u_0,w)>0$. Evaluating the sum on the left hand side of the formula \eqref{HeckeFormula2E6} in \ref{ReeTwist}, we obtain
\[0<m(u_0,w)=q^9(1+\zeta_\jj\epsilon_{(\dt D_5,-1)}\gamma_{(\dt D_5,-1)}),\]
and we conclude that the root of unity $\zeta_\jj\epsilon_{(\dt D_5,-1)}\gamma_{(\dt D_5,-1)}$ must be equal to $+1$.
\end{0}

\begin{0}{\bf The class $\dt D_4$.}
We have $A_\mf G(u)\cong\Z/2\Z$ for $u\in\dt D_4^F$, with $F$ acting trivially on $A_\mf G(u)$. According to Malle \cite[Table~9]{MGreenE6F4}, the element
\[u_2(1)u_{\alpha_4+\alpha_5+\alpha_6}(1)u_{\alpha_1+\alpha_3+\alpha_4}(1)u_{\alpha_3+\alpha_4+\alpha_5}(1)\]
(which is denoted by \enquote{$u_{16}$} in \cite{MGreenE6F4}) lies in $\dt D_4^F$. A similar computation as in \ref{D5twist} shows that conjugating Malle's \enquote{$u_{16}$} by $\omega_6\omega_1\omega_5\omega_3\omega_4\omega_2\omega_5\omega_3\omega_4\in\mf G^F$ yields
\[u_0:=u_4(1)u_3(1)u_5(1)u_2(1)\in\dt D_4^F.\]
By \Cref{LmBwBTwist}, $u_0$ is $\mf G^F$-conjugate to an element of $\mf B_0^Fw\mf B_0^F$ where $w=s_4s_5s_3s_2=\overline s_4\overline s_{35}\overline s_2\in\mf W^\sigma$. Evaluating formula \eqref{HeckeFormula2E6} in \ref{ReeTwist} then gives
\[0<m(u_0,w)=(q^{10}+q^7)(1-\zeta_\jj\epsilon_{(\dt D_4,-1)}\gamma_{(\dt D_4,-1)}),\]
which forces the root of unity $\zeta_\jj\epsilon_{(\dt D_4,-1)}\gamma_{(\dt D_4,-1)}$ to be equal to $-1$.
\end{0}

\begin{0}
Let $\jj=(\dt D_4,\cO_0,\cE_0)\in\cM_\mf G$ and, for $\ii\in\tau^{-1}(\jj)\subseteq\cN_\mf G$ such that $\ii\leftrightarrow x\in X(\mf W)$ under the generalised Springer correspondence, let us write $\epsilon_\ii:=\epsilon_{(x,1)}$ and $R_\ii:=R_{(x,1)}$. Now that we have determined the numbers $\zeta_\jj\epsilon_\ii\gamma_{\ii'}$ for $\ii,\ii'\in\tau^{-1}(\jj)$, we can explicitly compute the values of the three almost characters $R_\ii$ ($\ii\in\tau^{-1}(\jj)$) at unipotent elements in view of the discussion in \ref{IntermedTwist}. The functions $R_\ii|_{\mf G^F_{\mathrm{uni}}}$ are identically zero outside of the classes $\dt D_4,\dt D_5,\dt E_6$, and their values on these classes are given in \Cref{Tbl2E6p2Rx}. As discussed in \ref{Green2E6}, combined with Malle's results \cite[Thm.~2]{MGreenE6F4}, this completes the determination of the values of unipotent characters at unipotent elements for the groups $\leftidx{^2}{\dt E_6}(q)$ where $q$ is a power of $p=2$.
\begin{table}[htbp]
\centering
\begin{tabular}{LCCCCCC}
\toprule
\mf G\text{-conjugacy class:} & \multicolumn{2}{C}{\dt D_4} & \multicolumn{2}{C}{\dt D_5} & \multicolumn{2}{C}{\dt E_6} \\
\cmidrule[0.08em](r){1-1} \cmidrule[0.08em](rl){2-3} \cmidrule[0.08em](rl){4-5} \cmidrule[0.08em](l){6-7}
\mf G^F\text{-class (Malle \cite{MGreenE6F4}):}& u_{16} & u_{17} & u_{23} & u_{24} & u_{26} & u_{27} \\
\midrule
R_{(\dt E_6,-1)} & q^5 & -q^5 & -q^3 & q^3 & -q^2 & q^2  \\
\addlinespace[0.25em]
R_{(\dt D_5,-1)} & q^7-q^6 &  -q^7+q^6 & q^4 & -q^4 & 0 & 0  \\
\addlinespace[0.25em]
R_{(\dt D_4,-1)} & -q^8 & q^8 & 0 & 0 & 0 & 0 \\
\bottomrule
\end{tabular}
\caption{The values of $R_\ii|_{\mf G^F_{\mathrm{uni}}}$ for $\ii\in\tau^{-1}((\dt D_4,\cO_0,\cE_0))$, where $\mf G^F=\leftidx{^2}{\dt E_6}(q)$ with $q$ a power of $p=2$}
\label{Tbl2E6p2Rx}
\end{table}
\end{0}

\begin{Ack}
I thank Meinolf Geck and Frank Lübeck for useful discussions and comments. I also thank Gunter Malle for detailed comments on an earlier version. This work was supported by the Deutsche Forschungsgemeinschaft (DFG, German Research Foundation) --- Project-ID 286237555 -- TRR 195.
\end{Ack}

\bibliographystyle{abbrv}

\bibliography{E82E6valuni}

\end{document}